\documentclass{amsart}
\usepackage[T1]{fontenc}
\usepackage[utf8]{inputenc}
\usepackage[colorlinks=false, linktocpage=true]{hyperref}
\usepackage[all]{xy}
\usepackage{color}
\usepackage{enumerate}

\usepackage{longtable}

\usepackage{amsthm}
\usepackage{amssymb,amscd}
\usepackage[normalem]{ulem}
\usepackage{xfrac}
\usepackage{tikz-cd}
\usepackage{fullpage}

\usepackage{booktabs}
\usepackage{multirow}

\usepackage{mathrsfs}

\newcounter{notes}%

\title{Hitchin components for orbifolds}

\author{Daniele Alessandrini}
\address{Department of Mathematics, Columbia University, 2990 Broadway, New York, NY, 10027, USA.}
\email{daniele.alessandrini@gmail.com}

\author{Gye-Seon Lee}
\address{Department of Mathematics, Sungkyunkwan University, 2066 Seobu-ro, Jangan-gu, Suwon, Gyeonggi-do, 16419, South Korea.}
\email{gyeseonlee@skku.edu}

\author{Florent Schaffhauser}

\address{Departamento de Matemáticas, Universidad de Los Andes, Carrera 1 \#18A-12, 111 711 Bogot\'a, Colombia \& Institut de Recherche Mathématique Avancée, Universit\'e de Strasbourg, 7 rue Ren\'e Descartes, 67 000 Strasbourg, France.}
\email{florent@uniandes.edu.co, schaffhauser@math.unistra.fr}

\newcommand{\PSL}{\mathbf{PSL}}
\newcommand{\SL}{\mathbf{SL}}
\newcommand{\GL}{\mathbf{GL}}
\newcommand{\PGL}{\mathbf{PGL}}
\newcommand{\PO}{\mathbf{PO}}
\newcommand{\PSO}{\mathbf{PSO}}
\newcommand{\PU}{\mathbf{PU}}
\newcommand{\PSp}{\mathbf{PSp}}
\newcommand{\Sp}{\mathbf{Sp}}
\newcommand{\OO}{\mathbf{O}}
\newcommand{\GG}{\mathbf{G}_2}

\newcommand{\R}{\mathbb{R}}
\newcommand{\C}{\mathbb{C}}
\newcommand{\Z}{\mathbb{Z}}

\newcommand{\PP}{\mathbf{P}}
\newcommand{\HH}{\mathbf{H}}

\newcommand{\cT}{\mathcal{T}}

\newcommand{\cE}{\mathcal{E}}
\newcommand{\cF}{\mathcal{F}}
\newcommand{\cA}{\mathcal{A}}
\newcommand{\cG}{\mathcal{G}}

\newcommand{\cM}{\mathcal{M}}
\newcommand{\cP}{\mathcal{P}}
\newcommand{\cD}{\mathcal{D}}

\newcommand{\fg}{\mathfrak{g}}
\newcommand{\fh}{\mathfrak{h}}
\newcommand{\fk}{\mathfrak{k}}
\newcommand{\fp}{\mathfrak{p}}
\newcommand{\fu}{\mathfrak{u}}
\renewcommand{\sp}{\mathfrak{sp}}
\newcommand{\so}{\mathfrak{so}}

\newcommand{\Aut}{\mathrm{Aut}}
\newcommand{\Out}{\mathrm{Out}}
\newcommand{\Hom}{\mathrm{Hom}}
\newcommand{\End}{\mathrm{End}}
\newcommand{\Int}{\mathrm{Int}}
\newcommand{\Hit}{\mathrm{Hit}}
\newcommand{\Rep}{\mathrm{Rep}}
\newcommand{\id}{\mathrm{Id}}
\newcommand{\Fix}{\mathrm{Fix}}
\newcommand{\Ad}{\mathrm{Ad}}
\newcommand{\ad}{\mathrm{ad}}

\newcommand{\rk}{\mathrm{rk}}

\renewcommand{\phi}{\varphi}
\renewcommand{\rho}{\varrho}
\renewcommand{\sl}{\mathfrak{sl}}
\renewcommand{\H}{\mathbf{H}}

\newcommand{\lra}{\longrightarrow}
\newcommand{\lmt}{\longmapsto}
\newcommand{\bs}{\backslash}
\newcommand{\Ga}{\Gamma}
\newcommand{\ga}{\gamma}
\newcommand{\Si}{\Sigma}
\newcommand{\si}{\sigma}
\newcommand{\sit}{\widetilde{\sigma}}
\newcommand{\Om}{\Omega}
\newcommand{\piY}{\pi_1Y}
\newcommand{\piX}{\pi_1X}
\newcommand{\Yt}{\widetilde{Y}}
\newcommand{\Xt}{\widetilde{X}}
\newcommand{\xt}{\widetilde{x}}
\newcommand{\RP}{\mathbb{R}\mathbf{P}}
\newcommand{\dbar}{\overline{\partial}}
\newcommand{\eps}{\varepsilon}
\newcommand{\fgC}{\mathfrak{g}_{\mathbb{C}}}

\renewcommand{\geq}{\geqslant}
\renewcommand{\leq}{\leqslant}

\newcommand{\cL}{\mathcal{L}}
\newcommand{\cB}{\mathcal{B}}
\newcommand{\te}{\tilde{e}}
\newcommand{\fgm}{\fg^{(d)}}
\newcommand{\fgmC}{\fg^{(d)}_\C}
\newcommand{\Ecan}{\cE_\mathrm{can}}

\newcommand{\ov}[1]{\overline{#1}}

\newtheorem{proposition}{Proposition}[section]
\newtheorem{theorem}[proposition]{Theorem}
\newtheorem{corollary}[proposition]{Corollary}
\newtheorem{lemma}[proposition]{Lemma}

\theoremstyle{definition}
\newtheorem{definition}[proposition]{Definition}
\newtheorem{example}[proposition]{Example}

\newtheorem{remark}[proposition]{Remark}

\allowdisplaybreaks[4]

\date{\today}

\numberwithin{equation}{section}
\numberwithin{figure}{section}

\subjclass[2010]{Primary 22E40, 57M50; Secondary 30F30, 53C07.}
\keywords{Teichm\"{u}ller space, Hitchin component, Orbifold, Real projective structure, Coxeter group}

\begin{document}

\begin{abstract}
We extend the notion of Hitchin component from surface groups to orbifold groups and prove that this gives new examples of higher Teichm\"{u}ller spaces. We show that the Hitchin component of an orbifold group is homeomorphic to an open ball and we compute its dimension explicitly. We then give applications to the study of the pressure metric, cyclic Higgs bundles, and the deformation theory of real projective structures on $3$-manifolds.
\end{abstract}

\maketitle

\contentsline {section}{\tocsection {}{1}{Introduction}}{1}{section.1}
\contentsline {subsection}{\tocsubsection {}{1.1}{Hitchin components}}{1}{subsection.1.1}
\contentsline {subsection}{\tocsubsection {}{1.2}{Results in the orbifold case}}{2}{subsection.1.2}
\contentsline {subsection}{\tocsubsection {}{1.3}{Applications}}{3}{subsection.1.3}
\contentsline {subsubsection}{\tocsubsubsection {}{1.3.1}{Rigidity}}{4}{subsubsection.1.3.1}
\contentsline {subsubsection}{\tocsubsubsection {}{1.3.2}{Geodesics for the pressure metric and one-parameter families of representations of surface groups}}{4}{subsubsection.1.3.2}
\contentsline {subsubsection}{\tocsubsubsection {}{1.3.3}{Cyclic Higgs bundles}}{4}{subsubsection.1.3.3}
\contentsline {subsubsection}{\tocsubsubsection {}{1.3.4}{Projective structures on Seifert manifolds}}{5}{subsubsection.1.3.4}
\contentsline {subsection}{\tocsubsection {}{}{Acknowledgments}}{5}{section*.2}
\contentsline {section}{\tocsection {}{2}{Hitchin representations for orbifolds}}{6}{section.2}
\contentsline {subsection}{\tocsubsection {}{2.1}{Hyperbolic 2-orbifolds}}{6}{subsection.2.1}
\contentsline {subsection}{\tocsubsection {}{2.2}{Principal representation}}{7}{subsection.2.2}
\contentsline {subsection}{\tocsubsection {}{2.3}{Hitchin representations}}{7}{subsection.2.3}
\contentsline {subsection}{\tocsubsection {}{2.4}{Restriction to subgroups of finite index}}{9}{subsection.2.4}
\contentsline {subsection}{\tocsubsection {}{2.5}{Properties of Hitchin representations for closed orbifolds}}{10}{subsection.2.5}
\contentsline {subsection}{\tocsubsection {}{2.6}{Orbifolds with boundary}}{12}{subsection.2.6}
\contentsline {section}{\tocsection {}{3}{Hitchin's equations in an equivariant setting}}{14}{section.3}
\contentsline {subsection}{\tocsubsection {}{3.1}{From orbifold representations to equivariant flat bundles}}{14}{subsection.3.1}
\contentsline {subsection}{\tocsubsection {}{3.2}{From equivariant flat bundles to equivariant harmonic bundles}}{16}{subsection.3.2}
\contentsline {subsection}{\tocsubsection {}{3.3}{From equivariant Higgs bundles to equivariant harmonic bundles}}{18}{subsection.3.3}
\contentsline {subsection}{\tocsubsection {}{3.4}{Non-Abelian Hodge correspondence}}{21}{subsection.3.4}
\contentsline {section}{\tocsection {}{4}{Parameterization of Hitchin components}}{22}{section.4}
\contentsline {subsection}{\tocsubsection {}{4.1}{Equivariance of the Hitchin fibration}}{22}{subsection.4.1}
\contentsline {subsection}{\tocsubsection {}{4.2}{Invariant Hitchin representations}}{24}{subsection.4.2}
\contentsline {section}{\tocsection {}{5}{Invariant differentials}}{25}{section.5}
\contentsline {subsection}{\tocsubsection {}{5.1}{Regular differentials on orbifolds}}{25}{subsection.5.1}
\contentsline {subsection}{\tocsubsection {}{5.2}{The dimension of Hitchin components}}{27}{subsection.5.2}
\contentsline {subsection}{\tocsubsection {}{5.3}{Approximation formula}}{28}{subsection.5.3}
\contentsline {section}{\tocsection {}{6}{Applications}}{29}{section.6}
\contentsline {subsection}{\tocsubsection {}{6.1}{Rigidity phenomena}}{29}{subsection.6.1}
\contentsline {subsection}{\tocsubsection {}{6.2}{Geodesics for the pressure metric}}{31}{subsection.6.2}
\contentsline {subsection}{\tocsubsection {}{6.3}{Cyclic Higgs bundles}}{32}{subsection.6.3}
\contentsline {subsection}{\tocsubsection {}{6.4}{Projective structures on Seifert-fibered 3-manifolds}}{33}{subsection.6.4}
\contentsline {section}{\tocsection {}{}{References}}{36}{section*.3}
\contentsline {section}{\tocsection {}{}{Appendices}}{37}{appendix.A}

\section{Introduction}

\subsection{Hitchin components}

Teichm\"{u}ller space is the deformation space of hyperbolic structures on a closed orientable surface $X$ of genus $g \geqslant 2$. It can also be seen as a connected component of the representation space
$ \Rep(\piX,\PGL(2,\R)) := \Hom(\piX,\PGL(2,\R))/\PGL(2,\R) $
consisting of the conjugacy classes of discrete and faithful representations of $\piX$ into $\PGL(2,\R)$. It is well-known that Teichm\"{u}ller space is homeomorphic to an open ball of dimension $(6g-6)$. In 1992, N.\ Hitchin \cite{Hitchin_Teich} replaced the group $\PGL(2,\R)$ by the split real form $G$ of a complex simple Lie group, and found a special component of $\Rep(\piX,G) := \Hom(\piX,G)/G$ homeomorphic to an open ball of dimension $(2g-2) \dim G$, which is now called the \emph{Hitchin component} of the surface group $\piX$ into $G$. The geometry of the representations in the Hitchin components was studied by Goldman \cite{Goldman90} and Choi and Goldman \cite{CG93} for $G = \PGL(3,\R)$, Labourie \cite{Labourie} and Guichard \cite{Guichard_hyperconvex} for $G=\PGL(n,\R)$, Guichard and Wienhard \cite{GW} for $G=\PGL(4,\R)$ and \cite{GW2} in higher generality. From these works, we see that the Hitchin components share many properties with Teichm\"uller space, and they are part of an interesting family of spaces called \textit{higher Teichm\"uller spaces} (see Wienhard \cite{Wienhard_ICM} for a survey of this theory).

In this paper, we generalize the notion of Hitchin components of surface groups to a more general family of groups, namely fundamental groups of compact $2$-dimensional orbifolds with negative orbifold Euler characteristic. This is a large family, consisting of all groups isomorphic to a convex cocompact subgroup of $\PGL(2,\R)$. It contains in particular the fundamental groups of all surfaces of finite type (with or without boundary, orientable or not), and the $2$-dimensional hyperbolic Coxeter groups. The first instance of Hitchin components for orbifold fundamental groups was studied by Thurston \cite{Thurston_notes} who showed that the Teichm\"uller space (i.e.\ the space of hyperbolic structures on a closed orbifold $Y$) is a connected component of $\Rep(\piY,\PGL(2,\R))$, consisting of the conjugacy classes of discrete and faithful representations of the orbifold fundamental group $\piY$ into $\PGL(2,\R)$, and described its topology. Then Choi and Goldman \cite{CG} studied the Hitchin component for $\piY$ in $\PGL(3,\R)$, describing its topology and showing that it is the deformation space of convex real projective structures on $Y$. Finally, Labourie and McShane \cite{Labourie_McShane} introduced $\PGL(n,\R)$-Hitchin components for orientable surfaces with boundary, in order to generalize McShane-Mirzakhani identities from hyperbolic geometry to arbitrary cross ratios. Here, we study the general case. Given a complex simple Lie algebra $\fg_\C$, we fix a split real form $\fg$ and denote by $\tau$ the corresponding involution of $\fg_\C$. We then define $G = \Int(\fg_\C)^\tau$ to be the group of real points of $G_\C = \Int(\fg_\C)$, the adjoint group of $\fg_\C$. For the classical Lie algebras, $G$ is one of the groups $\PGL(n,\R)$, $\PO(m,m+1)$, $\PSp^{\pm}(2m,\R)$, and $\PO^{\pm}(m,m)$. A representation of $\piY$ in $G$ is said to be \emph{Fuchsian} if it is the composition of a discrete, faithful, convex cocompact representation of $\piY$ in $\PGL(2,\R)$ with the \emph{principal representation} $\kappa:\PGL(2,\R) \lra G$, and \emph{Hitchin representations} are deformations of Fuchsian representations (see Definitions \ref{Hit_rep_def} and \ref{Hitchin_rep_bdry_case}). The space of Hitchin representations of $\piY$ in $G$ up to conjugation will be called the \emph{Hitchin component}, and denoted by $\Hit(\piY,G)$.

\subsection{Results in the orbifold case}

We first extend to the orbifold case known dynamical and geometric properties of $\PGL(n,\R)$-Hitchin representations of surface groups \cite{Labourie,Guichard_hyperconvex,Labourie_McShane,GW2}.

\begin{theorem}[Sections \ref{section:properties} and \ref{boundary_case}]\label{theorem:intro_A}
Let $Y$ be a compact connected $2$-orbifold of negative Euler characteristic. A Hitchin representation $\rho : \piY \lra \PGL(n,\R)$ is $B$-Anosov, where $B$ is an arbitrary Borel subgroup of $\PGL(n,\R)$, and it is discrete, faithful and strongly irreducible. Moreover, for all $\gamma$ of infinite order in $\piY$, the element $\rho(\gamma)$ is diagonalizable with distinct real eigenvalues. If $Y$ is closed, a representation of $\piY$ in $\PGL(n,\R)$ is Hitchin if and only if it is hyperconvex.
\end{theorem}

Theorem \ref{theorem:intro_A} implies that Hitchin components for orbifold groups form a family of higher Teichm\"uller spaces in the sense of \cite{Wienhard_ICM}. Having established this, we determine the topology of Hitchin components for orbifold groups, which is the main result of this paper.

\begin{theorem}[Theorem \ref{theorem:dimension} and Corollary \ref{dimension_bdry_case}]\label{theorem:main}
Let $Y$ be a compact connected $2$-orbifold of negative orbifold Euler characteristic, with $k$ cone points, of respective orders $m_1$, $\ldots$ , $m_k$, and $\ell$ corner reflectors, of respective orders $n_1$, $\ldots$, $n_\ell$. Denote by $b$ the number of boundary components of $Y$ that are full $1$-orbifolds and by $|Y|$ the underlying topological surface of $Y$. Let $G = \Int(\fg_\C)^\tau$, where $\fg$ is a simple split Lie algebra of rank $r$ with exponents $d_1, \dots, d_r$. Then the Hitchin component $\Hit(\piY,G)$ is homeomorphic to an open ball of dimension  
$$-\chi(|Y|)\dim G + \sum_{\alpha=1}^{\rk\, G} 
\left(2 \sum_{i=1}^k O(d_\alpha+1,m_i) + \sum_{j=1}^\ell O(d_\alpha+1,n_j) \, +\,  2 b\, \left\lfloor \frac{d_\alpha+1}{2}\right\rfloor \right),$$
where $O(d,m) = \left\lfloor d-\tfrac{d}{m}\right\rfloor$ denotes the biggest integer not bigger than $\left(d -\tfrac{d}{m}\right)$. 
\end{theorem}

For instance, the $\PGL(2m,\R)$, resp.\ $\PGL(2m+1,\R)$, Hitchin component of the reflection group associated to a right-angled hyperbolic $\ell$-gon ($\ell > 4$) is homeomorphic to an open ball of dimension $(\ell-4)m^2+1$, resp.\ $(\ell-4)(m^2+m)$.

\begin{corollary}[Remark \ref{about_the_Labourie_McShane_component} and Corollary \ref{cases_of_validity_of_Hitchin_s_formula}]
Let $S$ be an orientable surface with boundary. Then the $\PGL(n,\R)$-Hitchin component of $S$ in the sense of Labourie and McShane \cite{Labourie_McShane} is homeomorphic to an open ball of dimension $-\chi(S)(n^2-1)$.
\end{corollary}

We also prove an alternative formula for the dimension of Hitchin components, more similar to the ones given by Thurston \cite{Thurston_notes} and Choi and Goldman \cite{CG} for $G=\PGL(2,\R)$ and $\PGL(3,\R)$.

\begin{theorem}[Theorem \ref{alternate_formula_for_dim} and Corollary \ref{alternate_formula_for_dim_bdry_case}]\label{intro_alternate_formula_for_dim}
Under the assumptions of Theorem \ref{theorem:main}, set $k_m := \#\{i \mid m_i = m \}$, $\ell_n := \#\{j \mid n_j = n \}$ and $M:=\max_{1\leqslant \alpha\leqslant r}d_\alpha$. Then 
$$\begin{array}{rcl} \dim(\Hit(\piY,G)) & = & \displaystyle
- \chi(|Y|) \dim G + \frac{1}{2}(\dim G - \rk\,G) \left( 2k + \ell +2b \right) - 2 \sum_{m=2}^{M} \left( \sum_{\alpha = 1}^{\rk\,G} \left\lceil \frac{d_\alpha + 1}{m} - 1 \right\rceil \right) k_m \\
 & &\displaystyle - \sum_{n=2}^{M} \left( \sum_{\alpha = 1}^{\rk\,G} \left\lceil \frac{d_\alpha + 1}{n} - 1 \right\rceil \right) \ell_n\, - 2b  \sum_{\alpha = 1}^{\rk\,G} \left\lceil \frac{d_\alpha - 1}{2} \right\rceil ,
 \end{array}$$ where $\lceil x \rceil$ denotes the smallest integer not smaller than $x$.
\end{theorem}

\begin{remark}
In the special case where $Y$ is a sphere with $3$ cone points and $G = \PGL(n,\R)$ (resp.\ $G = \PSp^{\pm}(2m,\R)$ or $\PO(m,m+1)$), Long and Thistlethwaite \cite{Long_Thistlethwaite} (resp.\ Weir \cite{Weir}) have computed the dimension of the Hitchin component of $\piY$ into $G$. Our result confirms their formulas, and in addition shows that those Hitchin components are homeomorphic to open balls. 
\end{remark}
 
The key to the proof of Theorems \ref{theorem:main} and \ref{intro_alternate_formula_for_dim} is that Hitchin components for orbifold groups can be seen as subspaces of Hitchin components for surface groups. Note that this statement does \textit{not} hold in general: it is a special property of Hitchin components. Indeed, when an orbifold $Y$ of negative Euler characteristic is seen as the quotient of a closed orientable surface $X$ by the action of a finite group $\Sigma$, the map $j:\Hom(\piY,G)/G \lra \Fix_{\Si}(\Hom(\piX,G)/G)$ taking $\rho:\piY\lra G$ to $\rho|_{\piX}:\piX\lra G$ is in general neither injective nor surjective. But, \emph{when restricted to} $\Hit(\piY,G)$, it becomes a homeomorphism onto $\Fix_{\Si}(\Hit(\piX,G))$.

\begin{theorem}[Theorem \ref{main_obs_about_Hit_comp_for_orbifolds}]\label{main_obs_intro}
Let $Y$ be a closed connected $2$-orbifold of negative Euler characteristic and let $G = \Int(\fg_\C)^\tau$. Given a presentation $Y \simeq [\Si \bs X]$, the map $\rho\lmt\rho|_{\piX}$ induces a homeomorphism 
$j:\Hit(\piY,G) \overset{\simeq}{\lra} \Fix_{\Si}(\Hit(\piX,G))$, 
between the Hitchin component of $\Rep(\piY,G)$ and the $\Si$-fixed locus in $\Hit(\piX,G)$.
\end{theorem}

In order to prove Theorem \ref{main_obs_intro}, we develop the following  $\Si$-equivariant version of the non-Abelian Hodge correspondence.

\begin{theorem}[Section \ref{NAHC_for_orbifolds}]
Under the assumptions of Theorem \ref{main_obs_intro}, there is a homeomorphism between the representation space $\Hom^{\mathrm{c.r.}}(\piY,G)/G$ of completely reducible representations of the orbifold fundamental group $\piY$ into $G$ and the moduli space  $$\cM_{(X,\Si)}(G) := \left\{
\begin{matrix}
\Si\textrm{-polystable equivariant}\ G\textrm{-Higgs bundles}\\ 
(\cE,\phi,\tau)\ \textrm{with vanishing first Chern class on}\ X
\end{matrix}
\right\} \big/\, \mathrm{isomorphism}$$ of isomorphism classes of $\Si$-polystable equivariant $G$-Higgs bundles with vanishing first Chern class on $X$.
\end{theorem}

We then prove that the Hitchin fibration with Hitchin base $\cB_X(\fg)$ admits a $\Sigma$-equivariant section (the Hitchin section), thus inducing a homeomorphism $\Fix_\Si(\cB_X(\fg)) \simeq \Fix_\Si(\Hit(\piX,G))$ (Lemma \ref{equivariance_of_the_Hitchin_section}), and we show how this implies Theorem \ref{main_obs_intro}, as well as the following result. 

\begin{corollary}[Corollary \ref{Sigma_fixed_pts_in_Hitchin_base_of_X}]\label{invariant_diffs_intro}
Under the assumptions of Theorem \ref{main_obs_intro}, the Hitchin component $\Hit(\piY,G)$ is homeomorphic to the real vector space $\Fix_\Si(\cB_X(\fg))$. In particular, it is a contractible space.
\end{corollary}

In order to define the Hitchin base $\cB_Y(\fg)$ of $Y$ (see \eqref{Hitchin_base_of_an_orbifold}), we introduce spaces of regular differentials on orbifolds (Definition \ref{reg_diff_def}), compute their dimension (Theorem \ref{prop:dimension diff non-orient}), and prove that the Hitchin component $\Hit(\piY,G)$ is homeomorphic to $\cB_Y(\fg)$ (Theorem \ref{theorem:dimension}), which completes the proof of Theorems \ref{theorem:main} and \ref{intro_alternate_formula_for_dim}, as well as Corollary \ref{cor_orientation_cover_intro} below. Note that Theorem \ref{prop:dimension diff non-orient} provides an explanation why numbers of the form $O(d,m)=\lfloor d -\frac{d}{m} \rfloor$ appear in the formula for the dimension in Theorem \ref{theorem:main}. 

\begin{corollary}[Corollary \ref{cor:non orient half}]\label{cor_orientation_cover_intro}
Let $Y$ be a closed non-orientable orbifold and let $Y^+\lra Y$ be its orientation double cover. Then
$\dim \Hit(\piY,G) = \frac{1}{2} \dim \Hit(\pi_1 Y^+,G)$.
\end{corollary}

\subsection{Applications} 

We list below a few applications of our results, to be presented in Section \ref{section:application}. Theorem \ref{theorem:main} also found an application to the study of the theory of compactifications of the character varieties, see Burger, Iozzi, Parreau and Pozzetti \cite{BIPP}.

\subsubsection{Rigidity} We encounter the following two types of rigidity phenomena.

\begin{theorem}[Theorem \ref{thm:dimension 0}]\label{no_non_trivial_deformations_intro}
Let $G=\Int(\fg_\C)^\tau$ and let $Y$ be a closed orientable orbifold of genus $0$ with $3$ cone points, of respective orders $m_1 \leqslant m_2 \leqslant m_3$. Assume that the tuple $(G,m_1,m_2,m_3)$ satisfies one of the following conditions:

\begin{enumerate}
\item $G=\PGL(2,\R) \simeq \PO(1,2) \simeq \PSp^\pm(2,\R)$ and $\frac{1}{m_1}+\frac{1}{m_2}+\frac{1}{m_3}<1$.
\item $G=\PGL(3,\R)$, $m_1=2$ and $\frac{1}{m_2}+\frac{1}{m_3}<\frac{1}{2}$.
\item $G=\PGL(4,\R)\simeq \PO^\pm(3,3)$, $\PGL(5,\R)$, $\PSp^{\pm}(4,\R)\simeq\PO(2,3)$, $m_1 = 2$, $m_2 = 3$ and $m_3\geqslant7$.
\item $G=\PSp^{\pm}(4,\R)\simeq\PO(2,3)$, $m_1=m_2=3$ and $m_3 \geqslant4$.
\item $G=\GG$, $m_1=2$ and $m_2=4$ or $5$, and $m_3=5$.
\end{enumerate}

Then $\dim \Hit(\piY,G) =0$, so any two Hitchin representations of $\piY$ into $G$ are $G$-conjugate in this case, and this happens for infinitely many orbifolds. 

Moreover, for all other pairs $(G,Y)$ with $Y$ closed orientable, Hitchin representations of $\piY$ into $G$ admit non-trivial deformations, i.e.\ $\dim\Hit(\piY,G)>0$.
\end{theorem}

\noindent Cases (1) and (2) above are already known \cite{Thurston_notes,CG}. If $Y$ is non-orientable, Corollary \ref{cor_orientation_cover_intro} shows that $\dim\Hit(\piY,G)=0$ if and only if $Y$ is a quotient of one of the (infinitely many) spheres with cone points listed in Theorem \ref{no_non_trivial_deformations_intro}, i.e.\ $Y$ is either a disk with three corner reflectors or a disk with one cone point and one corner reflector. For surface groups, the dimension of Hitchin components is always positive and grows quadratically with the rank of the group (this last property also holds for orbifold groups: Proposition \ref{proposition:dim_estimate}).

The second type of rigidity phenomenon that we encounter has to do with Zariski density of representations of $\piY$ into $G$: \textit{those may not exist in Hitchin components for orbifold groups, and we find infinite families of such groups}. This is surprising, and contrasts with what happens for surface groups, for which the subset of Zariski dense representations is always dense in the Hitchin component.

\begin{theorem}[Theorem \ref{thm_none_is_zariski_dense}]\label{absence_of_Zariski_dense_reps_intro}
Let $G=\Int(\fg_\C)^\tau$ and let $Y$ be an orientable orbifold of genus $g$ with $k$ cone points, of respective orders $m_1 \leqslant\, \ldots\, \leqslant m_k$. If the triple $(Y,G,H)$ is one of those listed in Theorem \ref{thm_none_is_zariski_dense}, then the image of a Hitchin representation $\rho:\piY\lra G$ lies in a conjugate of $H$ in $G$. In particular, a Hitchin representation of $\piY$ into $G$ can never be Zariski-dense. For all other triples $(Y,G,H)$, this phenomenon does not occur.
\end{theorem}

\subsubsection{Geodesics for the pressure metric and one-parameter families of representations of surface groups}\label{intro_applications_geodesics}

In view of Theorem \ref{main_obs_intro}, Hitchin components for orbifold groups may be considered as submanifolds of Hitchin components for surface groups. These submanifolds are totally geodesic for all $\Out(\piX)$-invariant Riemannian metrics on Hitchin components, for instance the \emph{Pressure metric} \cite{BCLS_pressure} and the \emph{Liouville pressure metric} \cite{BCLS_liouville} (Proposition \ref{tot_geod_submanifold}). In particular, one-dimensional Hitchin components provide explicit examples of geodesics for these metrics. In Section \ref{subsection_on_geodesics}, we classify all Hitchin components of dimension $1$ and we prove that, for the group $\PGL(n,\R)$, one-dimensional Hitchin components exist if and only if $n \leqslant 11$ (Theorem \ref{thm:dimension 2}). We find in this way natural, geometric examples of $1$-parameter families of representations of surface groups, parametrized by spaces of holomorphic differentials.

\begin{example}\label{example_of_geodesic}
Let $\mathcal{K}$ be the Klein quartic, a Riemann surface of genus $3$, and let $n$ be an integer such that $6\leqslant n\leqslant 11$. Then the orbifold $\mathcal{T}:=\Aut^\pm(\mathcal{K})\bs \mathcal{K}$ is a triangle of type $(2,3,7)$ and the $\PGL(n,\R)$-Hitchin component of $\pi_1\mathcal{T}$ embeds onto a geodesic of $\Hit(\pi_1\mathcal{K},\PGL(n,\R))$.
\end{example}

\subsubsection{Cyclic Higgs bundles}\label{cyclic_bundles_intro}

In Section \ref{cyclic_Higgs_bundles_section} we extend the notion of cyclic and $(n-1)$-cyclic Higgs bundles to Hitchin representations of orbifold groups. In the case of surface groups, these notions were introduced in \cite{Baraglia_thesis,Collier_thesis}. These special Higgs bundles are particularly useful because the Hitchin equations can be put in a simplified form, where the analysis can be understood, and many results can be proved only in this case, see e.g. \cite{Baraglia_thesis, Collier_thesis, CollierLi, DaiLi1, DaiLi2}.

We show that the same properties are true for cyclic and $(n-1)$-cyclic representations in the Hitchin components of orbifold groups. Moreover, we prove the following:

\begin{theorem}[Corollary \ref{comp_with_only_cyclic_bdles}] \label{components_with_only_cyclic_Higgs_bundles}
Let $Y$ be a sphere with $k$ cone points of respective orders $m_1 \leqslant \dotsc \leqslant m_k$ and let $G$ be one of the groups listed in Table \ref{table:cyclic_Higgs}. Then $\Hit(\pi_1 Y, G)$ consists only of cyclic or $(n-1)$-cyclic representations.
\end{theorem}

This phenomenon never happens for surface groups: it is specific to certain orbifolds. In this case, the results about cyclic or $(n-1)$-cyclic representations are valid for \emph{all} points of these Hitchin components. For example, the description of the asymptotic behavior of families of Higgs bundles going at infinity given in \cite{CollierLi} gives a good description of the behavior at  infinity of these Hitchin components.

The proof of Theorem \ref{components_with_only_cyclic_Higgs_bundles} comes from a classification of all the Hitchin components that are parametrized by a Hitchin base where only a differential of one type can appear (see Theorem \ref{thm:single_differential}). We then find an application of Theorem \ref{components_with_only_cyclic_Higgs_bundles}: we use orbifold groups to construct examples of representations of surface groups that lie in some special loci of the Hitchin components that are not well understood geometrically, see Corollary \ref{corol:reps in special loci}.

\subsubsection{Projective structures on Seifert manifolds}

In \cite{GW}, Guichard and Wienhard proved that the Hitchin component of a surface group in $\PGL(4,\R)$ is the deformation space of convex foliated projective structures on the unit tangent bundle of that surface, and we give below a generalization of their result for arbitrary finite covers of unit tangent bundles of closed orientable orbifolds. Let $M$ be a closed $3$-manifold and let $\cD_{\PSL(2,\R)}(M)$ be the deformation space of $\PSL(2,\R)$-structures on $M$. We denote by $\cD_{\RP^3}(M)$ the deformation space of projective structures on $M$ and by $\cD_{\RP_\omega^3}(M)\subset \cD_{\RP^3}(M)$  the deformation space of contact projective structures.

\begin{theorem}[Proposition \ref{Seifert-fibered_three_mflds}  
and Theorem \ref{theorem:deformation_space_3manifold_Hitchin}]\label{deformation_space_3manifold_Hitchin_intro}
If $\cD_{\PSL(2,\R)}(M)\neq\emptyset$, then $M$ is a finite cover of the unit tangent bundle of a well-defined closed orientable $2$-orbifold $Y$, and the image of the canonical map $$\cD_{\PSL(2,\R)}(M) \lra \cD_{\RP_\omega^3}(M)\subset \cD_{\RP^3}(M)$$ is homeomorphic to $\Hit(\piY,\PGL(2,\R))$. It picks out connected components of $\cD_{\RP_\omega^3}(M)$ and $\cD_{\RP^3}(M)$ respectively homeomorphic to $\Hit(\piY,\PSp^\pm(4,\R))$ and $\Hit(\piY,\PGL(4,\R))$. In particular, these components are homeomorphic to open balls of dimensions $-10 \chi(|Y|) + (8k-2k_2 -2 k_3)$ and $-15 \chi(|Y|) + (12 k - 4 k_2 - 2 k_3)$. 
\end{theorem}

Put together with Theorem \ref{no_non_trivial_deformations_intro}, this enables us to produce examples of closed $3$-manifolds with rigid real projective structure (see also Diagram \eqref{rigid_proj_structures_diagram}). More precisely, let $M$ be a finite cover of the unit tangent bundle of a closed orientable orbifold $Y$ and recall from Theorem \ref{deformation_space_3manifold_Hitchin_intro} that $\Hit(\piY,\PGL(2,\R))\subset \cD^{\,0}_{\RP^3_\omega}(M)\subset\cD^{\,0}_{\RP^3}(M)$. So, if $\cD^{\,0}_{\RP^3_\omega}(M)$ or $\cD^{\,0}_{\RP^3}(M)$ is zero-dimensional, $Y$ has to be a sphere with three cone points.

\begin{corollary}\label{rigid_proj_structures_intro}
Let $M$ be a finite cover of the unit tangent bundle of a sphere with three cone points, of respective orders $m_1\leqslant m_2\leqslant m_3$ with $\frac{1}{m_1}+\frac{1}{m_2}+\frac{1}{m_3}<1$. We then have $\cD_{\PSL(2,\R)}(M)=\mathrm{pt}$ and:
\begin{enumerate}
\item If $m_1 = 2$, $m_2 = 3$ and $m_3\geqslant7$, then $\cD^{\,0}_{\RP^3_\omega}(M)=\cD^{\,0}_{\RP^3}(M)=\mathrm{pt}$, so the canonical projective structure of $M$ is rigid in that case (and it is a contact projective structure). 
\item If $m_1=m_2=3$ and $m_3 \geqslant4$, then $\cD^{\,0}_{\RP^3_\omega}(M)=\mathrm{pt}$ but $\cD^{\,0}_{\RP^3}(M)$ has positive dimension: $M$ is contact rigid but not projectively rigid.
\item For all other triples $(m_1,m_2,m_3)$, $\cD^{\,0}_{\RP^3_\omega}(M)$ and $\cD^{\,0}_{\RP^3}(M)$ have positive dimension. 

\end{enumerate} 
\end{corollary}

Finally, it will be a consequence of Part (b) of Theorem \ref{thm_none_is_zariski_dense} that we can have $\cD^{\,0}_{\RP^3}(M)=\cD^{\,0}_{\RP^3_\omega}(M) \neq \mathrm{pt}$ if $Y$ is a sphere with $k$ cone points, namely when $k=3$, $m_1=2$ and $m_2,m_3\geqslant 4$, or $k=4$ and exactly one cone point has order greater than $2$, or $k=5$ and all cone points have order $2$. So any non-trivial deformation of the canonical projective structure of $M$ is contact in these cases.

\subsection*{Acknowledgments}

We thank Brian Collier, Olivier Guichard, Qiongling Li, Anna Wienhard and Tengren Zhang for helpful conversations, and the referee for various suggestions that have helped improve the exposition of the material presented in this paper.

D. Alessandrini was supported by the DFG grant AL 1978/1-1 within the Priority Programme SPP 2026 ``Geometry at Infinity''. G.-S. Lee was supported by the DFG research grant ``Higher Teichm\"uller Theory'', by the European Research Council under ERC-Consolidator Grant 614733, and by the DFG grant LE 3901/1-1 within the Priority Programme SPP 2026 ``Geometry at Infinity''. F. Schaffhauser was supported by \textit{Convocatoria 2018-2019 de la Facultad de Ciencias (Uniandes), Programa de investigaci\'on ``Geometr\'ia y Topolog\'ia de los Espacios de M\'odulos''}, the \textit{European Union's Horizon 2020 research and innovation programme under grant agreement No 795222} and the \textit{University of Strasbourg Institute of Advanced Study (USIAS)}. The authors acknowledge support from U.S. National Science Foundation grants DMS 1107452, 1107263, 1107367 ``RNMS: Geometric structures And Representation varieties'' (the GEAR Network).

\section{Hitchin representations for orbifolds}\label{Hitchin_reps_section}

\subsection{Hyperbolic 2-orbifolds} 

For background on orbifolds, we refer for instance to \cite{Thurston_notes,Scott,CG}. Let $Y$ be a closed connected smooth orbifold of dimension $2$. Recall that a \emph{singularity} of $Y$ is a point $y\in Y$ of one of the following three types:
\begin{enumerate}
\item a \emph{cone point} of order $m$ if $y$ has a neighborhood isomorphic to $(\sfrac{\Z}{m\Z})\bs \R^2$, where $\sfrac{\Z}{m\Z}$ acts on $\R^2$ via a rotation of angle $\frac{2\pi}{m}$,
\item a \emph{mirror point} if $y$ has a neighborhood isomorphic to $(\sfrac{\Z}{2\Z})\bs \R^2$, where $\sfrac{\Z}{2\Z}$ acts on $\R^2$ via a reflection though a line,
\item a \emph{corner reflector} (or \emph{dihedral point}) of order $n$ if $y$ has a neighborhood isomorphic to $D_n\bs\R^2$, where the action of the dihedral group $D_n \simeq (\sfrac{\Z}{n\Z}) \rtimes (\sfrac{\Z}{2\Z})$ on $\R^2$ is generated by the reflections through 
two lines with angle $\frac{\pi}{n}$ between them.
\end{enumerate} 

In the rest of the paper, for an orbifold $Y$, we will denote by $k$ the number of cone points (of respective orders $m_1$, $\ldots$ , $m_k$) and by $\ell$ the number of corner reflectors (of respective orders $n_1$, $\ldots$, $n_\ell$). We will denote by $\Yt$ the orbifold universal cover of $Y$ (for more details, see \cite[Section 2]{Scott}); recall that $\Yt$ is necessarily simply connected as a topological space but, in general, it may have non-trivial orbifold structure. For example, the teardrop orbifold (the rightmost orbifold in Figure \ref{fig:eleorb}) has underlying topological space $S^2$, has a single cone point of order $m \geq 2$ and it is its own universal cover. We will denote by $\piY$ the \emph{orbifold fundamental group} of $Y$, defined as the group $\piY:=\Aut_Y(\Yt)$ of deck transformations of $\Yt$. 
The underlying topological space $|Y|$ of a $2$-orbifold $Y$ is always homeomorphic to a compact surface, which has boundary if and only if $Y$ has mirror points, in which case $\partial |Y|$ is the set of mirror points and corner reflectors of $Y$. A $2$-orbifold $Y$ is called \emph{orientable} if $|Y|$ is orientable and $Y$ has only cone points as singularities. For instance, the universal orbifold covering $\Yt$ is always orientable as an orbifold. Recall that $Y$ may be non-orientable as an orbifold even though $|Y|$ is an orientable surface (this happens if and only if $|Y|$ is an orientable topological surface with non-empty boundary). Note that the setting that we have just described includes the case where $|Y|$ is non-orientable as a topological surface (possibly with boundary). In particular, our results will hold for non-orientable surfaces with trivial orbifold structure (or with only mirror points as orbifold singularities).

\begin{figure}[ht]
\centering
\includegraphics[width=10.0cm]{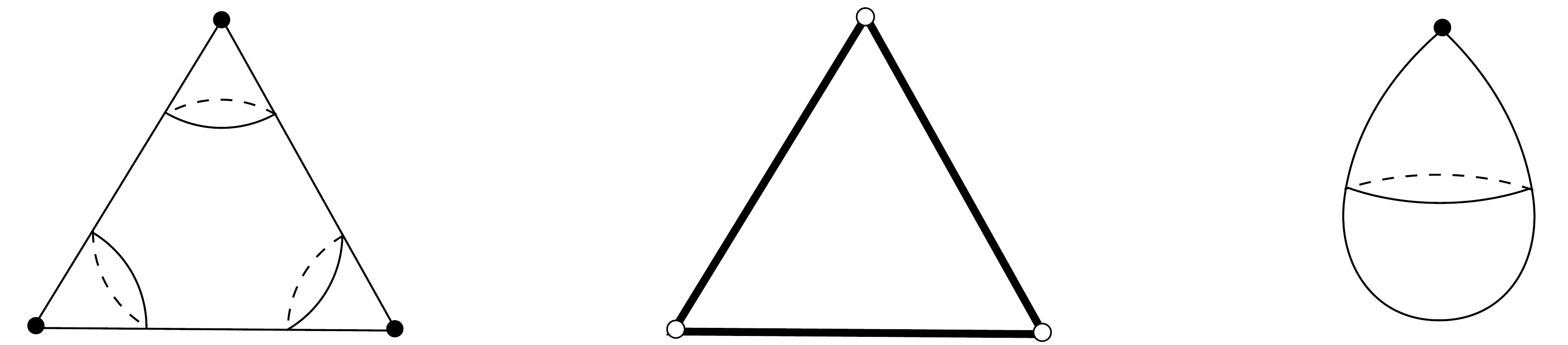}
\caption{These three examples are all closed orbifolds: the leftmost example is a sphere with three cone points (black dots); in the middle example, points lying on the sides of the triangle (excluding vertices) are mirror points, while the vertices are corner reflectors (white dots). The leftmost orbifold is the orientation double cover of the middle one. The rightmost example is a simply connected orbifold which is not a manifold.}\label{fig:eleorb}
\end{figure}

We shall assume throughout that $Y$ has negative (orbifold) Euler characteristic, i.e. the rational number $\chi(Y)$ defined below is strictly negative: 
\begin{equation} \label{form:def euler}
\chi(Y) := \chi(|Y|) - \sum_{i=1}^k \left(1-\frac{1}{m_i}\right) - \frac{1}{2} \sum_{j=1}^\ell \left(1-\frac{1}{n_j}\right) < 0. 
 \end{equation}

\noindent Every orbifold of negative Euler characteristic can always be seen as a quotient of a closed orientable surface, as follows from Selberg's lemma. 

\begin{definition}\label{pres_of_Y_def}
A \emph{presentation} of a closed connected orbifold $Y$ is a triple $(X,\Si,\phi)$, where $X$ is a closed connected orientable surface, $\Si$ a finite subgroup of $\mathrm{Diff}(X)$, and $\phi$  an orbifold isomorphism $\phi: Y \lra \Si \bs X$.
\end{definition}

In other words, a presentation is a finite, Galois, orbifold cover $p:X\lra Y$ of $Y$ by a closed, connected, orientable surface $X$. In the following, to keep the notation more compact, we will denote a presentation $(X,\Si,\phi)$ simply by $Y \simeq [\Si \bs X]$, leaving $\phi$ implicit. Crucially for us, this implies the existence of a short exact sequence:
\begin{equation}\label{fund_ses}
1\lra \piX \lra\piY \lra \Si \lra 1.
\end{equation} 
The group homomorphism $\Si \subset \mathrm{Diff}(X) \lra \mathrm{MCG}(X)\simeq \Out(\piX)$
taking a transformation $\si:X\lra X$ to its (extended) mapping class coincides, through the Dehn-Nielsen-Baer theorem (see e.g. \cite[Theorem 8.1]{Farb_Margalit}), with the canonical group homomorphism $\Si \lra \Out(\piX)$ induced by the short exact sequence \eqref{fund_ses}. In general, it does not lift to a group homomorphism $\Si\lra\Aut(\piX)$ (it does if $\Si$ happens to have a global fixed point in $X$, in which case the short exact sequence \eqref{fund_ses} splits and $\piY$ is isomorphic to the semidirect product $\piX\rtimes \Si$, see \cite{Sch_CRM}; this fact will be used in Remark \ref{about_the_Labourie_McShane_component}).

\subsection{Principal representation}\label{subsection:principal}

Let $\fg$ be a (real or complex) semisimple Lie algebra. Recall that the adjoint representation $\ad:\fg \lra \End(\fg)$ is faithful. Its image $\ad(\fg)$ is a subalgebra of $\End(\fg)$ isomorphic to $\fg$. The \emph{adjoint group} of $\fg$, denoted by $\Int(\fg)$, is defined as the connected Lie subgroup of $\GL(\fg)$ whose Lie algebra is $\ad(\fg)$. This group has trivial center. In the rest of the paper, when $\fgC$ is a complex semisimple Lie algebra, we will denote its adjoint group by $G_\C := \Int(\fgC)$. A \emph{real form} of $\fgC$ is a real Lie subalgebra $\fg\subset \fgC$ which is the set of fixed points of a real involution $\tau:\fgC \lra \fgC$. The involution $\tau$ also induces an involution on $\Int(\fgC)$. We will denote by $G$ the group $G = {\Int(\fgC)}^\tau < G_\C$ consisting of all the inner automorphisms of $\fgC$ that commute with $\tau$. The group $G$ is a real semisimple Lie group with Lie algebra $\fg$. It has trivial center, but it is not connected in general. Its identity component is the group $\Int(\fg)$.  

\begin{example}\label{basic_examples_of_gps}
Here are some examples:
\begin{itemize}
\item If $\fg=\sl(n,\R)$, then $G_\C\simeq \PSL(n,\C)$ and $G \simeq \PGL(n,\R) \simeq \PSL^{\pm}(n,\R)$, which is connected if and only if $n$ is odd, and $\Int(\fg)
\simeq\PSL(n,\R)$ for all $n$. Here, for each subgroup $H$ of $\GL(n,\mathbb{K})$, where $\mathbb{K}=\R$ or $\C$, we denote by $\mathbf{P}H$ the projectivization of $H$, i.e. $\mathbf{P}H = H / (H \cap C)$ with $C$ the center of $\GL(n,\mathbb{K})$, and
$$\SL^{\pm}(n,\R) = \{ A \in \GL(n,\R) \mid \mathrm{det}(A) = \pm 1 \}.$$
\item If $\fg=\sp(2m,\R)$, then $G_\C\simeq \PSp(2m,\C)$ and $G\simeq \PSp^\pm(2m,\R)$, which has two connected components. We recall that given a symplectic form $\omega$ on $\R^{2m}$, we have
$$\Sp^\pm(2m,\R) = \{ A \in \GL(2m,\R) \mid A^T \omega A = \pm \omega \}.$$
\item If $\fg=\so(p,q)$ with $0<p\leq q$, then $G_\C\simeq \PO(p+q,\C)$ and $G\simeq \PO^{\pm}(p,q)$, which is always disconnected. Again recall that given a non-degenerate bilinear form $J$ of signature $(p,q)$, we have $$ \OO^\pm(p,q) = \{ A \in \GL(p+q,\R) \mid A^T J A = \pm J \}. $$ If $p \neq q$, then $\mathbf{O}^{\pm}(p,q) = \mathbf{O}(p,q)$. In this paper, we are mainly interested in the case when $\fg$ is split, i.e. $\fg=\so(m,m+1)$ or $\so(m,m)$.
\item If $\fg$ is the split real form of type $G_2$, we will denote $G_\C$ by $\GG(\C)$, and $G$ simply by $\GG$, a disconnected group. We will consider $\GG(\C)$ as a subgroup of $\PO(7,\C)$ and $\GG$ as a subgroup of $\PO(3,4)$. 
\end{itemize}
\end{example}  

Let us assume, from now on, that $\fg$ is the split real form of a complex simple Lie algebra $\fgC$, defined by an involution $\tau$. As in \cite{Hitchin_Teich}, we can choose a principal $3$-dimensional subalgebra $\mathfrak{sl}(2,\C)\hookrightarrow \fg_\C$ such that $\mathfrak{sl}(2,\C)$ is $\tau$-invariant and induces a subalgebra $\mathfrak{sl}(2,\R)\hookrightarrow \fg$.
Denote by
$\kappa_\C: \PGL(2,\C) \simeq \Int(\mathfrak{sl}(2,\C))
\lra G_\C
$
the induced group homomorphism, and let 
\begin{equation}\label{re_ppal_sl2}
\kappa : \PGL(2,\R)\lra G
\end{equation}
be its restriction to the subgroup $\PGL(2,\R) < \PGL(2,\C)$. We will call $\kappa$ the \emph{principal representation} of $\PGL(2,\R)$ in $G$. In this paper, we use that the representation $\kappa$ is defined on the whole group $\PGL(2,\R)$.

In the examples discussed above (Example \ref{basic_examples_of_gps}), the principal representation $\kappa$ can be described explicitly. Consider the $n$-dimensional vector space $H_{n-1}$ of homogeneous polynomials of degree $(n-1)$ in two variables $X, Y$. A matrix $A\in \GL(2,\R)$ induces a linear map $\widetilde{\kappa}(A)$ that sends a polynomial $P(X,Y) \in H_{n-1}$ to the polynomial $P(A^{-1}\cdot (X,Y))$. This gives an explicit irreducible representation $\widetilde{\kappa}:\GL(2,\R)\lra \GL(n,\R)$ whose projectivization is conjugate to the principal representation $\kappa:\PGL(2,\R)\lra \PGL(n,\R)$. In this way, we can see that $\kappa$ makes the Veronese embedding $\RP^1 \ni \left[a:b\right] \lmt \left[(a X - b Y)^{n-1}\right] \in \mathbf{P}(H_{n-1})$ $\PGL(2,\R)$-equivariant. If $n=2m$ is even, the image of $\widetilde{\kappa}$ is contained in $\Sp^{\pm}(2m,\R)$, and if $n=2m+1$ is odd, it is contained in $\OO(m,m+1)$. If $n=7$, then the projective image of $\widetilde{\kappa}$ is contained in $\GG$. So the projectivization of $\widetilde{\kappa}$ is an explicit model for the principal representation in $\PGL(n,\R)$, $\PSp^{\pm}(2m,\R)$, $\PO(m,m+1)$ and $\GG$. The principal representation in $\PO^{\pm}(m,m)$ is given by the composition of $\kappa:\PGL(2,\R)\lra \PO(m-1,m)$ with the block embedding $\PO(m-1,m) \hookrightarrow \PO^{\pm}(m,m)$.

\subsection{Hitchin representations} 

Thurston \cite{Thurston_notes} studied the space of hyperbolic structures on a closed $2$-orbifold of negative Euler characteristic, called the Teichm\"uller space of $Y$ and denoted by $\cT(Y)$. The map sending a hyperbolic structure to its holonomy representation induces a homeomorphism between $\cT(Y)$ and a connected component of the representation space
$$\Rep(\piY,\PGL(2,\R)) := \Hom(\piY,\PGL(2,\R))/\PGL(2,\R).$$ 
This connected component consists exactly of the $\PGL(2,\R)$-conjugacy classes of discrete and faithful representations from $\piY$ to $\PGL(2,\R) \simeq \mathrm{Isom}(\mathbf{H}^2)$. Such representations are usually called \emph{Fuchsian representations} and, in what follows, we will constantly identify $\cT(Y)$ with the space of the conjugacy classes of Fuchsian representations. Thurston proved that  $\cT(Y)$ is homeomorphic to an open ball of dimension
\begin{equation}\label{Thurston_formula_for_PGL2}
-\chi(|Y|)\dim\PGL(2,\R) +2k+\ell=-3\chi(|Y|)+2k+\ell.
\end{equation}

Let $\fg$ be the split real form of a complex simple Lie algebra $\fgC$ and let $G=\Int(\fgC)^\tau$ as in Section \ref{subsection:principal}. In this paper we will study the Hitchin component, a connected component of the representation space 
$$\Rep(\piY,G) := \Hom(\piY,G)/G$$
that generalizes the Teichm\"uller space. The first step is to use the principal representation to define Fuchsian representations taking values in $G$. 

\begin{definition}[Fuchsian representation]\label{Fuchsian_rep} Let $Y$ be a closed connected $2$-orbifold of negative Euler characteristic. A group homomorphism $\rho:\piY\lra G$ is called a \emph{Fuchsian representation} if there is a discrete, faithful representation $h:\piY \lra \PGL(2,\R)$ such that $\kappa\circ h=\rho$, where $\kappa$ is the principal representation from (\ref{re_ppal_sl2}).
\end{definition}

\noindent Definition \ref{Fuchsian_rep} says that a representation $\rho:\piY\lra G$ is Fuchsian if and only if there exists a hyperbolic structure on $Y$ whose holonomy representation $h$ makes the following diagram commute.

$$\xymatrix{
& \PGL(2,\R) \ar^{\kappa}[d]\\ 
\piY \ar^{\rho}[r] \ar@{-->}^{h}[ru] & G
}$$

\noindent In particular, as $\chi(Y)<0$, there exist Fuchsian representations of $\piY$. The set of $G$-conjugacy classes of Fuchsian representations is called the \emph{Fuchsian locus} of $\Rep(\piY,G)$. This defines a continuous map (which is actually injective, see Corollary \ref{embedding_of_Teichmuller_space}) 
\begin{equation}\label{Fuchsian_embedding}
\cT(Y) \lra \Rep(\piY,G)
\end{equation}
\noindent from the Teichm\"uller space onto the Fuchsian locus of the representation space. Since $\cT(Y)$ is connected, the Fuchsian locus is contained in a well-defined connected component of $\Rep(\piY,G)$ called the \textit{Hitchin component} and denoted by $\Hit(\piY,G)$. For instance, $\Hit(\piY,\PGL(2,\R)) \simeq \cT(Y)$. As any two principal $3$-dimensional subalgebras $\sl(2,\R)\subset \mathfrak{g}$ are related by an interior automorphism of $\fg$ (see \cite{Kostant_Rallis_orbits_in_symmetric_spaces}), the map \eqref{Fuchsian_embedding} does not depend on that particular choice in the construction.

\begin{definition}[Hitchin representation]\label{Hit_rep_def}
Let $Y$ be a closed connected $2$-orbifold of negative Euler characteristic. A group homomorphism $\rho:\piY\lra G$ is called a \textit{Hitchin representation} if its $G$-conjugacy class $[\varrho]$ is an element of the Hitchin component $\Hit(\piY,G)$.
\end{definition}

\begin{remark}
It follows from the definition of a Fuchsian representation that if $Y$ is orientable (for instance, if $Y=X$ is a closed orientable surface), then any Fuchsian representation of $\piY$ in $G$ is in fact contained in $\Hom(\piY,G_0)$, where $G_0$ is the identity component of $G$ (because the holonomy representation of any hyperbolic structure on an orientable orbifold is contained in $\PSL(2,\R)$). If we consider such representations up to $G_0$-conjugacy, it may happen that there are two connected components of $\Hom(\piY,G_0)/G_0$ containing conjugacy classes of Fuchsian representations, but these are related by an interior automorphism of $G$. This happens for instance if $\fg=\sl(2,\R)$, in which case $G\simeq\PGL(2,\R)$ and $G_0\simeq\PSL(2,\R)$.
\end{remark}

\begin{remark}\label{embeddings_of_HC}
The morphism $\kappa:\PGL(2,\R) \lra \PGL(n,\R)$ has image contained in $\PSp^\pm(2m,\R)$ if $n=2m$, $\PO(m,m+1)$ if $n=2m+1$ and $\GG$ if $n = 7$ (see \cite[Chapter 6, \S 2]{GOV}). So, given an orbifold $Y$, we have maps
$\Hit(\piY,\PSp^\pm(2m,\R)) \lra \Hit(\piY,\PGL(2m,\R))$,
$\Hit(\piY,\PO(m,m+1)) \lra \Hit(\piY,\PGL(2m+1,\R))$
and $\Hit(\piY,\GG) \lra \Hit(\piY,\PO(3,4)) \lra \Hit(\piY,\PGL(7,\R))$. 
If $Y=X$ is a closed orientable surface, it is a consequence of Hitchin's parameterization \cite{Hitchin_Teich} recalled in Section \ref{section:parameterization_Hitchin} that these maps are injective. For the same reason, $\cT(X)\simeq\Hit(\piX,\PGL(2,\R))$ embeds into each $\Hit(\piX,G)$.
\end{remark}

\subsection{Restriction to subgroups of finite index}

Assume now that $Y \simeq [\Si\bs X]$ is a presentation of $Y$ 
in the sense of Definition \ref{pres_of_Y_def}. 
In particular, $\piX$ is a normal subgroup of finite index of $\piY$ and $\Si\simeq \piY/\piX$. The restriction of a representation gives a map $j:  \Rep(\piY,G) \ni \left[\rho\right] \lmt  \left[\rho|_{\piX}\right] \in  \Rep(\piX,G)$.
Recall that there is a canonical group homomorphism $\Si\lra \Out(\piX)$ and that $\Out(\piX)$ acts on the space $\Rep(\piX,G)$. We will denote by $\Fix_\Si(\Rep(\piX,G))$ the fixed locus of this action.

\begin{lemma}\label{map_to_fixed_locus}
The image of the map $j$ is contained in $\Fix_\Si(\Rep(\piX,G))$. 
\end{lemma}
\begin{proof} Take $\si\in\Si$ and choose a lift $\ga\in\piY$. If $\rho:\piY\lra G$ is a representation, then $\si\cdot[\rho|_{\piX}]$ is, by definition, the $G$-conjugacy class of the representation $\si\cdot\rho|_{\piX}: \piX \ni \delta \lmt \rho|_{\piX}(\ga^{-1}\delta\ga) \in G$. As $\rho|_{\piX}(\ga^{-1}\delta\ga) = \rho(\ga)^{-1} \rho|_{\piX}(\delta) \rho(\ga)$ with $\rho(\ga)\in G$, we have indeed that $\si\cdot\rho|_{\piX}$ lies in the $G$-conjugacy class of $\rho|_{\piX}$. 
\end{proof} 

Note that the formula $(\ga\cdot \rho)(\delta):=\rho(\ga)^{-1}\rho(\delta)\rho(\ga)$ indeed defines a left action of $\piY$ on $\Hom(\piY,G)$ because $(\ga_1\cdot(\ga_2\cdot\rho))(\delta) = (\ga_2\cdot\rho)(\ga_1)^{-1} (\ga_2\cdot\rho)(\delta) (\ga_2\cdot\rho)(\ga_1)$. In general, the map 
\begin{equation}\label{map_to_fixed_pt_set}
j:\Rep(\piY,G)\lra \Fix_\Si(\Rep(\piX,G))
\end{equation} 
defined by means of Lemma \ref{map_to_fixed_locus} is neither injective nor surjective. A crucial observation of the present paper is that if we restrict to Hitchin components, $j$ induces a bijective map. 

\begin{lemma}\label{rest_of_Hitchin_is_Hitchin}
If $\rho:\piY\lra G$ is a Hitchin representation and $Y'\lra Y$ is a finite orbifold cover,  then $\rho|_{\pi_1Y'}:\pi_1Y'\lra G$ is a Hitchin representation.
\end{lemma}
\begin{proof}
If $\rho:\piY\lra G$ is a Fuchsian representation, then, for every finite orbifold cover $Y'\lra Y$, the representation $\rho|_{\pi_1Y'}$ is also Fuchsian. As Hitchin components are connected, this implies the statement.
\end{proof}

\noindent Lemma \ref{rest_of_Hitchin_is_Hitchin} implies that $j(\Hit(\piY,G)) \subset \Hit(\piX,G)$. Moreover, the group $\Out(\piX)$ acts on $\Rep(\piX,G)$ preserving the Fuchsian locus, therefore it also preserves the Hitchin component. We denote the fixed locus of the induced $\Si$-action by $\Fix_{\Si}(\Hit(\piX,G))$. Hence we have a map $j:\Hit(\piY,G)\lra \Fix_{\Si}(\Hit(\piX,G))$. To prove that the map $j$ is injective, we will need the following lemma.
\begin{lemma}   \label{lemma:centralizer}
Let $\rho:\piY\lra G$ be a Hitchin representation. Then $\rho$ is $G_\C$-strongly irreducible, meaning that its restriction to every finite index subgroup is $G_\C$-irreducible. Moreover, $\rho$ has trivial centralizer in $G$ and in $G_\C$, i.e. if an element $g\in G_\C$ satisfies $g \rho(\gamma) = \rho(\gamma) g$ for every $\gamma \in \piY$, then $g$ is the identity. 
\end{lemma} 
Recall that for a (real or complex) reductive Lie group $H$, a representation is $H$-\emph{irreducible} if its image is not contained in a proper parabolic subgroup of $H$. When $G=\PGL(n,\R)$ or $\PGL(n,\C)$, this is equivalent to the well-known definition \cite[(1.3.1)]{Serre05}. As expected, being $G_\C$-irreducible implies being $G$-irreducible.
\begin{proof}[Proof of Lemma \ref{lemma:centralizer}]
Choose a presentation $Y \simeq [\Si\bs X]$, and consider $\rho|_{\piX}$.  
Hitchin proved in \cite[Section 5]{Hitchin_Teich} that the Higgs bundles in the Hitchin components are smooth points of the moduli space of $G_\C$-Higgs bundles, and hence these Higgs bundles are $G_\C$-stable and simple. By the non-Abelian Hodge correspondence, this implies that the representation $\rho|_{\piX}$ is $G_\C$-irreducible and has trivial centralizer in $G_\C$. The same properties therefore hold for $\rho$. Moreover, if $\Ga' <  \piY$ is a finite index subgroup, then $\Ga'$ is the orbifold fundamental group of a finite orbifold covering $Y'$, and by Lemma \ref{rest_of_Hitchin_is_Hitchin} we see that the restriction to $\Ga'$ is still  $G_\C$-irreducible.
\end{proof}

\begin{proposition}     \label{proposition:injectivity}
The map $j :\Hit(\piY,G)\lra \Fix_{\Si}(\Hit(\piX,G))$ is injective.
\end{proposition}
\begin{proof}
Let $\rho_1,\rho_2$ be two Hitchin representations of $\piY$ into $G$ such that $\rho_1|_{\piX}$ and $\rho_2|_{\piX}$ are $G$-conjugate. Replacing $\rho_2$ by $\Int_g\circ \rho_2$ for some $g\in G$ if necessary, we may assume that $\rho_1|_{\piX}$ and $\rho_2|_{\piX}$ are equal. The abstract situation (compare \cite[Lemma 3.1]{Long_Reid}) is then as follows: we have a normal subgroup $N\lhd\Ga$ and two representations $\rho_1,\rho_2:\Ga\lra G$ such that $\rho_1|_N=\rho_2|_N=:\rho$ has trivial centralizer in $G$. For a fixed $\ga\in\Ga$, consider the representation $N\lra G$ defined, for all $n\in N$, by $n\lmt \rho_1(\ga^{-1})\rho_2(\ga) \rho(n) \rho_2(\ga^{-1})\rho_1(\ga)$. This is equal to $\rho_1(\ga^{-1}) \rho_2(\ga n\ga^{-1}) \rho_1(\ga) = \rho_1(\ga^{-1}) \rho_1(\ga n\ga^{-1}) \rho_1(\ga) = \rho(n)$ because $\ga n\ga^{-1}\in N \lhd \Ga$. Hence $\rho_1(\ga^{-1})\rho_2(\ga)$ centralizes the Hitchin representation $\rho$, and by Lemma \ref{lemma:centralizer} it is the identity. Thus, $\rho_1(\ga)=\rho_2(\ga)$ for all $\ga\in\piY$.
\end{proof}

\begin{corollary}\label{embedding_of_Teichmuller_space} 
Let $\rho_1, \rho_2: \piY \lra G$ be two Fuchsian representations: $\rho_1=\kappa\circ h_1$ and $\rho_2=\kappa\circ h_2$ where $h_1, h_2:\piY\lra\PGL(2,\R)$ are discrete and faithful representations. If $\rho_1$ and $\rho_2$ are $G$-conjugate, then $h_1$ and $h_2$ are $\PGL(2,\R)$-conjugate. Equivalently, the map \eqref{Fuchsian_embedding} induces a bijection between the Teichm\"{u}ller space $\cT(Y)$ of the orbifold $Y$ and the Fuchsian locus of $\Rep(\piY,G)$.
\end{corollary}
\begin{proof}
Let $Y \simeq [\Si \bs X]$. Recall that $\cT(Y)\simeq\Hit(\piY,\PGL(2,\R))$, and similarly for $X$. The map $h\lmt\ h|_{\piX}$ then induces a commutative diagram $$\xymatrix{
\cT(Y) \ar[r] \ar[d] & \cT(X) \ar[d] \\
\Hit(\piY,G) \ar[r] & \Hit(\piX,G)
} 
$$ whose vertical arrows are induced by composition by the principal representation $\kappa:\PGL(2,\R)\lra G$ and whose horizontal arrows are injective, by Proposition \ref{proposition:injectivity} (as a matter of fact, we only need the injectivity of the top one). Since the vertical arrow $\cT(X)\lra \Hit(\piX,G)$ is injective (see \cite{Hitchin_Teich} and Remark \ref{embeddings_of_HC}), it follows that so is the vertical arrow $\cT(Y)\lra\Hit(\piY,G)$.
\end{proof}

\begin{theorem}\label{main_obs_about_Hit_comp_for_orbifolds}
Let $Y$ be a closed connected $2$-orbifold of negative Euler characteristic. Let $\fg$ be the split real form of a complex simple Lie algebra and let $G$ be the group of real points of $\Int(\fg\otimes\C)$. Given a presentation $Y \simeq [\Si \bs X]$, the map $\rho\lmt\rho|_{\piX}$ induces a homeomorphism 
$j:\Hit(\piY,G) \overset{\simeq}{\lra} \Fix_{\Si}(\Hit(\piX,G))$ 
between the Hitchin component of $\Rep(\piY,G)$ and the $\Si$-fixed locus in $\Hit(\piX,G)$.
\end{theorem}
The injectivity was proved in Proposition \ref{proposition:injectivity}. We postpone the proof of surjectivity to Section \ref{inv_Hitchin_reps}. 

\begin{corollary}\label{fixed_pts_sets_in_Hit_of_arbitrary_finite_Galois_cover}
Let $Y'\lra Y$ be a finite Galois cover of $Y$ and let $\Si':= \piY/\piY'$. Then the map $\rho\lmt\rho|_{\pi_1Y'}$ induces a homeomorphism $\Hit(\piY,G)\simeq \Fix_{\Si'}(\Hit(\pi_1Y',G))$.
\end{corollary}

\begin{proof}
Let $X$ be a finite Galois cover of $Y$ by a closed orientable surface. By pulling back this cover to $Y'$ if necessary, we can assume that $X$ is a (finite and Galois) cover of $Y'$. Then, by Theorem \ref{main_obs_about_Hit_comp_for_orbifolds}, one has
$$\Hit(\piY,G) \simeq \Fix_{\piY/\piX}(\Hit(\piX,G)) = \Fix_{\piY/\piY'} \left( \Fix_{\pi_1Y'/\piX}(\Hit(\piX,G)) \right)$$ which is homeomorphic to $\Fix_{\Si'}(\Hit(\piY',G))$, again by Theorem \ref{main_obs_about_Hit_comp_for_orbifolds}.
\end{proof}

\begin{corollary}\label{rep_whose_rest_is_Hitchin_is_itself_Hitchin}
Let $\rho:\piY\lra G$ be a representation and let $Y'\lra Y$ be a finite cover of $Y$, not necessarily Galois. Then $\rho$ is Hitchin if and only if $\rho|_{\pi_1Y'}$ is Hitchin.
\end{corollary}

\begin{proof}
The obvious direction of the corollary is given by Lemma \ref{rest_of_Hitchin_is_Hitchin}. Conversely, assume that $\rho:\piY\lra G$ satisfies that $\rho|_{\pi_1Y'}$ is Hitchin. Let $Y''$ be a finite Galois cover of $Y$ that covers $Y'$ (again, this may be obtained by pullback of a finite Galois cover of $Y$). By Lemma \ref{rest_of_Hitchin_is_Hitchin}, the representation $\rho|_{\pi_1Y''}$ is Hitchin.  And by \eqref{map_to_fixed_pt_set}, $\rho|_{\pi_1Y''}$ lies in the fixed-point set of $\pi_1Y/\pi_1Y''$ in $\Hit(\pi_1Y'',G)$. Therefore, Corollary \ref{fixed_pts_sets_in_Hit_of_arbitrary_finite_Galois_cover} shows that $\rho$ is Hitchin.
\end{proof}

\subsection{Properties of Hitchin representations for closed orbifolds}\label{section:properties}

We present in this section a series of properties satisfied by $\PGL(n,\R)$-Hitchin representations of fundamental groups of closed orbifolds that directly generalize known ones for fundamental groups of closed orientable surfaces (strong irreducibility, discreteness, faithfulness, hyperconvexity). The first property is a special case of Lemma \ref{lemma:centralizer}. The next two are simple consequences of the fact that $\piY$ contains the fundamental group of a closed orientable surface $X$ as a normal subgroup of finite index (Definition \ref{pres_of_Y_def}). For the remaining one, we apply Corollary \ref{rep_whose_rest_is_Hitchin_is_itself_Hitchin}. We shall assume that $G\simeq\PGL(n,\R)$ until the end of this section. We refer to \cite[Definition 2.10]{GW2} for the definition of Anosov representations.

\begin{remark}
By Remark \ref{embeddings_of_HC}, the results of this subsection apply to Hitchin representations in the groups $\PSp^{\pm}(2m,\R)$, $\PO(m,m+1)$ and $\GG$, since they are also Hitchin representations in $\PGL(n,\R)$. 
\end{remark}

\begin{proposition}\label{prop:Anosov}
Let $B$ be a Borel subgroup of $\PGL(n,\R)$. Then every Hitchin representation $\rho:\piY \lra \PGL(n,\R)$ is $B$-Anosov.
\end{proposition}
\begin{proof}
Choose a presentation $Y \simeq [\Si \bs X]$. Labourie \cite{Labourie} proved that $\rho|_{\piX}$ is $B$-Anosov. Now \cite[Corollary 3.4]{GW2} implies that $\rho$ is also $B$-Anosov since $\piX$ is a finite index subgroup of $\piY$.
\end{proof}

\begin{corollary} \label{Hitchin_rep_discrete}
Every Hitchin representation $\rho:\piY \lra \PGL(n,\R)$ has discrete image.
\end{corollary}
\begin{proof}
By \cite[Theorem 5.3]{GW2}, all Anosov representations have this property.
\end{proof}

\noindent Moreover, the theory of domains of discontinuity of Guichard and Wienhard \cite{GW2} and Kapovich, Leeb and Porti \cite{KLP} can be applied to Hitchin representations of orbifold groups (Section \ref{subsec:projective structures}).

\begin{proposition}  \label{Hitchin_rep_faithful}
Every Hitchin representation $\rho:\piY \lra \PGL(n,\R)$ is faithful.
\end{proposition}

\begin{proof}
Let $\rho$ be a Hitchin representation of $\piY$. By \cite[Proposition 3.4]{Labourie}, elements of $(\mathrm{Im}\,\rho|_{\piX}) \setminus \{1\}$ in $\PGL(n,\R)$ are diagonalizable with distinct, real eigenvalues. In particular, $\rho|_{\piX}$ is faithful. Consider now $\ga\in\piY$. Since $\piY/\piX$ is a finite group, there exists a minimal integer $q$ such that $\ga^q\in\piX$. If $\ga^q\neq 1_{\piX}$, then by Labourie's result, $\rho(\ga)^q=\rho(\ga^q)\neq 1$ in $\PGL(n,\R)$. In particular, $\rho(\ga)\neq 1$. If $\ga^q=1_{\piX}$, then (as $\rho(\ga)^q=\rho(\ga^q)=1$) the eigenvalues of $\rho(\ga)$, as an endomorphism of $\fg_\C$, are all $q$-th roots of unity. Since those form a finite (in particular, discrete) subset, the latter is invariant by continuous deformation of $\rho$. Let us then consider a Fuchsian representation $\rho':\piY \hookrightarrow \PGL(n,\R)$. By what we have just said, $\rho'(\ga)$ and $\rho(\ga)$ have the same eigenvalues. Since $\ga$ is not trivial in $\piY$, the element $\rho'(\ga)$ is not trivial in $\PGL(n,\R)$, so it has an eigenvalue that is not equal to $1$. Therefore $\rho(\ga)$ also has an eigenvalue that is not equal to $1$. In particular, $\rho(\ga)\neq 1$.
\end{proof}

\begin{remark}
Wienhard \cite{Wienhard_ICM} defines higher Teichm\"uller spaces as unions of connected components of
$\Rep(\pi_1 X,G)$ in which each representation is discrete and faithful. Here, $\pi_1 X$ is the fundamental group of a closed orientable surface.  If we  generalize that definition to orbifold groups, then Corollary \ref{Hitchin_rep_discrete} and Proposition \ref{Hitchin_rep_faithful} say that Hitchin components for orbifold groups are examples of higher Teichm\"uller spaces.
\end{remark}

\begin{proposition}   \label{prop:Hitchin rep loxodromic}
If $\rho:\piY\lra \PGL(n,\R)$ is a Hitchin representation, then for all $\ga$ of infinite order in $\piY$, the element $\rho(\ga)$ of $\PGL(n,\R)$ is purely loxodromic (i.e.\ diagonalizable with distinct real eigenvalues).
\end{proposition}

\begin{proof}
In the case of a closed orientable surface $X$, all non-trivial elements of $\piX$ are of infinite order and Labourie has shown in \cite{Labourie} that their image under a Hitchin representation is purely loxodromic. If $Y \simeq [\Si \bs X]$ with $\Si$ finite, then for all $\ga\in\piY$ there exists $q\geqslant 1$ such that $\ga^q\in\piX$ and if $\ga$ is of infinite order, then $\ga^q\neq 1$ in $\piX$. So $\rho(\ga)^q=\rho(\ga^q)$ is purely loxodromic. Therefore, so is $\rho(\ga)$.
\end{proof}

\noindent However, if $\ga$ is of finite order in $\piY$, then $\rho(\ga)\in G$ may have non-distinct eigenvalues, as we can already see from the case $G=\PGL(3,\R)$. For instance, if $Y$ has a cone point $x$ of order $2$, then a small loop around $x$ will map, under the holonomy representation of a hyperbolic structure on $Y$, to an element of $\PGL(2,\R)$, the rotation matrix of angle $\frac{\pi}{2}$. It is conjugate to $\mathrm{diag}(i,-i)$ in $\PGL(2,\C)$ and maps to $\mathrm{diag}(-1,1,-1)$ under $\kappa_\C$ (and also $\kappa$).

Finally, by applying Theorem \ref{main_obs_about_Hit_comp_for_orbifolds}, we can extend the Labourie-Guichard characterization of Hitchin representations into $G=\PGL(n,\R)$ as hyperconvex representations \cite{Labourie,Guichard_hyperconvex} to the orbifold case. Following Labourie \cite{Labourie}, a $\PGL(n,\R)$-representation $\rho$ of $\piY$ is called \emph{hyperconvex} if there exists a continuous map $\xi:\partial_\infty\piY \lra \RP^{n-1}=\mathbf{P}(\R^n)$ that is $\piY$-equivariant with respect to $\rho$ and hyperconvex in the sense that for all $n$-tuples of pairwise distinct points $(x_1,\,\ldots\,,x_n)$ in $\partial_\infty\piY\simeq\partial\H^2 \simeq S^1$, we have 
$\xi(x_1) + \dotsm + \xi(x_n) = \R^n$.

\begin{lemma} \emph{\cite{KB}} \label{lemma:canonical}
If $X\lra Y$ is a finite cover, then there is a canonical homeomorphism $\partial_\infty\piX \simeq \partial_\infty\piY$, which is $\piX$-equivariant with respect to the inclusion $\piX\hookrightarrow\piY$.
\end{lemma}

\begin{theorem}\label{Hitchin_iff_hyperconvex}
A representation of $\piY$ in $\PGL(n,\R)$ is Hitchin if and only if it is hyperconvex.
\end{theorem}

\begin{proof}
Let $Y \simeq\, [\Si\bs X]$ be a presentation of $Y$ (see Definition \ref{pres_of_Y_def}), and let $\rho$ be a Hitchin representation of $\piY$. Then $\rho|_{\piX}$ is Hitchin by Lemma \ref{rest_of_Hitchin_is_Hitchin}. By \cite[Theorem 1.4]{Labourie}, there exists a $\piX$-equivariant, hyperconvex curve $\xi:\partial_\infty\piX\lra \RP^{n-1}$. Given an element $\ga\in\piY$, let us consider the map $$\rho(\ga)\circ\xi\circ\ga^{-1}:(\partial_\infty\piY = \partial_\infty\piX) \lra \RP^{n-1},$$
where $\partial_\infty\piX$ is identified with $\partial_\infty\piY$ via the $\piX$-equivariant homeomorphism in Lemma \ref{lemma:canonical}. It is straightforward to check that this map is hyperconvex. Moreover, it is $\piX$-equivariant: if $\delta\in\piX$, we have, as $\piX$ is normal in $\piY$, that $(\rho(\ga) \circ \xi \circ \ga^{-1}) \circ \delta = \rho(\ga)\circ (\rho(\ga^{-1}\delta\ga)\circ\xi\circ\ga^{-1}) = \rho(\delta) \circ (\rho(\ga)\circ\xi\circ\ga^{-1})$. So, by uniqueness of such a map \cite[Proposition 16]{Guichard_hyperconvex}, $\rho(\ga)\circ\xi\circ\ga^{-1} = \xi$. As this holds for all $\ga\in\piY$, we have that $\xi$ is $\piY$-equivariant. 
Conversely, assume that $\rho$ is hyperconvex and let $\xi:\partial_\infty\piY\lra\RP^{n-1}$ be the associated $\piY$-equivariant hyperconvex curve. Since $\piX\hookrightarrow\piY$, the curve $\xi$ is also $\piX$-equivariant. So, by \cite[Théorème 1]{Guichard_hyperconvex}, $\rho|_{\piX}$ is a Hitchin representation. It then follows from Corollary \ref{rep_whose_rest_is_Hitchin_is_itself_Hitchin} that $\rho$ is a Hitchin representation of $\piY$.
\end{proof}

\begin{remark}\label{unicity_of_the_hyperconcex_curve} 
In the course of the proof, we have seen that if $\rho:\piY\lra \PGL(n,\R)$ is a hyperconvex representation of $\piY$, then the $\piY$-equivariant hyperconvex curve $\xi:\partial_\infty\piY\lra \RP^{n-1}$ is unique.
\end{remark}

\subsection{Orbifolds with boundary}\label{boundary_case}

We refer to \cite{CG} for background on orbifolds with boundary: each point of a smooth orbifold with boundary admits an open neighborhood with a presentation of the form $[\Ga\bs U]$ where $U$ is an open subspace of a closed half-space and $\Ga$ is a finite group, acting faithfully on $U$ by diffeomorphisms. In particular, the boundary of an $n$-dimensional orbifold with boundary is an $(n-1)$-dimensional closed orbifold. In dimension $2$, the boundary components of a compact orbifold with boundary are therefore either circles (with trivial orbifold structure) or segments with endpoints that are mirror points. The latter are called \textit{full 1-orbifolds} in \cite{CG} (see Figure \ref{fig:boundary}). The Euler characteristic of an orbifold with boundary $Y$ is given by the formula: $$\chi(Y) = \chi(|Y|) - \sum_{i=1}^k \left(1-\frac{1}{m_i}\right) - \frac{1}{2} \sum_{j=1}^\ell \left(1-\frac{1}{n_j}\right) -\frac{1}{2}b,$$ where $k$ is the number of cone points, $\ell$ the number of corner reflectors, and $b$ is the number of full $1$-orbifold boundary components of $Y$ (\cite[p.1029]{CG}). 

\begin{figure}[ht]
\centering
\includegraphics[width=12.0cm]{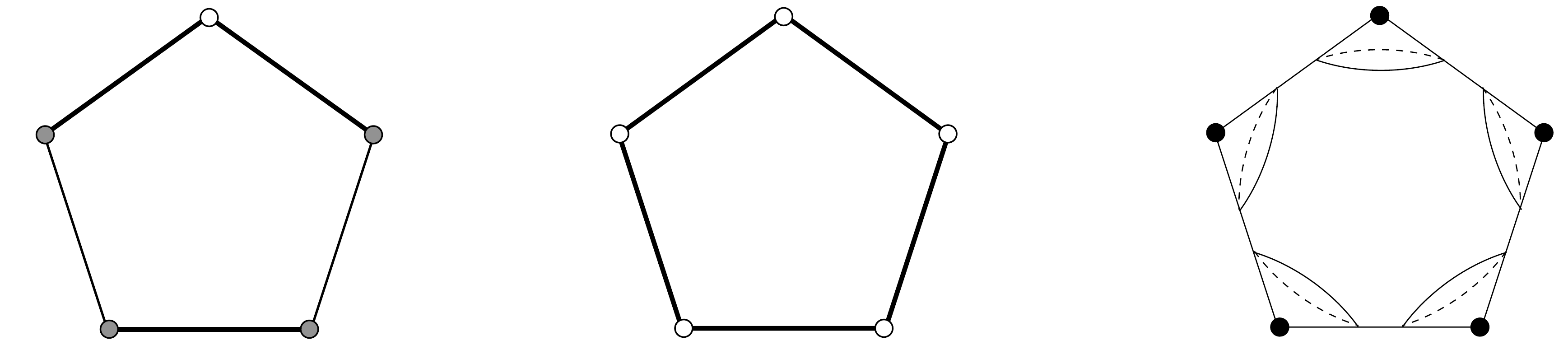}
\small
\put (-300, -10){$Y$}
\put (-180, -10){$m Y$}
\put (-47, -10){$d Y$}
\caption{In $Y$ and $mY$, the points lying on a \emph{bold} side of the polygon (excluding vertices) are mirror points, and the white dots are corner reflectors of order $2$. In $dY$, the black dots are cone points of order $2$. The leftmost orbifold $Y$ has two boundary components, which are full $1$-orbifolds with gray extremal points. The middle orbifold $mY$ (resp.\ rightmost orbifold $d Y$) is a \emph{closed} orbifold with five corner reflectors (resp.\ cone points); $dY$ is the orientation double cover of $mY$. The underlying topological space of $Y$ and $mY$ is a disk.}\label{fig:boundary}
\end{figure}

To introduce a notion of Hitchin component for orbifolds with boundary, we need to impose an extra condition: the Teichmüller space of an orbifold with boundary, for instance, is defined as the deformation space of hyperbolic structures with totally geodesic boundary. Holonomy representations of such hyperbolic structures are exactly those representations $\rho:\piY\lra\PGL(2,\R)$ that are discrete, faithful and convex cocompact. Given a $2$-orbifold with boundary $Y$, we can construct a closed orbifold associated to $Y$, denoted by $mY$ and obtained by decreeing that all boundary points of $Y$ (lying in both circles and full $1$-orbifolds) are mirror points. We shall call $mY$ the  \emph{mirror} of $Y$ (see Figures \ref{fig:boundary} and \ref{fig:mirror}).

It has the same $k$ cone points as $Y$, the same $\ell$ corner reflectors, plus an extra $2b$ corner reflectors, each  one of order $2$, corresponding to extremal points of boundary full $1$-orbifolds of $Y$. In particular, $\chi(mY)=\chi(Y)$. By definition of the Teichmüller space of $Y$, one has $\cT(Y)\simeq\cT(mY)\simeq\Hit(\pi_1(mY),\PGL(2,\R))$ and it follows from the formula for a closed orbifold \eqref{Thurston_formula_for_PGL2} applied to $mY$ (see \cite[p.1094]{CG}) that: $$\dim\cT(Y)=-3\chi(|Y|)+2k+\ell+2b= \dim\cT(mY).$$ Below we denote by $dY$ the orientation double cover of $mY$: it has $2k+\ell+2b$ cone points and its Euler characteristic is equal to $2\chi(mY)$. The Teichmüller space of $mY$ can be identified with the space of hyperbolic structures on $dY$ that are invariant under the canonical involution of $dY$.

\begin{figure}[ht]
\centering
\includegraphics[width=12.0cm]{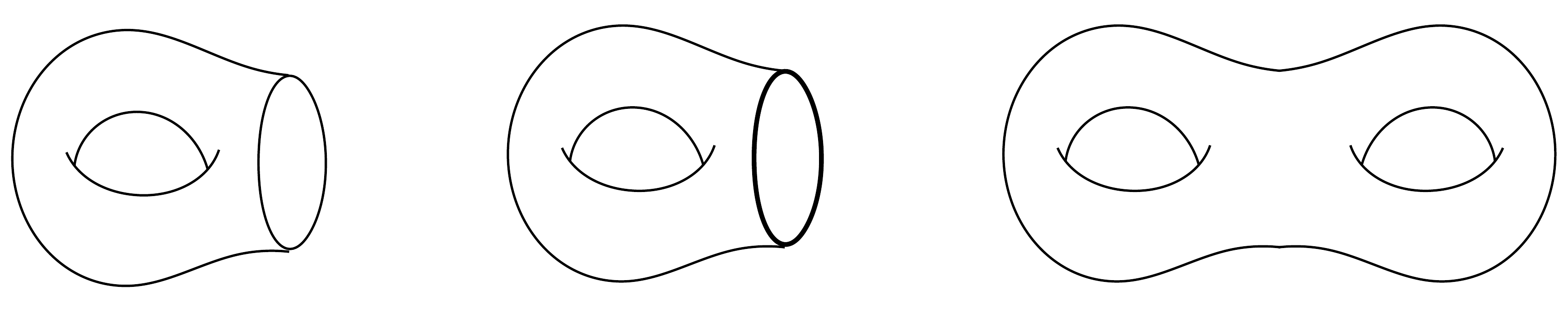}
\small
\put (-310, -8){$Y$}
\put (-200, -8){$m Y$}
\put (-67, -8){$d Y$}
\caption{The leftmost orbifold $Y$ is a torus with boundary (it has \emph{trivial} orbifold structure), the middle orbifold $mY$ is a non-orientable \emph{closed} orbifold (points lying on the circle represented in bold are mirror points), and the rightmost orbifold $dY$ is a surface of genus $2$ (with trivial orbifold structure); $dY$ is the orientation double cover of $mY$. The underlying topological space of $Y$ and $mY$ is a torus minus an open disk.}\label{fig:mirror}
\end{figure}

Let now $\fg$ be the split real form of a simple complex Lie algebra $\fg_\C$, with corresponding real structure $\tau$, and set $G=\Int(\fgC)^\tau$. In order to define Hitchin representations for the fundamental group of an orbifold with boundary $Y$, note first that $\piY$ is a subgroup of $\pi_1(mY)$. 

\begin{remark}
The index $[\pi_1(mY):\piY]$ is infinite. Indeed, we can choose a hyperbolic structure on $m Y$ and denote its holonomy representation by $\rho : \pi_1(m Y) \longrightarrow \PGL(2,\R)$. It induces the hyperbolic structure on $Y$ with geodesic boundary and $\rho(\pi_1 Y ) \bs \mathbb{H}^2$ is the complete hyperbolic orbifold made up of $Y$ with funnels attached along the boundary components. We have that the hyperbolic orbifold $\rho(\pi_1 Y) \bs \mathbb{H}^2$ has infinite area but $\rho(\pi_1(m Y)) \bs \mathbb{H}^2$ has finite area. Hence, the group $\piY$ is of infinite index in $\pi_1(mY)$.
\end{remark}

Let $\kappa:\PGL(2,\R)\lra G$ be the homomorphism induced by the choice of a principal $3$-dimensional subalgebra $\mathfrak{sl}(2,\R) \subset \fg$. As in Definition \ref{Fuchsian_rep}, a representation $\rho:\piY\lra G$ is called \emph{Fuchsian} if it lifts to a holonomy representation of hyperbolic structure (with totally geodesic boundary) on $Y$, i.e.\ if $\rho$ extends to a Fuchsian representation $\ov{\rho}:\pi_1(mY)\lra G$. Let us now introduce the map 
$$
\begin{array}{rcl}
\Phi : \Rep(\pi_1(mY),G) & \longrightarrow & \Rep(\piY,G)\\
\left[\chi\right] & \longmapsto &  \left[\chi|_{\piY}\right]
\end{array}$$ This map induces a homeomorphism between the Fuchsian locus of $\pi_1(mY)$ in $\Rep(\pi_1(mY),G)$ and the Fuchsian locus of $\piY$ in $\Rep(\piY,G)$, which motivates the following definition.

\begin{definition}\label{Hitchin_rep_bdry_case}
Let $Y$ be a compact connected $2$-orbifold with boundary. We define the \emph{Hitchin component} $\Hit(\piY,G)$ to be $\Phi(\Hit(\pi_1(mY),G))$, meaning the (connected) subspace of $\Rep(\piY,G)$ consisting of (conjugacy classes of) $G$-representations of $\piY$ that extend to a Hitchin representation of $\pi_1(mY)$ in $G$. A \emph{Hitchin representation} of $\piY$ is a representation $\rho:\piY\lra G$ whose conjugacy class lies in $\Hit(\piY,G)$.
\end{definition}

\noindent Note that the Fuchsian locus of $\piY$ in $\Rep(\piY,G)$ is indeed contained in $\Hit(\piY;G)$.

\begin{proposition}\label{Hitchin_comp_for_orbifold_with_boundary}
Let $Y$ be a compact connected $2$-orbifold with boundary and let $mY$ be the associated closed orbifold with mirror boundary. Then the map $\Phi:\chi \lmt \chi|_{\piY}$ induces a homeomorphism 
$$\Hit(\pi_1(mY),G) \overset{\simeq}{\lra} \Hit(\piY,G).$$
\end{proposition}

\begin{proof}
The map $\Phi : \Hit(\pi_1(mY),G) \rightarrow \Hit(\piY,G)$ given by $\Phi ([\chi]) = [\chi|_{\piY}]$ is continuous, and by definition of the Hitchin component of $\piY$, it is surjective. Thus, it remains to show that $\Phi$ is injective and that its inverse is also continuous. 

Each boundary component $b$ of $Y$ is either a circle or a full $1$-orbifold, and the group $\pi_1 b$ may identify with a subgroup of $\piY$. If $b$ is a circle, then $\pi_1 b$ is an infinite cyclic subgroup $C$ of $\piY$. And, if $b$ is a full $1$-orbifold, the group $\pi_1 b$ is the infinite dihedral group $ C \rtimes (\sfrac{\Z}{2\Z})$, where $C$ is a cyclic subgroup of $\pi_1 b$. In either case, we denote by $\gamma_b$ a generator of $C$. Note that there exists an order $2$ element $\delta_b \in \pi_1(mY)$ corresponding to $b$ such that $\delta_b$ commutes with $\gamma_b$. 

Suppose now that $\rho = \ov{\rho}|_{\piY}$ for some $[\ov{\rho}] \in \Hit(\pi_1(mY),G)$. By definition, there exists a continuous path $(\, \ov{\rho}_{t} \,)_{t \in [0,1]} $ in $\Hom(\pi_1(mY), G)$ beginning at a Fuchsian representation $\ov{\rho}_0$ and ending at $\ov{\rho}_1 = \ov{\rho}$ such that the $G$-conjugacy class of $\ov{\rho}_t$ lies in $\Hit(\pi_1(mY),G)$. 

If $mY \simeq [\Si' \bs X']$ is a presentation of $mY$, where $X'$ is a closed orientable surface and $\Si'$ is a finite group, then there exists $q \geq 1$ such that $\gamma_b^q \in \piX'$. By \cite[Theorem 1.9]{FG}, the element $\rho_t(\gamma_b^q) = \ov{\rho}_t(\gamma_b^q)$ is positive hyperbolic, so it is semisimple and regular. Since $\ov{\rho}_t(\delta_b)$ commutes with $\rho_t(\gamma_b)$ (and also $\rho_t(\gamma_b^q)$), it is also semisimple. It then follows from the fact that $\delta_b$ is of order $2$ that the set of candidates for $\ov{\rho}_t(\delta_b)$ is finite, in particular, discrete. By the continuity of $\ov{\rho}_t$, the element $\ov{\rho}_t(\delta_b)$ is conjugate to $\ov{\rho}_0(\delta_b)$ and hence to $\kappa(J_2) \in G$ with $J_2 = \mathrm{diag}(1,-1) \in \PGL(2,\R)$. Thus, the element $\ov{\rho}_t(\delta_b)$ is the unique order $2$ element in $G$ that commutes with $\rho(\gamma_b)$ and that is conjugate to $\kappa(J_2)$, and hence $\ov{\rho}$ is uniquely determined by $\rho$ and the map $\rho \lmt \ov{\rho}(\delta_b)$ is continuous. It implies that the map $\Phi$ is injective and its inverse is continuous.
\end{proof}

Note that when $G=\PGL(3,\R)$, the Hitchin components for the orbifold with boundary $Y$ we have defined coincides with the deformation space of convex real projective structures with principal geodesic boundary on $Y$ described in \cite{CG}.  

\begin{remark}\label{about_the_Labourie_McShane_component}
In \cite{Labourie_McShane}, Labourie and McShane introduced a notion of Hitchin component for the $\PGL(n,\R)$-representation space of the fundamental group of a compact orientable surface with boundary $S$. The boundary condition that they impose on Fuchsian representations is that a simple loop around a boundary component should go to a purely loxodromic element of $\PGL(n,\R)$, and the Hitchin component in their sense, which we denote by $\mathscr{H}_S$, is then the connected component of this Fuchsian locus inside the subspace of $\Hom(\pi_1S,\PGL(n,\R))/\PGL(n,\R)$ consisting of all representations satisfying that boundary condition, thus generalizing the classical Teichm\"uller space of hyperbolic structures with totally geodesic boundary on $S$. They show in \cite[Theorem 9.2.2.2]{Labourie_McShane} that $\rho:\pi_1S\lra \PGL(n,\R)$ is a Hitchin representation in their sense if and only if it extends to a Hitchin representation $\widehat{\rho}:\pi_1(dS)\lra \PGL(n,\R)$, where $dS$ is the doubled surface, such that $\widehat{\rho}$ is $\sfrac{\Z}{2\Z}$-equivariant with respect to the natural involution of $dS$ and the involution of $\PGL(n,\R)$ given by conjugation by $J_n:=\mathrm{diag}(1,-1,1,-1,\,\ldots\, )\in\PGL(n,\R)$. Equivalently, $\widehat{\rho}$ is a representation of $\pi_1(dS) \rtimes (\sfrac{\Z}{2\Z})$ in $\PGL(n,\R)$. Since $\pi_1(dS)\rtimes (\sfrac{\Z}{2\Z})$ is isomorphic to the orbifold fundamental group of the orbifold $mS$ with underlying space $S$ obtained by decreeing that all boundary points of $S$ are mirror points, the Hitchin component $\mathscr{H}_S$ of $S$ in the sense of Labourie and McShane is indeed homeomorphic to the Hitchin component of $mS$ in the sense of Definition \ref{Hit_rep_def}, thus to the Hitchin component of $S$ in the sense of Definition \ref{Hitchin_rep_bdry_case}. In particular, Theorem \ref{theorem:dimension} and Corollary \ref{cases_of_validity_of_Hitchin_s_formula} will show that $\mathscr{H}_S$ is homeomorphic to an open ball of dimension $-\chi(S) (n^2-1)=-\chi (|mS|) (n^2-1)$.
\end{remark}

Finally, we prove that Theorem \ref{theorem:intro_A} indeed holds for orbifolds with boundary.

\begin{theorem}
Let $Y$ be a compact connected $2$-orbifold with boundary of negative Euler characteristic. Then a Hitchin representation $\rho : \piY \lra \PGL(n,\R)$ is $B$-Anosov, discrete, faithful and strongly irreducible. Moreover, for all $\gamma$ of infinite order in $\piY$, the element $\rho(\gamma)$ is diagonalizable with distinct real eigenvalues. 
\end{theorem}

\begin{proof}
By Definition \ref{Hitchin_comp_for_orbifold_with_boundary}, a Hitchin representation $\rho:\piY\lra\PGL(n,\R)$ extends to a Hitchin representation $\ov{\rho}:\pi_1(mY)\lra\PGL(n,\R)$. By the results of Section \ref{section:properties} (Theorem \ref{theorem:intro_A} for closed orbifolds), $\ov{\rho}$ is $B$-Anosov, strongly irreducible, discrete and faithful. Moreover, it sends elements of infinite order in $\piY$ to purely loxodromic elements of $\PGL(n,\R)$. So these last three properties also hold for $\rho=\ov{\rho}|_{\piY}$ and there only remains to prove that $\rho$ is $B$-Anosov and strongly irreducible. 

Since $\ov{\rho}$ is $B$-Anosov and $\piY\hookrightarrow \pi_1(mY)$ is a quasi-isometric embedding with respect to the word metric on each group, we have that $\rho=\ov{\rho}|_{\piY}$ is $B$-Anosov (\cite[Theorem 1.3]{GGKW_Anosov_actions}). As a consequence of the $B$-Anosov property, we have that $\rho$ is proximal relative to full flags (\cite[Lemma 3.1]{GW2}). But since $\rho=\ov{\rho}|_{\piY}$ with $\ov{\rho}$ Hitchin, \cite[Lemma 5.12]{GW2} shows that $\rho$ is strongly irreducible (the assumptions of Lemma 5.12 are satisfied because of \cite{Labourie} and its generalization to closed orbifolds in Section \ref{section:properties}).
\end{proof}

\section{Hitchin's equations in an equivariant setting}\label{Hitchin_s_equations}

In this section, we give a short presentation of the results of equivariant non-Abelian Hodge theory that we need for our purposes, using previous work of Simpson \cite{Simpson_JAMS,Simpson_local_systems}, Ho, Wilkin and Wu \cite{HWW} and Garc\'ia-Prada and Wilkin \cite{GPW}. Since we are only interested in certain particular groups of adjoint type in this paper, we can afford to work with Lie algebra bundles.

\subsection{From orbifold representations to equivariant flat bundles} 

Let us fix a presentation $Y\simeq [\Si\bs X]$ as in Definition \ref{pres_of_Y_def}. In particular, there is a short exact sequence $1\lra\piX\lra \piY\lra\Si\lra 1$ and the universal covers of $X$ and $Y$ are $\piX$-equivariantly isomorphic: $\Xt \simeq \Yt$. It is well-known that, if $G$ is a Lie group of adjoint type with Lie algebra $\fg$, and $\rho:\piX\lra G$ is a representation of $\piX$ in $G$, then there is, associated to it, a flat Lie algebra $G$-bundle $\cE_\rho := \piX\bs (\Xt\times\fg)$ on $X$. In Proposition \ref{Sigma_equivariant_structure_from_rep} below, we recall that, if $\rho: \piX \lra G$ is the restriction to $\piX$ of a representation of $\piY$ into $G$, then the action of $\Sigma$ on $X$ lifts to $\cE_\rho$, giving it a structure of $\Si$-equivariant bundle in the following sense.

\begin{definition}[Equivariant bundle]\label{equiv_bdle_smooth_case}
A $\Si$-equivariant Lie algebra $G$-bundle over $(X,\Si)$ is a pair $(E,\tau)$ consisting of a smooth Lie algebra $G$-bundle $E$ and a family $\tau=(\tau_\si)_{\si\in\Si}$ of bundle homomorphisms 
\begin{equation}\label{lift_of_the_action}
\xymatrix{
E \ar^{\tau_{\si}}[r] \ar[d] & E \ar[d]\\
X \ar^{\si}[r] & X
}\end{equation}
satisfying $\tau_{1_\Si} = \id_E$ and, for all $\si_1,\si_2$ in $\Si$, $\tau_{\si_1\si_2} = \tau_{\si_1}\tau_{\si_2}$.
A homomorphism of $\Si$-equivariant bundles over $X$ is a bundle homomorphism (over $\id_X$) that commutes to the $\Si$-equivariant structures. A $\Si$-sub-bundle of $(E,\tau)$ is a sub-bundle $F\subset E$ such that, for all $\si\in \Si$, $\tau_\si(F)\subset F$. In particular, $(F,\tau|_F)$ is itself a $\Si$-equivariant bundle on $X$.
\end{definition}

When we say Lie algebra $G$-bundle, we mean a locally trivial $G$-bundle whose fibers are modeled on a Lie algebra  equipped with an effective action of $G$ by Lie algebra automorphisms. The most important case for us is when $G$ is of adjoint type and $\fg = \mathrm{Lie}(G)$. Homomorphisms of such bundles are understood to be Lie algebra homomorphisms fiberwise. A definition similar to Definition \ref{equiv_bdle_smooth_case} of course holds for usual vector bundles, as well as for principal bundles. If $(E,\tau)$ is a $\Si$-equivariant bundle on $X$, there are canonical isomorphisms $\phi_\si: E\overset{\simeq}{\lra} \si^* E$, satisfying $\phi_{1_\Si} = \id_E$ and $\phi_{\si_1\si_2} = (\si_2^* \phi_{\si_1}) \phi_{\si_2}$ for all $\si_1,\si_2$ in $\Si$. Conversely, such a family $(\phi_\si)_{\si\in\Si}$ defines a $\Si$-equivariant structure $\tau$ on $E$, the relation between the two notions being given by 
$\tau_\si=\sit^{E} \circ\phi_\si$, where $\sit^{E}$ is the canonical map $\si^*E\lra E$ over $\si:X\lra X$, satisfying $\widetilde{\si_1\si_2}^{E}=\widetilde{\si_1}^{E}\,\widetilde{\si_2}^{\si_1^* E}$ and $\si_2^*\phi_{\si_1} = (\widetilde{\si_2}^{\si_1^* E})^{-1}\phi_{\si_1}\widetilde{\si_2}^{E}$. In what follows, given a $\Si$-equivariant bundle $(E,\tau)$, we will always identify $E$ with $\si^*E$ using $\phi_\si$. In particular, there is an induced action of $\Si$ on the space of $G$-connections on $E$, which we denote by $\cA_E$. 
 This action is defined as follows: if $\nabla$ is a $G$-connection on $E$ and $\si\in\Si$, then $\nabla^\si:=\si^*\nabla$ is a connection on $\si^*E$, which has been canonically identified with $E$ via $\phi_{\si}$. This sets up a right action of $\Si$ on $\cA_E$, which may, equivalently, be defined by noting that $\Si$ acts on $\Om^k(X;E)=\Ga(\Lambda^k T^*X\otimes E)$, by $\si\cdot\omega := (\si\otimes \tau_\si) \circ \omega \circ \si^{-1}$ (see \eqref{action_on_Higgs_fields}), and setting $\nabla^\si:=\si^{-1}\nabla \si$ (see Proposition \ref{action_on_Chern_connections}). The group $\Si$ also acts on the gauge group $\cG_E$ of $E$ via $u^\si := \si^* u$ (or, equivalently, $u^\si=\tau_\si^{-1} \circ u \circ \tau_{\si}$). The $\Si$-action on $\cA_E$ is then compatible with the gauge action on that space, in the sense that $(u^{-1}\nabla u)^\si = (u^\si)^{-1}\nabla^\si u^\si.$ In particular, $\Fix_\Si\,(\cG_E)$ acts on $\Fix_\Si(\cA_E)$. We also observe that $F_{\nabla^\si}=\si^*F_\nabla=:F_\nabla^\si$ in $\Om^2(X;E)$. In particular, $\Si$ acts on the set $F^{-1}(0)$ of flat connections on $E$. It remains to see that, if $\rho:\piY\lra G$ is a group homomorphism, then there is indeed a canonical $\Si$-equivariant structure on the Lie algebra bundle $\cE_\rho = \piX\bs(\Xt\times\fg)$ over $X$.

\begin{proposition}\label{Sigma_equivariant_structure_from_rep}
Given $\ga\in\piY$, the map \begin{equation}\label{Si_equiv_structure_on_flat_bundle}
\widetilde{\tau}_\ga:  \Xt\times \fg \ni  (\eta,v)  \lmt  (\ga\cdot \eta,\rho(\ga)\cdot v) \in \Xt\times \fg
\end{equation} 
descends to a map $\tau_\si$ on $\cE_\rho$ that only depends on the class $\si$ of $\ga$ in $\Si = \piY/\piX$. The collection $\tau:=(\tau_\si)_{\si\in\Si}$ of these maps defines a $\Si$-equivariant structure on $\cE_\rho$. Moreover, the canonical flat connection on $\cE_\rho$, induced by the trivial connection on $\Xt\times\fg$, is $\Si$-invariant with respect to the action of $\Si$ on the space of connections on $\cE_\rho$ associated to $\tau$.
\end{proposition}
\begin{proof}
Let us check that $\widetilde{\tau}_\ga$ descends to $\cE_\rho$: if $\delta\in\piX$, then $\widetilde{\tau}_\ga \widetilde{\tau}_\delta = \widetilde{\tau}_{\delta'} \widetilde{\tau}_\ga $ with $\delta' := \ga\delta\ga^{-1}\in\piX$. A similar computation shows that the induced transformation of $\cE_\rho$ indeed only depends on the class of $\ga$ modulo $\piX$. The connection on $\cE_\rho$ induced by the trivial connection on $\mathcal{V} := \Xt\times\fg$ is $\Si$-invariant because the trivial connection on $\mathcal{V}$ is $\piY$-invariant with respect to the $\piY$-equivariant structure $(\widetilde{\tau}_\ga)_{\ga\in\piY}$ on $\mathcal{V}$. 
\end{proof}

Therefore, given a presentation $Y\simeq[\Si\bs X]$, we have set up a map 
\begin{equation}\label{from_orbi_reps_to_flat_equiv_bundles}
\Hom(\piY,G)/G \lra \{ \Si\textrm{-equivariant}\ \textrm{flat}\ G\textrm{-bundles\ on}\ X \} /\textrm{isomorphism},
\end{equation} where a $\Si$-equivariant flat bundle is defined as follows.

\begin{definition}[Equivariant flat bundle]\label{equiv_flat_bdle_def}
A $\Si$-\textit{equivariant flat bundle} on $(X,\Si)$ is a triple $(E,\nabla,\tau)$ where $(E,\nabla)$ is a flat bundle on $X$ and $\tau$ is a $\Si$-equivariant structure on $E$ that leaves the connection $\nabla$ invariant. A homomorphism of $\Si$-equivariant flat bundles is a homomorphism of flat bundles that commutes with the $\Si$-equivariant structures.
\end{definition}

An inverse map to \eqref{from_orbi_reps_to_flat_equiv_bundles} is provided by the holonomy of $\Si$-invariant flat connections. More precisely, as in \cite{Ho-Liu,Sch_JDG}, we will have one such holonomy map for each isomorphism class of $\Si$-equivariant bundles. To prove this, we first need a description of $\piY$ in terms of paths in $X$. Let us choose a point $x\in X$ and consider the set $P_x$, consisting of pairs $([c],\si)$ where $\si\in \Si$ and $[c]$ is the homotopy class of a path $c:[0,1]\lra X$ satisfying $c (0)=x$ and $c(1)= \si(x)$, equipped with the group law $([c_1],\si_1)([c_2],\si_2) = ([c_1(\si_1\circ c_2)],\si_1\si_2)$. Our convention for concatenating paths is from left to right, so the above group law is well-defined. Note that, if $\Si$ has fixed points in $X$ and $x\in \Fix_{\Si}(X)$, then $P_x\simeq\piX\rtimes\Si$ for the natural left action of $\Si$ on $\piX$. In what follows, we denote by $\xt$ the base point of $\Xt$ corresponding to the homotopy class of the constant path at $x$ in $X$. Recall that $\piY=\Aut_Y(\Yt)$.

\begin{lemma}
The map $\piY \lra P_x$ sending $\ga\in\piY$ to $([c_\ga],\si)$, where $c_\ga:[0,1]\lra X$ is the projection to $X$ of an arbitrary path from $\xt$ to $\ga(\xt)$ in $\Xt$ and $\si$ is the class of $\ga$ in $\Si= \piY/\piX$, is a group isomorphism.
\end{lemma}

\begin{proof}
Note that $[c_\ga]$ is well-defined because $\Xt$ is simply connected. Moreover, the map $\ga\lmt([c_\ga],\si)$ is a group homomorphism. To see that it is injective, assume that $([c_\ga],\si)= (\xt,1_\Si)$. As $\si=1_\Si$, we have that $\ga\in\piX$. And since $c_\ga$ is homotopic to the constant path at $x$ in $X$, the path it lifts to in $\Xt$ goes from $\xt$ to $\xt$. In particular, $\ga(\xt)=\xt$, and since $\ga\in\piX$ and $\piX$ acts freely on $\Xt$, this implies that $\ga=1_{\piX}=1_{\piY}$. To see that our map $\piY\lra P_x$ is also surjective, take $([c],\si)$ in $P_x$ and let us denote by $q$ the universal covering map $q:\Xt\lra X$. The path $c$ goes from $x$ to $\si(x)$ in $X$ and it lifts to a path from $\xt$ to a point $\eta$ in the fiber of $q$ over $\si(x)$. Since $\si\circ q: \Xt\lra X$ is also a universal covering map, there exists a unique continuous map $\ga:\Xt\lra\Xt$ such that $q\circ \ga=\si\circ q$ and $\ga(\xt)=\eta$. Since $\ga:\Xt\lra\Xt$ lies over $\si:X\lra X$, we have that $\ga$ maps to $\si$ in $\piY/\piX$ and, since $\si:X\lra X$ lies over $\id_Y$, we also have that $\ga\in\Aut_Y(\Yt)=\piY$. Finally, by definition of $[c_\ga]$, we have that $[c_\ga]=[c]$.
\end{proof}

\noindent For all $([c],\si)\in P_x$, we consider the map $\tau_{\si}^{-1}\circ T^\nabla_c: E_x \lra E_x$ obtained by composing the parallel transport operator along the path $c$ with respect to $\nabla$ by the bundle map $\tau_\si^{-1}$. Because of our convention on concatenation of paths, this will be a group anti-homomorphism from $P_x$ to $\Aut(E_x)$, as we now show.

\begin{theorem}\label{hol_of_inv_conn}
Given a presentation $Y\simeq[\Si\bs X]$ and a $\Si$-equivariant flat $G$-bundle $(E,\nabla,\tau)$ over $X$, there is a group homomorphism $\widetilde{\rho_\nabla}: \piY\simeq P_x \lra G$ obtained by taking $([c],\si)\in P_x$ to $\tau_\si^{-1}\circ T^\nabla_c\in\Aut(E_x)$. Moreover, the restriction of $\widetilde{\rho_\nabla}$ to $\piX < \piY$ is the holonomy representation $\rho_\nabla:\piX\lra G$, and two gauge-equivalent connections induce conjugate representations. We therefore obtain a continuous map $\Fix_\Si\,(F^{-1}(0)) / \Fix_\Si\,(\cG_E) \lra \Hom(\piY,G)/G$ from gauge orbits of $\Si$-invariant, flat connections on $(E,\tau)$ to $\Rep(\piY,G)$ which, composed with the map \eqref{from_orbi_reps_to_flat_equiv_bundles}, is the identity map of $\Fix_\Si\,(F^{-1}(0)) / \Fix_\Si\,(\cG_E)$.
\end{theorem}

\begin{proof}
The statement follows from the definition of $P_x$ and the properties of parallel transport operators, namely that, if $T^\nabla_c$ is the parallel transport operator along the path $c$ with respect to a connection $\nabla$, there is a commutative diagram $$\xymatrix{
E_{c(0)} \ar^{T^{\nabla^\si}_c}[r] \ar^{\tau_\si}[d] & E_{c(1)} \ar^{\tau_\si}[d] \\
E_{\si(c(0))} \ar^{T^\nabla_{\si\circ c}}[r] & E_{\si(c(1))}
}$$ where, as earlier, $\nabla^\si=\si^*\nabla$ (by definition). For a detailed proof of the above, we refer for instance to \cite[Section 4.1]{Sch_JDG}. In particular, if $\nabla^\si=\nabla$, then $T^\nabla_{\si\circ c} = \tau_\si T^\nabla_c \tau_\si^{-1}$, which readily implies that the map $([c],\si)\lmt \tau_\si^{-1}\circ T^\nabla_c$ is a group anti-homomorphism from $\piY$ to $\Aut(E_x)$ (since $T^\nabla_{c_1(\si_1\circ c_2)} = T^\nabla_{\si_1\circ c_2} \circ T^\nabla_{c_1}$, due to our convention on concatenation of paths). The rest of the theorem is proved as in the case $\Si=\{1\}$.
\end{proof}

\begin{corollary}\label{rep_of_piY_and_flat_bdles_on_X}
Under the assumptions of Theorem \ref{hol_of_inv_conn}, there is a homeomorphism $$\Hom(\piY,G)/G \simeq \bigsqcup_{[E,\tau]\in \mathcal{P}_\Si} \Fix_\Si\,(F^{-1}(0)) /\Fix_\Si\,(\cG_E),$$ where $\mathcal{P}_\Si$ is the set of isomorphism classes of $\Si$-equivariant smooth Lie algebra $G$-bundles with fiber $\fg=\mathrm{Lie}(G)$ on $X$.
\end{corollary}

\subsection{From equivariant flat bundles to equivariant harmonic bundles}\label{from_flat_to_harmonic}

Let now $\fg$ be a real semisimple Lie algebra and let $G$ be the group of real points of $\Int(\fg\otimes\C)$. Let $(E,\nabla)$ be a flat Lie algebra bundle over $X$, with typical fiber $\fg$ and structure group $G$. Choose a Cartan involution $\theta:G\lra G$ and denote by $K:=\Fix(\theta) < G$ the associated maximal compact subgroup of $G$. The induced Lie algebra automorphism will also be denoted by $\theta$. Let $\rho_\nabla:\piX\lra G$ be the holonomy representation associated to the flat connection $\nabla$. For a flat $G$-bundle, a reduction of structure group from $G$ to $K=\Fix(\theta)$ (also called a $K$-reduction) can be defined as $\piX$-equivariant map $f:\Xt\lra G/K$, where $\piX$ acts on $G/K$ via $\rho_\nabla$ and left translations by elements of $G$. If 
such a map $f$ is given, then  $E=\piX\bs(\Xt\times \fg)$ inherits an involutive automorphism $\theta_f$, induced by the map \begin{equation}\label{Cartan_inv_on_reduced_flat_bundle}
\widetilde{\theta_f}: \Xt\times \fg \ni  (\eta,v) \lmt \big(\eta, \Ad_{f(\eta)\theta(f(\eta))^{-1}}\,\theta(v)\big) \in \Xt\times\fg
\end{equation} (it is immediate to check that $\widetilde{\theta_f}$ indeed descends to the bundle $\piX\bs(\Xt\times\fg)$), as well as a direct sum decomposition $E\simeq E_K\oplus P$ where $E_K:=\Fix(\theta_f)$ is a Lie algebra bundle with structure group $K$ and typical fiber $\fk := \mathrm{Lie}(K)$ and $P:=\{v\in E\ |\ \theta_f(v)=-v\}$ is a vector bundle with structure group $K$ and typical fiber $\fp:=\{v\in\fg\ |\ \theta(v)=-v\}$, satisfying $[E_K,P] \subset P$ and $[P,P]\subset E_K$ with respect to the fiberwise Lie bracket.  
The decomposition $E=E_K\oplus P$ will be called a \textit{Cartan decomposition} and the involution $\theta_f$ a \textit{Cartan involution} of $E$. It induces an identification $\cA_E \simeq \cA_{E_K}\times \Om^1(X;P)$, where $\cA_{E_K}$ is the space of $K$-connections on $E_K$.
Let then $(E,\nabla, \tau)$ be a $\Si$-equivariant flat bundle on $X$, in the sense of Definition \ref{equiv_flat_bdle_def}. By Theorem \ref{hol_of_inv_conn}, the holonomy representation $\rho_\nabla:\piX\lra G$ extends to a group homomorphism $\widetilde{\rho_{\nabla}}:\piY\lra G$.  We will be interested in $\piX$-equivariant maps $f:\Xt\lra G/K$ that are in fact $\piY$-equivariant (with respect to $\widetilde{\rho_{\nabla}}$).

\begin{proposition}\label{induced_equiv_structure_on_reduced_bundle}
Let $(E,\nabla,\tau)$ be a $\Si$-equivariant flat bundle on $X$ and let $Y$ be the orbifold $[\Si\bs X]$. We shall denote by $\eps$ the canonical morphism $\eps:\piY\lra \piY/\piX \simeq\Si$. Then the following properties hold:
\begin{enumerate}
\item The group $\piY$ acts on the set of $\piX$-equivariant maps $f:\Xt\lra G/K$ by $f^\ga :=\widetilde{\rho_{\nabla}} (\ga^{-1}) (f \circ \ga)$.
\item If $\theta_f$ is the Cartan involution of $E$ associated to $f$, then, for all $\ga\in\piY$, we have $\theta_{f^\ga} = 
\theta_f^{\eps(\ga)}$, where $\eps(\ga)\in\Si$ acts on $\theta_f$ via the $\Si$-action on gauge transformations of $(E,\tau)$: $\theta_f^\si=\tau_\si^{-1}\theta_f\tau_\si$ for all $\si\in\Si$.
\item Let $E\simeq E_K\oplus P$ be the Cartan decomposition of $E$ associated to $f$ and let $(A_f,\psi_f)\in \cA_{E_K}\times \Om^1(X;P)$ be the induced decomposition of $\nabla$ into a $K$-connection and a $P$-valued $1$-form, i.e.\ $\nabla=A_f+\psi_f$. Then, for all $\ga\in\piY$, we have $A_{f^\ga} = A_f^{\eps(\ga)}$ and $\psi_{f^\ga} = \psi_f^{\eps(\ga)}$.
\item If the map $f:\Xt\lra G/K$ is $\piY$-equivariant, then $\theta_f$ commutes to the $\Si$-equivariant structure $\tau$ on $E$. In particular, the restriction of $\tau$ induces $\Si$-equivariant structures on the $K$-bundles $E_K$ and $P$, therefore also a $\Si$-action on $\cA_{E_K}$ and $\Om^1(X;P)$, and the pair $(A_f,\psi_f)\in \cA_{E_K}\times \Om^1(X;P)$ is $\Si$-invariant.
\end{enumerate}
\end{proposition}

\begin{proof}
The proof boils down to the following:
\begin{enumerate}
\item We check that the map $\widetilde{\rho_{\nabla}} (\ga^{-1}) (f \circ \ga)$ from $\Xt$ to $G/K$ is $\piX$-equivariant. Since $\ga\delta\ga^{-1}\in\piX\lhd \piY$, we have, for all $\eta\in \Xt$ and all $\delta\in\piX$, that
$$\widetilde{\rho_{\nabla}} (\ga^{-1}) (f \circ \ga) (\delta\cdot\eta) = \widetilde{\rho_{\nabla}} (\ga^{-1}) \rho_\nabla(\ga\delta\ga^{-1}) f(\ga\cdot\eta) = \rho_\nabla(\delta) \big(\widetilde{\rho_{\nabla}}(\ga^{-1})(f\circ\ga)(\eta)\big).$$
\item It suffices to show that the map $\widetilde{\theta_f}$ defined in \eqref{Cartan_inv_on_reduced_flat_bundle} satisfies $\widetilde{\theta_{f^\ga}} = \widetilde{\tau}_\ga^{-1} \circ \widetilde{\theta_f} \circ \widetilde{\tau}_\ga$, where $\widetilde{\tau}_\ga$ is the map defined in \eqref{Si_equiv_structure_on_flat_bundle}, 
which follows from a direct computation. 
\item This is a simple computation, using the explicit definition of $A_f=\frac{1}{2}(\nabla + \theta_f\nabla\theta_f^{-1})$ and $\psi_f=\nabla-A_f$, 
as well as the $\Si$-invariance of $\nabla$. 
\item This follows immediately from the previous three properties. \qedhere
\end{enumerate}
\end{proof}

\noindent Of course, for Proposition \ref{induced_equiv_structure_on_reduced_bundle} to be useful, we need to make sure that $\piY$-equiva\-riant maps $f:\Xt\lra G/K$ indeed exist. One way to see this is as follows. Let $P_E$ be the principal $G$-bundle associated to $E$ and let $P_E(G/K):= P_E\times_G (G/K)$ be the bundle whose sections are $K$-reductions of $P_E$. The $\Si$-equivariant structure $\tau$ on $E$ induces a $\Si$-equivariant structure on $P_E$, that we shall still denote by $\tau$. Note that this $\tau$ is a $\Si$-equivariant structure in the principal bundle sense, so we have the compatibility relation $\tau_\si(p\cdot g) = \tau_\si(p)\cdot g$ between $\tau$ and the action of $G$ on $P_E$. 
Then $P_E(G/K)$ also has a $\Si$-equivariant structure, given by $$P_E\times_G (G/K) \ni  \left[p,gK\right] \lmt \left[\tau_\si(p),gK\right] \in P_E\times_G(G/K),$$ 
which is indeed well-defined by the previous remark. In particular, it makes sense to speak of $\Si$-equivariant sections of $P_E(G/K)$, and these do exist as we can average an arbitrary section of $P_E(G/K)$ over the finite group $\Si$, since there is a notion of center of mass in the simply connected complete Riemannian manifold of non-positive curvature $G/K$ (see e.g.\ \cite[Section~3.2]{Jost_nonpositive_curvature}). 
We then have the following result, whose proof is straightforward.
\begin{lemma}\label{existence_of_invariant_K_reductions}
Denote by $q:\Xt\lra X$ the universal covering map. Then there is a $\piX$-equivariant isomorphism $q^*P_E(G/K)\simeq \Xt\times G/K$, which induces a bijection between $\Si$-equivariant sections of $P_E(G/K)\lra X$ and $\piY$-equivariant maps $f:\Xt\lra G/K$. In particular, the latter do exist.
\end{lemma}

\noindent We cannot speak of a $\Si$-equivariant map $f:\Xt\lra G/K$, because the $\Si$-action on $X$ does not lift to $\Xt$ in general. However, we will (slightly abusively) speak of $\Si$-invariant reductions in the following sense.

\begin{definition}[Invariant $K$-reduction]\label{invariant_reduction}
Let $(E,\nabla,\tau)$ be a $\Si$-equivariant flat bundle on $X$. A $\piY$-equivariant map $f:\Xt\lra G/K$ will be called a \textit{$\Si$-invariant $K$-reduction} of $(E,\nabla,\tau)$.
\end{definition}

As Lemma \ref{existence_of_invariant_K_reductions} shows, $\piY$-equivariant maps $f:\Xt\lra G/K$ exist, and as Proposition \ref{induced_equiv_structure_on_reduced_bundle} shows, the Cartan decomposition $E\simeq E_K\oplus P$ associated to such an $f$ is compatible with the $\Si$-action in the sense that $\tau_\si(E_K)\subset E_K$ and $\tau_\si(P)\subset P$ for all $\si\in\Si$. 
In the context of $\Si$-equivariant flat bundles, we then have the following notion of stability, which generalizes Corlette's definition \cite[Definition 3.1]{Corlette}\label{stability_for_flat_bdles} and will eventually lead to a generalization of the Donaldson-Corlette Theorem \cite{Don_twisted}, \cite[Theorem 3.4.4]{Corlette}.

\begin{definition}[Stability condition for equivariant flat bundles]\label{Sigma_stability_for_flat_bdles}
A $\Si$-equivariant flat Lie algebra $G$-bundle $(E,\nabla,\tau)$ on $X$ is called:
\begin{itemize}
\item $\Si$\textit{-irreducible} (or $\Si$-\textit{stable}) if it contains no non-trivial $\nabla$-invariant $\Si$-sub-bundle (or equivalently, if the extended holonomy representation $\widetilde{\rho_\nabla}:\piY\lra G\subset\Aut(\fg)$ of Theorem \ref{hol_of_inv_conn} turns $\fg$ into an irreducible $\piY$-module).
\item $\Si$\textit{-completely reducible} (or $\Si$\textit{-polystable}) if $(E,\nabla,\tau)$ is isomorphic to a direct sum $\oplus_{1\leqslant i\leqslant k} (E_i,\nabla_i,\tau_i)$ of irreducible $\Si$-equivariant flat bundles (or equivalently, if $\fg$ is isomorphic, as a $\piY$-module, to a direct sum $\oplus_{1\leqslant i\leqslant k} \fg_i$ of irreducible $\piY$-modules).
\end{itemize}
\end{definition}

Here, a $\piY$-module is a pair $(\fg,\rho)$ consisting of a Lie algebra $\fg$ and a homomorphism $\rho:\piY\lra G$ to the group of real points of $\Int(\fg\otimes\C)$, and another possible characterization of complete reducibility is to say that every $\nabla$-invariant $\Si$-sub-bundle $F$ of the flat bundle $(E,\nabla)$ has a complement that is both $\Si$-invariant and  $\nabla$-invariant (or equivalently, that any sub-$\piY$-module of $\fg$ has a $\piY$-invariant complement). Evidently, if a flat bundle $(E,\nabla)$ is stable, then, for any $\Si$-equivariant structure $\tau$ leaving $\nabla$ invariant, the equivariant flat bundle $(E,\nabla,\tau)$ is $\Si$-stable. But $\Si$-stability of $(E,\nabla,\tau)$ only implies polystability of $(E,\nabla)$ in general. As a matter of fact, $(E,\nabla,\tau)$ is $\Si$-polystable if and only if $(E,\nabla)$ is polystable, as follows from the following result.

\begin{proposition}\label{pst_and_Sigma_pst_for_flat_bundles}
Let $\fg$ be a real semisimple Lie algebra and let $G$ be the group of real points of $\Int(\fg\otimes\C)$. Let $[\Si\bs X]\simeq Y$ be a presentation of the orbifold $Y$ and let $\rho:\piY\lra G$ be a representation of the orbifold fundamental group of $Y$ in $G$. Then $\fg$ is completely reducible as a $\piY$-module if and only if it is completely reducible as a $\piX$-module.
\end{proposition}

\begin{proof}
Since $\piX$ is a normal subgroup of finite index of $\piY$, the result follows for instance from \cite{Serre_cr_ss_gpe_normal}.
\end{proof}

The next result lays the groundwork for the first half of the non-Abelian Hodge correspondence for $\Si$-equivariant bundles: if the $\Si$-equivariant flat bundle $(E,\nabla,\tau)$ is $\Si$-polystable, it admits a $\Si$-invariant harmonic $K$-reduction $f$ (in the sense of Definition \ref{invariant_reduction}), which defines a \textit{$\Si$-equivariant harmonic bundle} $(E,\nabla,f,\tau)$, i.e.\ a harmonic bundle $(E,\nabla,f)$ endowed with a $\Si$-equivariant structure $\tau$ that leaves the connection $\nabla$, the harmonic $K$-reduction $f$, the connection $A_f$ and the $1$-form $\psi_f$ all invariant.

\begin{theorem}[Invariant harmonic reductions of equivariant bundles]\cite[Theorem 2.2]{HWW}\label{Sigma_equiv_DC_thm}
Let $\fg$ be a real semisimple Lie algebra and let $G$ be the group of real points of $\Int(\fg\otimes\C)$. Let $K < G$ be a maximal compact subgroup and let $(E,\nabla,\tau)$ be a $\Si$-equivariant flat Lie algebra $G$-bundle with fiber $\fg$ over $X$. Then $(E,\nabla,\tau)$ admits a $\Si$-invariant harmonic $K$-reduction $f:\Xt\lra G/K$ if and only if it is $\Si$-polystable.
\end{theorem}

\begin{remark}
In \cite{HWW}, Theorem \ref{Sigma_equiv_DC_thm} is proved in the special case where $\Si \simeq \sfrac{\Z}{2\Z}$, but their techniques extend to the case where $\Si$ is any finite group. Note also that, in \cite{HWW}, $X$ is of arbitrary dimension.
\end{remark}

\subsection{From equivariant Higgs bundles to equivariant harmonic bundles}\label{equiv_Higgs_bdles}

Let $(X,\Si)$ be a closed orientable surface equip\-ped with an action of a finite group $\Si$. We fix an orientation and a $\Si$-invariant Riemannian metric $g$ on $X$, and denote by $J$ the associated complex structure. Then a transformation $\si \in \Si$ is holomorphic with respect to $J$ if it preserves the orientation of $X$; otherwise, it is anti-holomorphic (note that, here, $\Si$ is a subgroup of $\mathrm{Diff}(X)$, not $\mathrm{MCG}(X)=\pi_0(\mathrm{Diff}(X))$, so finding a complex structure $J$ on $X$ such that $\Sigma\subset \Aut^{\pm}(X,J)$ is elementary). A $\Si$-equivariant structure $\tau$ on a holomorphic vector bundle $\cE\lra X$ is a family $\tau=(\tau_\si)_{\si\in\Si}$ of either holomorphic or anti-holomorphic transformations of $\cE$ satisfying:
\begin{enumerate}
\item For all $\si\in\Si$, Diagram \eqref{lift_of_the_action} (with $E$ replaced by $\cE$) is commutative,
\item The bundle map $\tau_\si$ is fiberwise $\C$-linear if $\si:X\lra X$ is holomorphic and fiberwise $\C$-anti-linear if $\si:X\lra X$ is anti-holomorphic,
\item One has $\tau_{1_\Si} = \id_\cE$ and, for all $\si_1,\si_2\in\Si$, $\tau_{\si_1\si_2} = \tau_{\si_1}\tau_{\si_2}$.
\end{enumerate} For instance, the canonical bundle $K_X$ of $X$ has a $\Si$-equivariant structure induced by the $\Si$-action on $X$. Moreover, if $(\cE,\tau)$ is a $\Si$-equivariant holomorphic vector bundle on $X$, then any associated bundle inherits a $\Si$-equivariant structure. For instance, $\End(\cE)\simeq \cE^*\otimes \cE$ has the induced $\Si$-equivariant structure $\xi\otimes v\lmt (\xi\circ\tau_\si^{-1})\otimes \tau_\si(v)$, 
which we shall simply denote by $\tau$. As a consequence, $K_X\otimes \End(\cE)$ also has an induced $\Si$-equivariant structure, that we denote by $(\si\otimes\tau_\si)_{\si\in\Si}$, and the space of sections of $K_X\otimes \End(\cE)$, being the space of sections of an equivariant bundle, inherits a $\Si$-action defined, for $\si\in\Si$, by 
\begin{equation}\label{action_on_Higgs_fields} \si(\phi) := (\si\otimes\tau_\si) \circ \phi \circ \si^{-1}.\end{equation} Whenever the holomorphic vector bundle $\cE$ has an extra structure (for instance, a holomorphic Lie bracket), we will assume, in the definition of a $\Si$-equivariant structure $\tau$, that the bundle maps $\tau_\si:\cE\lra\cE$ are compatible with that structure. In this paper, we consider $G$-Higgs bundles on $X$ for $G$ a real form of a connected semisimple complex Lie group of adjoint type $G_\C$. We denote by $\fg$ the Lie algebra of $G$.  If $\theta:G\lra G$ is a Cartan involution, $K:=\Fix(\theta) < G$ is the associated maximal compact subgroup and $K\lra\GL(\fp)$ is the isotropy (adjoint) representation of $K$ on the $(-1)$-eigenspace of $\theta:\fg\lra\fg$, then, by definition, a $G$-Higgs bundle on $X$ is a pair $(\cP,\phi)$ consisting of a holomorphic principal $K_\C$-bundle $\cP$, where $K_\C$ is the complexification of $K$, and a holomorphic section $\phi\in H^0(X;K_X\otimes\cP(\fp_\C))$, where $\fp_\C:=\fp\otimes\C$ and $\cP(\fp_\C):=\cP\times_{K_\C}\fp_\C$. Let  us now specialize this definition to the case where $G$ is the group of real points of $G_\C:=\Int(\fg_\C)$, where $\fg$ is a real semisimple Lie algebra and $\fg_\C:=\fg\otimes\C$. We let $\fk$ be a maximal compact Lie subalgebra of $\fg$, with respect to the Killing form $\kappa$, and we denote by $K_\C < G_\C$ be the connected subgroup corresponding to the Lie algebra $\fk_\C:=\fk\otimes\C$ (i.e.\ here, $K_\C=\Int(\fk\otimes\C)$). We denote by $\theta_\C$ (resp.\ $\kappa_\C$) the $\C$-linear extension to $\fg_\C$ of the Cartan involution $\theta$ (resp.\ Killing form $\kappa$) of $\fg$. Then $K_\C=\Fix(\theta_\C)$ in $G_\C$ and we set $K:=\Fix(\theta)$ in $G$. Moreover, the positive definite quadratic form $B_\theta(x,y):=-\kappa(\theta(x),y)$ on $\fg$ induces a non-degenerate $\C$-valued quadratic form $B_{\theta_\C}$ on $\fg_\C$, whose group of isometries contains $K_\C$ and whose space of symmetric endomorphisms contains the space of adjoint transformations of the form $\ad_x=[x,\,\cdot\,]$ for $x\in\fp_\C :=\fp\otimes\C$. Using the faithful representations $K_\C\hookrightarrow \mathbf{O}(\fg_\C,B_{\theta_\C})$ and $\fp_\C\hookrightarrow \mathrm{Sym}(\fg_ \C,B_{\theta_\C})$, we can now give the following definition of a $G$-Higgs bundle for $G$ as above.

\begin{definition}[Higgs bundles for real forms of connected complex semisimple Lie groups of adjoint type]\label{G_Higgs_bdle_def}
Let $\fg$ be a real semisimple Lie algebra and let $G$ be the group of real points of $G_\C:=\Int(\fg\otimes\C)$. Let $\fg=\fk\oplus\fp$ be a Cartan decomposition of $\fg$ with Cartan involution $\theta$. Let $K$ be the maximal compact subgroup of $G$ with Lie algebra $\fk$ and let $K_\C$ be the complex subgroup of $G_\C$ with Lie algebra $\fk\otimes\C$ and maximal compact subgroup $K$.

By a $G$-Higgs bundle on the Riemann surface $X$, we shall mean a pair $(\cE,\phi)$ consisting of
\begin{itemize}
\item a holomorphic Lie algebra bundle $\cE$ with typical fiber $\fg_\C:=\fg\otimes \C$ and structure group $K_\C$, and
\item a holomorphic $1$-form $\phi\in H^0(X;K_X\otimes \ad_{\fp_\C}(\cE))$, called the Higgs field,
\end{itemize} where by $\ad_{\fp_\C}(\cE)$ we mean the bundle of symmetric adjoint endomorphisms of $\cE$, i.e.\ endomorphisms of $\cE$ locally of the form $\ad_{\xi}=[\xi,\,\cdot\,]:\fg_\C\lra\fg_\C$, for some $\xi\in\fp_\C:=\fp\otimes\C$. This notion is indeed independent of the choice of local trivialization because the adjoint action of $K_\C$ preserves $\fp_\C$.
\end{definition}

By construction, the group $K_\C$ is reductive. It is not necessarily connected (its identity component is $\Int(\fk\otimes\C)$). For instance, when $G=\PGL(2,\R)$, one has $K_\C=\PO(2,\C)$, which has two connected components, the identity component being $\PSO(2,\C)\simeq\C^*/\{\pm 1\}$.

\color{black}

\begin{remark}
Giving a $G$-Higgs bundle in the sense of Definition \ref{G_Higgs_bdle_def} is equivalent to giving a triple $(\cE,\beta,\phi)$ where:
\begin{itemize}
\item $\cE$ is a holomorphic Lie algebra bundle with typical fiber $\fg_\C:=\fg\otimes \C$ and structure group $G_\C$, 
\item $\beta\in H^0(X;S^2\cE^*)$ is a non-degenerate quadratic form on $\cE$ which is compatible with the Lie bracket in the sense that $\beta([v_1,v_2],v_3) = \beta(v_1,[v_2,v_3])$, and
\item $\phi\in H^0(X;K_X\otimes \ad(\cE))$ is symmetric with respect to $\beta$.
\end{itemize} Indeed, $\beta$ will be fiberwise of the form $B_{\theta_\C}$ for $\theta:\fg\lra\fg$ a (fixed) Cartan involution, thus inducing a reduction of structure group from $G_\C$ to $K_\C$, so the Higgs field $\phi$ is symmetric with respect to $\beta$ if and only if it is $\ad_{\fp_\C}(\cE)$-valued.
\end{remark}

It will be convenient, at times, to see a holomorphic vector bundle $\cE$ as a pair $(E,\ov{\partial}_E)$ consisting of a smooth complex vector bundle $E$ on $X$ and a Dolbeault operator $\ov{\partial}_{E}:\Om^0(X;E)\lra \Om^{0,1}(X;E)$. As an example of $G$-Higgs bundle for $G$ as above, consider the case where $\fg=\fh_\C$ is already a complex semisimple Lie algebra. Then $K_\C\simeq H_\C$, $\fp_\C\simeq\fh_\C$ and $\cP(\fp_\C) \simeq\cP\times_{H_\C}\fh_\C \simeq \ad(\cP)$. So, when $G=H_\C$ is a connected complex semisimple Lie group of adjoint type, an $H_\C$-Higgs bundle can be thought of as a holomorphic Lie algebra vector bundle $\cE$, with typical fiber $\fh_\C$ and structure group $H_\C\simeq\Int(\fh_\C)$, equipped with a holomorphic $1$-form $\phi$ with values in adjoint endomorphisms of $\cE$.
 Another fundamental example is given by the case where $\fg=\fk$ is a compact semisimple Lie algebra. Then $G=K$, so $\fp=0$, and a $K$-Higgs bundle is a pair $(\cE,\phi)=(\cE,0)$ consisting of a holomorphic Lie algebra vector bundle $\cE$, with typical fiber $\fk_\C$ and structure group $K_\C$. Note that when $\fk=\fu(n)$, then $K\simeq \PU(n)$. 
A more elaborate example is given as follows: given a real semisimple Lie algebra $\fg$ and $G$ the group of real points of $\Int(\fg_\C)$, if $f:\Xt\lra G/K$ is a harmonic $K$-reduction of a polystable flat Lie algebra $G$-bundle $(E,\nabla)$, with associated Cartan decomposition $E=E_K\oplus P$ and $\nabla=A_f+\psi_f$, then the harmonic bundle $(E\otimes\C,d_{A_f}^{0,1}, \psi_{f}^{1,0})$ is a $G$-Higgs bundle. 
In this last example, the vector bundle $\cE$ in particular has vanishing Chern classes. In our context, the following definition is then natural (and is a special case of the notion of pseudo-equivariant $G$-Higgs bundle developed for an arbitrary semisimple Lie group $G$ in \cite{GPW,Heller_Schaposnik}).

\begin{definition}[Equivariant Higgs bundles]
A $\Si$-equivariant $G$-Higgs bundle on $(X,\Si)$ is a triple $(\cE,\phi,\tau)$ consisting of a $G$-Higgs bundle $(\cE,\phi)$ and a $\Si$-equivariant structure $\tau=(\tau_\si)_{\si\in\Si}$ leaving the Higgs field $\phi$ invariant, i.e.\ such that, for all $\si\in\Si$, one has $\si(\phi)=\phi$ with respect to the action of $\Si$ on $H^0(X;K_X\otimes\ad_{\fp_\C}(\cE))$ defined in \eqref{action_on_Higgs_fields}. A homomorphism of $\Si$-equivariant $G$-Higgs bundles is a homomorphism of $G$-Higgs bundles that commutes to the $\Si$-equivariant structures.
\end{definition}

\noindent The $\Si$-invariance condition on the Higgs field $\phi$ can also be phrased in the following way: for all $\si\in\Si$, the following diagram, where by $\tau$ we mean the $\Si$-equivariant structure of $\ad_{\fp_\C}(\cE)\subset \End(\cE)$ induced by that of $\cE$, is commutative.
$$
\xymatrix{
\ad_{\fp_\C}(\cE) \ar[r]^-{\phi} \ar[d]^{\tau_\si} & K_X\otimes \ad_{\fp_\C}(\cE) \ar[d]^{\si\otimes\tau_\si} \\
\ad_{\fp_\C}(\cE) \ar[r]^-{\phi} & K_X\otimes \ad_{\fp_\C}(\cE)
}
$$

We will now further restrict ourselves to $G$-Higgs bundles that have vanishing Chern classes, because, in that case, we can take semistability of a principal $G$-Higgs bundle $(\cP,\phi)$ to mean that the vector $G$-Higgs bundle $(\cP(V_\C),\phi_{V_\C})$ associated to $(\cP,\phi)$ via a faithful representation $G_\C\hookrightarrow \GL(V_\C)$ is semistable \cite[p.86]{Simpson_local_systems}. Here, as $G_\C$ is of adjoint type, we can take $V_\C:=\fg_\C$. In the $\Si$-equivariant setting, we then have the following definition, which will be sufficient for our purposes.

\begin{definition}[Stability condition for equivariant Higgs bundles]
Let $\fg$ be a real semisimple Lie algebra and let $G$ be the group of real points of $\Int(\fg\otimes\C)$. A $\Si$-equiva\-riant $G$-Higgs bundle $(\cE,\phi,\tau)$ with vanishing first Chern class on $X$ is called:
\begin{itemize}
\item $\Si$-semistable if, for all non-trivial sub-bundle $\cF\subset\cE$ such that $\phi(\cF)\subset K_X\otimes \cF$ and $\tau_\si(\cF)\subset\cF$ for all $\si\in\Si$, the degree of $\cF$ is non-positive, i.e.
$\deg(\cF)\leqslant 0$.
\item $\Si$-stable if the above inequality is strict,
\item $\Si$-polystable if it is isomorphic to a direct sum of $\Si$-stable equivariant Higgs bundles of degree 0.
\end{itemize}
\end{definition}

\noindent The point of this definition is that any $\Si$-semistable equivariant $G$-Higgs bundle has an associated $\Si$-polystable equivariant Higgs bundle (the graded object associated to any choice of a Jordan-H\"older filtration of the initial bundle, defined up to isomorphism) and that such objects admit a characterization in terms of special metrics, namely Hermitian-Yang-Mills metrics (Theorem \ref{invariant_HYM_metrics}, which is due to Simpson in \cite{Simpson_JAMS,Simpson_local_systems}). An example of such a $\Si$-polystable equivariant $G$-Higgs bundle is provided by the equivariant $G$-Higgs bundle $(E\otimes\C,d_{A_f}^{0,1}, \psi_{f}^{1,0},\tau)$ associated to a $\Si$-invariant harmonic $K$-reduction $f:\Xt\lra G/K$ of a $\Si$-polystable equivariant flat bundle $(E,\nabla,\tau)$. Note that such a map $f$ exists by Theorem \ref{Sigma_equiv_DC_thm}. Moreover, Theorem 1 of \cite{Simpson_JAMS} is already stated in a $\Si$-equivariant setting, for $\Si$ a finite group of holomorphic automorphisms of $X$. The extension to the case where $\Si$ is allowed to contain anti-holomorphic transformations of $X$ is not difficult, once one realizes that such a group $\Sigma$ still acts on the space of smooth Hermitian metrics on a holomorphic vector bundle $\cE$, by setting, for all $x\in X$ and all $v_1,v_2$ in $\cE_x$,
\begin{equation}\label{action_on_metrics}
h^\si_x (v_1,v_2) =  \left\{ \begin{array}{cl}
h_{\si(x)} \big(\tau_\si(v_1),\tau_\si(v_2)\big) & \mathrm{if}\ \si\ \textrm{is holomorphic on}\ X, \vspace{7pt}\\ 
\ov{h_{\si(x)} \big(\tau_\si(v_1),\tau_\si(v_2)\big)} & \mathrm{if}\ \si\ \textrm{is anti-holomorphic on}\ X.
\end{array}\right.
\end{equation} We can therefore use Simpson's theorem \cite[Theorem 1]{Simpson_JAMS}. Note that Simpson's version actually has one extra degree of generality, namely the Higgs field $\phi$ is \textit{not} assumed to be preserved by the $\Si$-action, instead it suffices that there exists a character $\chi:\Si\lra\C^*$ such that, for all $\si\in\Si$, $\si(\phi) = \chi(\si)\phi$; when $\Si$ contains anti-holomorphic transformations, the group homomorphism $\chi:\Si\lra\C^*$ should be replaced by a crossed homomorphism, with respect to the action of $\Si$ on $\C^*$ defined by the canonical morphism $\Si\lra \sfrac{\Z}{2\Z}$ followed by complex conjugation on $\C^*$, but in any case this is not necessary for us here. As a matter of fact, we also need Simpson's extension of his result to $G$-Higgs bundles with $G$ a real form of a complex semisimple Lie group \cite[Corollary 6.16]{Simpson_local_systems}. A different approach to Theorem \ref{invariant_HYM_metrics} below and its generalization to pseudo-equivariant $G$-Higgs bundles can be found in \cite[Theorem 4.4]{GPW}. 
Given a $G$-Higgs bundle $(\cE,\phi)$ (with vanishing first Chern class) equipped with a Hermitian metric $h$, we denote by $\phi^{*_h}$ the fiberwise adjoint of the Higgs field $\phi$ with respect to $h$, and by $A_h$ the Chern connection associated to $h$. Recall that $h$ is called a \textit{Hermitian-Yang-Mills} metric on $(\cE,\phi)$ if the Chern connection $A_h$ satisfies the self-duality equation $F_{A_h} + [\phi,\phi^{*_h}] = 0$. In such a case, the triple $(\cE,\phi,h)$ defines a harmonic bundle in the sense of Section \ref{from_flat_to_harmonic} and the next, fundamental, result of Simpson's says that all harmonic bundles arise in this way from polystable Higgs bundles with vanishing first Chern class.

\begin{theorem}\cite{Simpson_JAMS,Simpson_local_systems}\label{invariant_HYM_metrics}
Let $\fg$ be a real semisimple Lie algebra and let $G$ be the group of real points of $\Int(\fg\otimes\C)$. Let $(\cE,\phi,\tau)$ be a $\Si$-equivariant $G$-Higgs bundle with vanishing first Chern class on $X$. Then there exists a $\Si$-invariant Hermitian-Yang-Mills metric $h$ on the holomorphic vector bundle $\cE$ if and only if $(\cE,\phi,\tau)$ is $\Si$-polystable.
\end{theorem} 

\begin{corollary}\label{Si_equiv_NAHT_Higgs_side}
Let $(\cE,\phi,\tau)$ be a $\Si$-equivariant $G$-Higgs bundle with vanishing first Chern class on $X$. If $(\cE,\phi,\tau)$ is $\Si$-polystable as an equivariant $G$-Higgs bundle, then $(\cE,\phi)$ is polystable as a $G$-Higgs bundle.
\end{corollary}

\begin{proof}
Assume that $(\cE,\phi,\tau)$ is $\Si$-polystable. Then, by Theorem \ref{invariant_HYM_metrics}, it admits a $\Si$-invariant Yang-Mills metric $h$. Such a metric is in particular Hermitian-Yang-Mills, so $(\cE,\phi)$ is polystable as a $G$-Higgs bundle.
\end{proof}

Our next goal is to show that the Chern connection of a $\Si$-invariant metric is necessarily $\Si$-invariant. This will follow from an elementary observation (Proposition \ref{action_on_Chern_connections}). Let $\cE$ be a holomorphic vector bundle on $X$ and think of it as a smooth vector bundle $E$ equipped with a Dolbeault operator $\ov{\partial}_E: \Omega^0(X;E)\lra \Omega^{0,1}(X;E)$. Saying that $\tau=(\tau_\si)_{\si\in\Si}$ is a $\Si$-equivariant structure in the holomorphic sense on $\cE$ is equivalent to saying that $\tau$ is a $\Si$-equivariant structure in the smooth sense on $E$ such that, additionally, $\si\ov{\partial}_E\si^{-1} = \ov{\partial}_E$ for all $\si\in \Si$, i.e.\ the Dolbeault operator $\ov{\partial}_E$ is equivariant with respect to the $\Si$-actions induced by $\tau$ on $\Omega^0(X;E)$ and $\Omega^{0,1}(X;E)$. Indeed, that equivariance condition implies that each $\tau_\si$ preserves the space $\ker\ov{\partial}_E$ of holomorphic sections of $\cE$, therefore is either holomorphic or anti-holomorphic with respect to $\ov{\partial}_E$.

\begin{proposition}\label{action_on_Chern_connections}
Let $(E,\dbar_E,\tau)$ be a $\Si$-equivariant holomorphic vector bundle on $X$. Then there is a right action $D\lmt D^\si$ of the group $\Si$ on the space $\{D:\Omega^0(X;E)\lra \Omega^1(X;E)\ |\ D^{0,1}=\dbar_E\}$  of linear connections compatible with $\dbar_E$. Moreover, if $h$ is a Hermitian metric on $(E,\dbar_E)$ and $A_h$ is the Chern connection associated to $h$, then one has, for all $\si\in\Si$, $A_h^\si = A_{h^\si}$ with respect to the action of $\Si$ on the space of metrics defined in \eqref{action_on_metrics}. In particular, if $h$ is $\Si$-invariant, then so is the Chern connection $A_h$.
\end{proposition}

\begin{proof}
First, we define an action of $\Si$ on the space of linear connections on $E$ that are compatible with the holomorphic structure $\dbar_E$. Set $D^\si := \si^{-1} D \si$, where $\si$ acts on $\Omega^k(X;E)$ in the usual way (see for instance \eqref{action_on_Higgs_fields}). It is clear that this action preserves the subspaces of $(1,0)$ and $(0,1)$ pseudo-connections, as conjugation by $\si$ is a $\C$-linear operation. Then $(\si^{-1}D\si)^{0,1} = \si^{-1} D^{0,1} \si = \si^{-1} \dbar_E \si = \dbar_E$, so $D^\si$ is indeed compatible with $\dbar_E$. Next, we prove that $A_h^\si=A_{h^\si}$. Recall that the Chern connection $A_h$ associated to the metric $h$ and the holomorphic structure $\dbar_E$ is the linear connection $A_h:=D_h + \dbar_E$ where $D_h$ is the operator of type $(1,0)$ uniquely determined by the condition $\dbar_J \big(h(s_1,s_2)\big) = h(D_h s_1,s_2) + h(s_1,\dbar_E s_2)$ for all smooth sections $s_1,s_2$ of $E$ (where $\dbar_J$ is the Cauchy-Riemann operator associated to $J$ on $X$). Then, one has, for all $\si\in\Si$, that $\dbar_J(h^\si(s_1,s_2)) = h^\si((\si^{-1} D_h\si)s_1,s_2) + h^\si(s_1,\dbar_E s_2)$, so $D_{h^\si} = \si^{-1} D_h \si = D_h^\si$ and $A_{h^\si} \,=\, D_{h^\si} + \dbar_E \,=\, \si^{-1} D_h \si + \si^{-1}\dbar_E\si \,=\, \si^{-1} A_h \si \,=\, A_h^\si$.
\end{proof}

\noindent Combining Proposition \ref{action_on_Chern_connections} with Simpson's Theorem \ref{invariant_HYM_metrics}, we obtain the main result of this section, which is the second half of the non-Abelian Hodge correspondence for $\Si$-equivariant bundles: if the $\Si$-equivariant $G$-Higgs bundle $(\cE,\phi,\tau)=(E,\ov{\partial}_E,\phi,\tau)$ is $\Si$-polystable, it admits a $\Si$-invariant Hermitian-Yang-Mills metric $h$, which defines a \textit{$\Si$-equivariant harmonic bundle} $(E,\ov{\partial}_E,\phi,h,\tau)=(E,\nabla_h,h,\tau)$, with $\Si$-invariant flat connection $\nabla_h:=A_h+\psi_h$ where $\psi_h=\phi+\phi^{*_h}$.

\begin{theorem}
Let $\fg$ be a real semisimple Lie algebra and let $G$ be the group of real points of $\Int(\fg\otimes\C)$. Let $(\cE,\phi,\tau)$ be a $\Si$-equivariant $G$-Higgs bundle on $X$. Then $(\cE,\phi,\tau)$ is $\Si$-polystable if and only if there exists a $\Si$-invariant Hermitian metric $h$ on $\cE$ such that the associated Chern connection $A_h$ is a $\Si$-invariant solution of the self-duality equations, namely $F_{A_h} + [\phi,\phi^{*_h}] = 0$ and, for all $\si\in\Si$, $A_h^\si = A_h$. 
\end{theorem}

\begin{proof}
Assume that $(\cE,\phi,\tau)$ is $\Si$-polystable. The existence of a $\Si$-invariant metric $h$ such that $A_h$ satisfies $F_{A_h}+[\phi,\phi^{*_h}]=0$ is provided by \cite{Simpson_JAMS} and \cite{Simpson_local_systems}, as recalled in Theorem \ref{invariant_HYM_metrics}. The $\Si$-invariance of the associated Chern connection then comes from Proposition \ref{action_on_Chern_connections}.
\end{proof}

\subsection{Non-Abelian Hodge correspondence}\label{NAHC_for_orbifolds}

Putting together the results of Section \ref{Hitchin_s_equations}, we obtain, given a real semisimple Lie algebra $\fg$, a hyperbolic $2$-orbifold $Y$ and a presentation $Y\simeq[\Si\bs X]$ of that orbifold as a quotient of a closed orientable hyperbolic surface $X$ by the action of a finite group of isometries $\Si$, a homeomorphism between the representation space $\Hom^{\mathrm{c.r.}}(\piY,G)/G$ of completely reducible representations of the orbifold fundamental group $\piY$ into the group of real points of $\Int(\fg\otimes\C)$ and the moduli space  $$\cM_{(X,\Si)}(G) := \left\{
\begin{matrix}
\Si\textrm{-polystable equivariant}\ G\textrm{-Higgs bundles}\\ 
(\cE,\phi,\tau)\ \textrm{with vanishing first Chern class on}\ X
\end{matrix}
\right\} \big/\, \mathrm{isomorphism}$$ of isomorphism classes of $\Si$-polystable equivariant $G$-Higgs bundles with vanishing first Chern class on $X$. We shall refer to that homeomorphism as a \textit{non-Abelian Hodge correspondence for orbifolds}, depending on the presentation $Y\simeq[\Si\bs X]$, and we now proceed to analyzing the Hitchin component of $\Hom(\piY,G)/G$ in terms of that correspondence.

\section{Parameterization of Hitchin components}\label{section:parameterization_Hitchin}

Throughout this section, we fix a presentation $Y \simeq [\Si\bs X]$ of the orbifold $Y$, where $X$ is assumed to be a Riemann surface and $\Si$ acts on $X$ by transformations that are either holomorphic or anti-holomorphic (see Definition \ref{pres_of_Y_def}), and we let $G = \Int(\fgC)^\tau$, where $\fg$ is now a \emph{split} real form of a complex \textit{simple} Lie algebra. By Lemma \ref{lemma:centralizer}, if $
\rho:\piY \lra G \subset \GL(\fg)$ is a Hitchin representation, then $\fg$ is an irreducible $\piY$-module under $\rho$. So, using the non-Abelian Hodge correspondence for orbifolds recalled in Section \ref{NAHC_for_orbifolds}, we can think of $\Hit(\piY,G)$ as a connected component of \begin{equation}\label{NAHT_for_orbifolds}\Hom^{\mathrm{c.r.}}(\piY,G)/G \simeq \cM_{(X,\Si)}(G).\end{equation} By Corollary \ref{Si_equiv_NAHT_Higgs_side}, there is a well-defined map $J: \cM_{(X,\Si)}(G) \ni (\cE,\phi,\tau) \lmt  (\cE,\phi) \in \cM_X(G)$ 
forgetting the $\Si$-equivariant structure $\tau$. The group $\Si$ acts on $\cM_X(G)$ (by pullback of bundles and Higgs fields, see \eqref{action_on_Higgs_bundles}) and, as in Lemma \ref{map_to_fixed_locus}, the image of $J$ is contained in $\Fix_\Si(\cM_X(G))$ but the resulting map $J:\cM_{(X,\Si)}(G) \lra \Fix_\Si(\cM_X(G))$ is again neither injective nor surjective in general. In this section, we will show that, if we restrict it to $\Hit(\piY,G)\subset \cM_{(X,\Si)}(G)$, then the map $J$ induces a homeomorphism $\Hit(\piY,G) \simeq \Fix_\Si(\Hit(\piX,G))$.

\subsection{Equivariance of the Hitchin fibration}\label{subsec:Hitchin_fibration}

Recall that $\Si$ acts on $X$ by transformations that are either holomorphic or anti-holomorphic. The induced action on the canonical bundle $K_X$ defines a $\Si$-equivariant structure $(\tau_\si)_{\si\in\Si}$ in the holomorphic sense on $K_X$. As seen in Section \ref{equiv_Higgs_bdles}, this in turn induces an action of $\Si$ on all tensor powers $K_X^d$ of the canonical bundle, and on sections of such bundles: if $s\in H^0(X;K_X^d)$ and $\si\in\Si$, we set $\si(s) := \tau_\si \circ s \circ \si^{-1}$.
\begin{equation}\label{action_on_hol_diff}
\xymatrix{
K_X^d \ar[r]^{\tau_\si} \ar[d] & K_X^d\ar[d]\\
X \ar@/_/[u]_{s} \ar[r]^{\si} & X \ar@{.>}@/_/[u]_{\si(s)}
}
\end{equation} Since $\si$ and $\tau_\si$ are either simultaneously holomorphic or simultaneously anti-holo\-morphic, $\si(s)$ is indeed a holomorphic section of $K_X^d$. Explicitly for $d=1$, as $\tau_\si:K_X\lra K_X$ is just the transpose of the tangent map $T\si^{-1}$, we have $$\si(s) = \left\{
\begin{array}{rl}
(\si^{-1})^*s & \mathrm{if}\ \si\ \textrm{is holomorphic,} \vspace{7pt}\\
\ov{(\si^{-1})^*s} & \mathrm{if}\ \si\ \textrm{is anti-holomorphic,}
\end{array}\right.$$ where, by definition, $\ov{(\si^{-1})^*s}$ sends $v\in T_z X$ to $\ov{\big(s(\si^{-1}(z)) \circ T_z \si^{-1}\big)\cdot v}\in \C$. And finally, if $X$ is an open set in $\C$ with an action of $\Si$ and $s(z) = f(z)\, dz$, then 
\begin{equation}\label{Sigma_action_on_sections_local_form}
\si(s) = \left\{
\begin{array}{ll}
(f\circ \si^{-1}) \ (\partial\si)\, dz & \mathrm{if}\ \si\ \textrm{is holomorphic,}\vspace{7pt}\\
\ov{(f\circ \si^{-1})} \ (\partial\ov{\si}) \, dz & \mathrm{if}\ \si\ \textrm{is anti-holomorphic,}
\end{array}\right.
\end{equation} where by $\partial\si$ we denote the $\C$-linear part of the differential $d\si$ of the $\R$-differentiable map $\si:\C\lra\C$.
We now recall the definition of the Hitchin fibration $F:\cM_X(G) \lra \cB_X(\fg)$, where the \emph{Hitchin base} $\cB_X(\fg)$ will be defined in \eqref{Hitchin_fibration}. 

\begin{remark}
This fibration was introduced by Hitchin for simple complex Lie groups $G_\C$ in \cite{Hitchin_Duke}. For a real Lie group $G$ like ours ($=$ split real form of a connected simple complex Lie group of adjoint type), there are two possibilities to define the Hitchin fibration. Either, as in \cite{Hitchin_Teich}, by composing the original Hitchin fibration $F_{\C}: \cM_X(G_\C)\lra \cB_X(\fg_\C)$ with the canonical map $\cM_{X}(G) \lra \cM_X(G_\C)$, or, as in \cite{GPPNR}, by a direct definition generalizing the one in \cite{Hitchin_Duke}. The latter is perhaps preferable from our point of view, because it avoids the injectivity defect of the canonical map $\cM_{X}(G) \lra \cM_X(G_\C)$. For the two approaches to actually coincide, one needs in particular to have $\cB_X(\fg)=\cB_X(\fg_\C)$, which is true by the assumption that $\fg$ is a split real form of $\fg_\C$ (see \cite{GPPNR}). 
\end{remark}

Let $\fg$ be the split real form of a simple complex Lie algebra $\fg_\C$ and let $\fg=\fk\oplus\fp$ be a Cartan decomposition of $\fg$ (with respect to the Killing form). As usual, set $G_\C:=\Int(\fg_\C)$ and let $K$ be the maximal compact subgroup of $G$ with Lie algebra $\fk$. Finally, let $K_\C$ be the complex subgroup of $G_\C$ with Lie algebra $\fk\otimes\C$ and maximal compact subgroup $K$. The adjoint action of $K\subset G$ on $\fg$ preserves $\fp$, and  the induced action of $K_\C$ on $\fp_\C$ is compatible with the canonical real structures of these spaces, in the sense that $\ov{\Ad_k\, \xi} = \Ad_{\ov{k}}\, \ov{\xi}$ for all $k \in K_\C$ and all $\xi \in \fp_\C$. Let $r:=\rk (\fg)$ denote the real rank of $\fg$. Note that since $\fg$ is split by assumption, this is equal to the rank of $\fg_\C$. By a theorem due to Kostant and Rallis \cite{Kostant_Rallis_orbits_in_symmetric_spaces}, the $\R$-algebra $\R[\fp]^K $ of $K$-invariant regular functions on $\fp$ is generated by exactly $r$ homogeneous polynomials $(P_1,\,\ldots\,, P_r)$. 
We set $d_\alpha:=\deg P_\alpha -1$ for all $\alpha \in \{1, \,\ldots\,,r\}$. The $(d_\alpha)_{1\leqslant \alpha\leqslant r}$ depend only on the real Lie algebra $\fg$ and are called the \emph{exponents} of $\fg$. Following \cite{Hitchin_Duke} and \cite{GPPNR}, every such family defines a fibration 
\begin{equation}\label{Hitchin_fibration}
F: \cM_X(G) \ni (\cE,\phi) \lmt \big(P_1(\phi),\,\ldots\, ,P_r(\phi)) \in
 \cB_X(\fg):= \displaystyle\bigoplus_{\alpha=1}^r H^0(X;K_X^{d_\alpha+1}).
\end{equation}

\noindent By Definition \ref{G_Higgs_bdle_def}, the Higgs field $\phi\in H^0(X;K_X\otimes \ad_{\fp_\C}(\cE))$ of a $G$-Higgs bundle $(\cE,\phi)$ is a holomorphic $1$-form on $X$ with values in the bundle of symmetric adjoint endomorphisms of $\cE$, i.e.\ endomorphisms that are locally of the form $\ad_\xi$ for some $\xi\in\fp_\C$. Since $\ad_{\fp_\C}(\cE)$ has fiber $\fp_\C$ and structure group $K_\C$, and each $P_\alpha\in\R[\fp]^K$ defines a $K_\C$-invariant $\C$-valued polynomial function on $\fp_\C$, we have that $P_\alpha(\phi)$ is indeed a (homogeneous) holomorphic differential, of degree equal to $\deg P_\alpha = d_\alpha+1$ on $X$. We shall now see that the Hitchin fibration \eqref{Hitchin_fibration} is $\Si$-equivariant. Recall first (see (\ref{action_on_hol_diff})) that the finite group $\Si$, consisting of transformations of $X$ that are either holomorphic or anti-holomorphic, acts on each complex vector space $H^0(X;K_X^{d_\alpha+1})$. Moreover, if $(\cE,\phi)$ is a $G$-Higgs bundle on $X$ and $\si\in\Si$, then there is a $G$-Higgs bundle 
\begin{equation}\label{action_on_Higgs_bundles}
\big(\si(\cE),\si(\phi)\big) = \left\{
\begin{array}{rl}
\big((\si^{-1})^*\cE, (\si^{-1})^*\phi\big) & \mathrm{if}\ \si\ \textrm{is holomorphic},\vspace{7pt}\\
\big(\ov{(\si^{-1})^*\cE}, \ov{(\si^{-1})^*\phi}\big) & \mathrm{if}\ \si\ \textrm{is anti-holomorphic},
\end{array}
\right.
\end{equation} where $\si(\phi)\in H^0(X;\si(K_X)\otimes \ad_{\fp_\C}(\si(\cE))) = H^0(X;K_X\otimes \ad_{\fp_\C}(\si(\cE)))$, since $K_X$ has a canonical $\Si$-equivariant structure (therefore is canonically isomorphic to $\si(K_X)$). The point is that $\si(\phi)$ is indeed a Higgs field on the $K_\C$-bundle $\si(\cE)$. Note that if, additionally, a $\Si$-equivariant structure $\tau$ on $\cE$ has been given, then there is a canonical isomorphism $\si(\cE)\simeq\cE$, and $\si(\phi)$ may therefore be viewed as a Higgs field on the original holomorphic bundle $\cE$: we recover in this way the canonical $\Si$-action $\phi\lmt\si(\phi)$ on sections of the $\Si$-equivariant bundle $K_X\otimes \ad_{\fp_\C}(\cE)$, as defined in \eqref{action_on_Higgs_fields}. We now want to compare $F(\si(\cE),\si(\phi))$ and $\si(F(\cE,\phi))$, where $\si\in\Si$ and $F$ is the Hitchin fibration of \eqref{Hitchin_fibration}.

\begin{proposition}
Let $\fg$ be the split real form of a simple complex Lie algebra $\fg_\C$ and let $G$ be the associated real form of the simple complex Lie group of adjoint type $G_\C:=\Int(\fg_\C)$. Let $\fg=\fk\oplus\fp$ be a Cartan decomposition of $\fg$ and let $K$ be the compact real form of $\Int(\fk\otimes\C)$. Let $X$ be a compact connected Riemann surface of genus $g\geqslant 2$ and let $\Si$ be a finite group acting effectively on $X$ by transformations that are either holomorphic or anti-holomorphic. For any choice of generators $(P_1,\,\ldots\, ,P_r)$ of the $\R$-algebra $\R[\fp]^K$, we denote by $F$ the associated Hitchin fibration \eqref{Hitchin_fibration}, 
where $d_\alpha:=\deg P_\alpha -1$ and $r:=\rk(\fg)$. Then $F:\cM_X(G)\lra \cB_X(\fg)$ is $\Si$-equivariant with respect to the $\Si$-action on $\cM_X(G)$ defined in \eqref{action_on_Higgs_bundles} and the $\Si$-action on $\cB_X(\fg)$ defined by means of \eqref{action_on_hol_diff}.
\end{proposition}

\begin{proof}
The Higgs field $\phi$ is a section of $\ad_{\fp_\C}(\cE)\otimes K_X$, so it is locally of the form $\xi \otimes dz$, where $\xi$ is a $\fp_\C$-valued holomorphic function. So, on the one hand, for all $\alpha\in\{1,\,\ldots\, ,r\}$, the holomorphic differential $P_\alpha(\phi)$ is locally of the form $(P_\alpha\circ\xi)\, (dz)^{d_\alpha+1}$. And on the other hand, for all $\si\in\Si$, the Higgs field $\si(\phi)$ (on $\si(\cE)$) is locally of the form $\si(\xi)\otimes \si(dz)$, where $\si(dz) = (\si^{-1})^*dz$ and $\si(\xi)=\xi\circ\si^{-1}$ if $\si$ is holomorphic on $X$, and $\si(dz) = \ov{(\si^{-1})^*dz}$ and $\si(\xi) = \ov{\xi\circ\si^{-1}}$ if $\si$ is anti-holomorphic on $X$ (in the latter expression, complex conjugation in $\fp_\C$ is taken with respect to the real form $\fp$). 
To compute $P_\alpha(\si(\phi))$, let us recall that $(P_1,\,\ldots\, ,P_r)$ are generators of the $\R$-algebra $\R[\fp]^K$. Equivalently, they are generators of the $\C$-algebra $\C[\fp_\C]^{K_\C}\simeq \R[\fp]^K\otimes\C$ that, in addition, are fixed points of the canonical real structure of $\R[\fp]^K\otimes\C$. Therefore, if $\si$ is anti-holomorphic, we have $P_\alpha(\ov{\xi\circ\si^{-1}}) = \ov{P_\alpha(\xi\circ\si^{-1})}$, where complex conjugation on the right-hand side is the usual one on $\C$, i.e.\ the function $P_\alpha:\fp_\C\lra\C$ is a real function in the sense that it commutes to the given real structures of $\fp_\C$ and $\C$ (which is indeed the case for polynomial functions with real coefficients in a real basis of $\fp_\C$). We note that the $\fp_\C$-valued function $\xi$ depends on the choice of a local trivialization of $\ad_{\fp_\C}(\cE)$, but $P_\alpha\circ\xi$ is independent of such a choice because the function $P_\alpha:\fp_\C\lra \C$ is $K_\C$-invariant and $K_\C$ is the structure group of $\ad_{\fp_\C}(\cE)$. 
Thus, we have shown that, if $\phi$ is locally of the form $\xi \otimes dz$, then $P_\alpha(\phi)$ is locally of the form $(P_\alpha\circ\xi)\otimes (dz)^{d_\alpha+1}$ and $P_\alpha(\si(\phi))$ is locally of the form $$\left\{
\begin{array}{rl}
\big((P_\alpha\circ\xi)\circ\si^{-1}\big) \otimes \big((\si^{-1})^*dz\big)^{d_\alpha+1} & \mathrm{if}\ \si\ \textrm{is holomorphic},\vspace{7pt}\\
\big(\ov{(P_\alpha\circ\xi)\circ\si^{-1}}\big) \otimes \big(\ov{(\si^{-1})^*dz\big)^{d_\alpha+1}} & \mathrm{if}\ \si\ \textrm{is anti-holomorphic}.
\end{array}
\right.$$ Comparing this with the definition of the $\Si$-action on $H^0(X;K_X^{d_\alpha+1})$ given in \eqref{Sigma_action_on_sections_local_form}, we have indeed that $P_\alpha(\si(\phi)) = \si(P_\alpha(\phi))$.
\end{proof}

\subsection{Invariant Hitchin representations}\label{inv_Hitchin_reps}

In \cite{Hitchin_Teich}, Hitchin constructed a section $s:\cB_X(\fg)\lra \cM_X(G)$ of the Hitchin fibration $F:\cM_X(G) \lra \cB_X(G)$ whose image is exactly the Hitchin component $\Hit(\piX,G)$, and we will now check that this section is $\Si$-equivariant in our context. This will enable us to prove Theorem \ref{main_obs_about_Hit_comp_for_orbifolds}.
Let us first briefly recall Hitchin's construction of his section, which uses Lie-theoretic results of Kostant \cite{Kostant_AJM_1963}. One starts with a split real form $\fg$ of a complex simple Lie algebra $\fg_\C$ and a Cartan decomposition $\fg=\fk\oplus\fp$. Then one chooses a regular nilpotent element $e\in \fp$ (i.e.\ $\ad_e$ is a nilpotent endomorphism of $\fg$ whose centralizer is of the smallest possible dimension, equal to the rank of $\fg$). By the strong Jacobson-Morozov Lemma \cite[Proposition 4]{Kostant_Rallis_orbits_in_symmetric_spaces}, $e$ can be embedded in a copy of $\sl(2,\R)$ in $\fg=\fk\oplus\fp$, i.e.\ one can find $x\in\fk$ semisimple and $\te\in\fp$ nilpotent such that $[x,e]=e,\  
[x,\te]=-\te\  
\mathrm{and}\  
[e,\te]=x$. We henceforth fix such a triple $(x,e,\te)$ and we let 
\begin{equation}\label{equation:decomposition_into_irreducible}
\fg_\C=\bigoplus_{\alpha=1}^r V_\alpha
\end{equation} 
be the decomposition of the $\sl(2,\C)$-module $\fg_\C$ into $r=\rk(\fg_\C)$ irreducible representations \cite{Kostant_AJM_1959}: each $V_\alpha$ is of odd dimension $2d_\alpha+1$ where the $(d_\alpha)_{1\leqslant \alpha\leqslant r}$ are the exponents of $\fg_\C$ (or equivalently, of $\fg$, since we are assuming that $\fg$ is split), and the eigenvalues  of the restriction of $\ad_x$ to $V_\alpha$ are the integers in the interval $[-d_\alpha,d_\alpha]$. For all $\alpha \in\{1,\,\ldots\,,r\}$ and all $d\in[-d_\alpha,d_\alpha]\cap\Z$, let us denote by $\fgmC$ the subspace of $\fg_\C$ on which $\ad_x$ acts with eigenvalue $d$. Then $\fg_\C= \bigoplus_{d=-M}^{M} \fgmC$, where $M=\max_{1\leqslant \alpha\leqslant r}d_\alpha$. Note that the eigenvalues of $\ad_x$ are all real and that $\fgmC= \fgm\otimes\C$ has a canonical real structure (likewise $V_\alpha$ has a canonical real structure, induced by that of $\fg_\C$, since the latter is a real $\sl(2,\C)$-module with respect to the real form $\sl(2,\R)\subset\sl(2,\C)$). Let us now consider the Lie algebra bundle 
\begin{equation}\label{Hitchin_bdle}
\Ecan := \bigoplus_{d=-M}^M K_X^d \otimes \fgmC
\end{equation} with fiber $\fgC$ and structure group $K_\C:=\Int(\fk\otimes\C)$. This is the bundle introduced by Hitchin in \cite[Section 5]{Hitchin_Teich}. It is endowed with the canonical Higgs field $\phi_0:=\te$, where the latter element is seen as a section of $K_X\otimes (K_X^{-1}\otimes\fg^{(-1)}_\C)\simeq\fg_{\C}^{(-1)}$: indeed, $\te\in\fg_{\C}^{(-1)}$ because $[x,\te]=-\te$ by construction of the triple $(x,e,\te)$. We note that  $\Ecan$ has a canonical $\Si$-equivariant structure, induced by the canonical $\Si$-equivariant structure of $K_X$ and the canonical real structure of $\fgmC=\fgm\otimes\C$: if $\si\in\Si$, then $\si$ acts on $K_X^d\otimes \fgmC$ via $\tau_\si^d\otimes\eps_\si$, where $\tau_\si$ is the transformation of $K_X$ induced by $\si$ and $\eps_\si:\fgmC\lra\fgmC$ is the identity map if $\si$ is holomorphic on $X$ and complex conjugation with respect to $\fgm$ if $\si$ is anti-holomorphic on $X$. 
The Hitchin section is then defined as follows. Given $p=(p_1,\,\ldots\,,p_r)\in\cB_X(\fg)=\oplus_{\alpha=1}^r H^0(X;K_X^{d_\alpha+1})$, one sets $\phi(p) := \te+ \sum_{\alpha=1}^r p_\alpha\otimes e_\alpha$, where $e_1,\,\ldots\, ,e_r$ are the highest weight vectors of the $\sl(2,\C)$-module $\fg_\C$ (with respect to the choice of the Lie sub-algebra of $\fg_\C$ generated by the $\sl(2,\R)$-triple $(x,e,\te)$, i.e.\ $e_\alpha\in V_\alpha\cap\fg$ and $\ad_x\, e_\alpha = d_\alpha e_\alpha$). Since $p_\alpha$ is a section of $K_X^{d_\alpha+1}$ and $\te$ and all the $e_\alpha$ lie in $\fp$, one has $\phi(p)\in H^0\big(X;K_X\otimes \ad_{\fp_\C}(\Ecan)\big)$, so $\phi(p)$ is indeed a Higgs field on $\Ecan$. Hitchin proved in \cite{Hitchin_Teich} that the map $p \lmt (\Ecan,\phi(p))$ is a section of the Hitchin fibration \eqref{Hitchin_fibration}, whose image is exactly $\Hit(\piX,G)$.

\begin{lemma}\label{equivariance_of_the_Hitchin_section}
The Hitchin section $s: \cB_X(\fg) \ni p \lmt (\Ecan,\phi(p)) \in
 \Hit(\piX,G)$ is $\Si$-equivariant. In particular, it induces a homeomorphism $\Fix_\Si(\cB_X(\fg)) \simeq \Fix_\Si(\Hit(\piX,G))$.
\end{lemma}

\begin{proof}
As $\Ecan$ is $\Si$-equivariant, there are canonical identifications $\si(\Ecan)\simeq\Ecan$ for all $\si\in\Si$ and we can think of $\si(\phi(p))$ as a Higgs field on $\Ecan$ itself. Recall that, by definition, $\phi(p) = \te + \sum_{\alpha=1}^r p_\alpha \otimes e_\alpha$. Since $\te$ and all the $e_\alpha$ are real with respect to the canonical real structure of $\fg_\C = \fg\otimes\C$, the $\Si$-equivariance of $s$ follows immediately from the definition of the $\Si$-action on the Hitchin base $\cB_X(\fg)$ and the $\Si$-action on the set of Higgs fields on a fixed $\Si$-equivariant bundle: $ \phi\big(\si(p)\big) \ = \ \te + \sum_{\alpha=1}^r \si(p_\alpha) \otimes e_\alpha\ =\  \si(\te) + \sum_{\alpha=1}^r \si(p_\alpha) \otimes \si(e_\alpha) 
\ = \ \si\big(\te + \sum_{\alpha=1}^r p_\alpha \otimes e_\alpha\big) 
\ = \ \si\big(\phi(p)\big)$,
where $\te$ and all the $e_\alpha$ are indeed $\Si$-equivariant when seen as sections of the $\Si$-equivariant bundles $X\times \fgmC$ because they are real elements of $\fgmC=\fgm\otimes\C$.
\end{proof}

We can now prove Theorem \ref{main_obs_about_Hit_comp_for_orbifolds}. Later on, in Section \ref{inv_diff_section}, we will compute the dimension of $\Hit(\piY,G)$.

\begin{proof}[Proof of Theorem \ref{main_obs_about_Hit_comp_for_orbifolds}]
By Proposition \ref{proposition:injectivity}, we know that, given a presentation $Y\simeq[\Si\bs X]$, the map \begin{equation*}
j: \Hit(\piY,G) \ni \left[\rho\right] \lmt \left[\rho|_{\piX}\right] \in \Fix_\Si(\Hit(\piX,G))
\end{equation*} is injective. To prove that it is surjective, let us consider an element $[\rho] \in \Fix_\Si(\Hit(\piX,G))$ and let us fix a hyperbolic structure on $Y$ (or, equivalently, on $X$, with $\Si$ acting by transformations that are either holomorphic or anti-holomorphic). By Lemma \ref{equivariance_of_the_Hitchin_section}, there is a unique $p \in\Fix_\Si(\cB_X(\fg))$ such that $s(p)=(\Ecan,\phi(p))$ is the $G$-Higgs bundle corresponding to $\rho$.
Since $\Ecan$ is $\Si$-equivariant and $\phi(p)$ is $\Si$-invariant, the non-Abelian Hodge correspondence of Section \ref{NAHC_for_orbifolds} shows that there is a $\Si$-equivariant flat $G$-bundle $(E_\mathrm{can},\nabla_p)$ associated to the $\Si$-equivariant $G$-Higgs bundle $(\Ecan,\phi(p))$. In particular, the flat connection $\nabla_p$ is $\Si$-invariant, so, by Theorem \ref{hol_of_inv_conn}, the associated holonomy representation $\rho_{\nabla_p}=\rho$, from $\piX$ to $G$, extends to a representation $\widetilde{\rho_{\nabla_p}}:\piY\lra G$.
It remains to prove that $\widetilde{\rho_{\nabla_p}}$ is indeed a Hitchin representation of $\piY$. This follows from the connectedness of the real vector space $\Fix_\Si(\cB_X(\fg))$ and the fact that the representation $\widetilde{\rho_{\nabla_0}}:\piY\lra G$ associated to the origin $p=0$ via the construction above is precisely the Fuchsian representation associated to the fixed hyperbolic structure on $Y$.
\end{proof}

\begin{corollary}\label{Sigma_fixed_pts_in_Hitchin_base_of_X}
The Hitchin component $\Hit(\piY,G)$ is homeomorphic to the real vector space $\Fix_\Si(\cB_X(\fg))$. In particular, it is a contractible space.
\end{corollary}

\section{Invariant differentials}\label{inv_diff_section}

\subsection{Regular differentials on orbifolds}\label{subsection_on_differentials}

Assume first that the closed $2$-orbifold $Y$ is orientable, i.e.\ that its underlying topological surface $|Y|$ is orientable and $Y$ has only cone points as singularities. We denote by $g$ the genus of $|Y|$ and we shall sometimes say that \emph{$Y$ is an orientable orbifold of genus $g$}. If we fix an orbifold complex analytic structure on $Y$, then, as $Y$ has complex dimension $1$, there is an induced Riemann surface structure on the underlying compact surface $|Y|$ and we denote by $K_{|Y|}$ the canonical bundle of $|Y|$. Let us denote by $x_1, \dots, x_k$ the cone points of  $Y$, of respective orders $m_1, \dots, m_k \geqslant 2$. Given a point $x\in |Y|$, we denote by $\cL_x$ the point line bundle associated to the effective divisor $x$ and characterized as the holomorphic line bundle on $|Y|$ admitting a holomorphic section with a zero of order one at $x$ and no other zeros. Given integers $d,m \geqslant 2$, we define the number $O(d,m) := \left\lfloor d-\frac{d}{m}\right\rfloor$, where the brackets $\lfloor\dots\rfloor$ stand for integer part, and we consider the following holomorphic line bundle on $|Y|$:
\begin{equation*}\label{twisted_canonical_bundle_for_Y}
K(Y,d) := K_{|Y|}^d \otimes \bigotimes_{i=1}^k \cL_{x_i}^{O(d,m_i)}.
\end{equation*}
Holomorphic sections of $K(Y,d)$ can be seen as meromorphic $d$-differentials on $|Y|$ (sections of $K_{|Y|}^d$) with a pole of order at most $O(d,m_i)$ at the point $x_i$, and no other poles. The vector space of all such sections is denoted by
$H^0(Y,K(Y,d))$. We will say that a holomorphic section of $K(Y,d)$ is a \emph{regular $d$-differential} on $Y$. Such differentials are allowed to have poles of controlled order at the singular points of $Y$. Note that when $Y$ is a Riemann surface $X$ with trivial orbifold structure, then $K(Y,d) = K_X^d$.

Assume now that $Y$ is not orientable. We denote by $x_1,\,\ldots\, , x_k$ its cone points, of respective orders $m_1, \,\ldots\, , m_k \geqslant 2$, and by $y_1, \,\ldots\, , y_\ell$ its corner reflectors, of respective orders $n_1, \,\ldots\, , n_\ell \geqslant 2$. Denote by $Y^+$ its orientation double cover, an orientable orbifold equipped with a two-fold covering map $\eta:Y^+ \lra Y$ and a $\sfrac{\Z}{2\Z}$-action such that $(\sfrac{\Z}{2\Z}) \bs Y^+ = Y$. We denote by $u_i, v_i$ the two cone points of $Y^+$ in $\eta^{-1}(x_i)$, each of order $m_i$,  and by $w_j$ the cone point of $Y^+$ in $\eta^{-1}(y_{j})$, of order $n_j$. The $\sfrac{\Z}{2\Z}$-action sends $u_i$ to $v_i$ and fixes $w_j$, for all $i,j$. Note that $\chi(Y^+)= 2 \chi(Y)$ and $\chi(|Y^+|) = 2 \chi(|Y|)$. An \emph{orbifold dianalytic structure} on $Y$ can be defined as an orbifold complex analytic structure on $Y^+$ with $\sfrac{\Z}{2\Z}$-action given by an anti-holomorphic involution. In this case, the underlying topological surface $|Y|$ has a canonical Klein surface structure: the identification $\eta: |Y^+| \lra |Y| \simeq (\sfrac{\Z}{2\Z}) \bs  |Y^+|$ endows the topological space $|Y|$ with the subsheaf of $\eta_*\mathcal{O}_{|Y|^+}$ consisting of  $\sfrac{\Z}{2\Z}$-invariant holomorphic functions on $Y$ \cite{AG}. Moreover, the holomorphic line bundle $K(Y^+,d)$ on $|Y^+|$ has a canonical real structure (induced by the real structure $\tau:|Y^+| \lra |Y^+|$), so we can define its invariant Weil restriction $K(Y,d):= \Fix_\tau(\eta_*K(Y^+,d))$. This is a dianalytic line bundle on $|Y|$, for which one has $H^0(Y, K(Y,d)) \simeq \Fix_{\tau}(H^0(Y^+,K(Y^+,d)))$. We get in this way a uniform definition of $K(Y,d)$, which works both when $Y$ is orientable and when it is not. Indeed, if $Y$ is an orientable orbifold, then $Y^+=Y$, so $\eta$ and $\tau$ are trivial. 

\begin{definition}[Regular differentials]\label{reg_diff_def}
Let $Y$ be a closed $2$-orbifold, not necessarily orientable. Elements of the real vector space $H^0(Y,K(Y,d))$ will be called \emph{regular $d$-differentials on $Y$}.
\end{definition}

By definition, regular $d$-differentials are sections of the dianalytic line bundle $K(Y,d)$ on $Y$, which is a holomorphic line bundle if and only if the orbifold $Y$ is orientable. We now compute the dimension of the real vector space $H^0(Y,K(Y,d))$, starting with the orientable case.

\begin{lemma}  \label{prop:dimension diff orient} Let $d\geqslant 2$ be an integer.
For $Y$ orientable such that $\chi(Y)<0$, we have the formula:
$$\dim_\C H^0(Y,K(Y,d))  =  - \frac{1}{2} \chi(|Y|) (2d-1) + \sum_{i=1}^k O(d,m_i).$$
\end{lemma}

\begin{proof}
If $g$ is the genus of $|Y|$, the degree of $K(Y,d)$ is $$\deg K(Y,d)\, = \, \deg K_{|Y|}^d + \sum_{i=1}^{k} O(d,m_i) \,= \, 2d(g-1) + \sum_{i=1}^{k} O(d,m_i).$$ Now we claim that $\deg K(Y,d) > 2g-2$. To prove this, let us first note that for all $d,m \geqslant 2$, we have 
$O(d,m) \geqslant (d-1)\left(1 - \frac{1}{m} \right)$. Indeed, we can write $d = m Q + R$ with $1 \leqslant R \leqslant m$, so $O(d,m) = \left\lfloor m Q + R - \frac{m Q + R}{m} \right\rfloor = (m-1) Q + R - 1$ and $(d-1)\left(1 - \frac{1}{m} \right) =  (m-1) Q + R - 1 - \frac{R-1}{m}$. Then, by using the fact that $d\geqslant 2$ and $ - \chi(Y) = 2(g-1) + \sum_{i=1}^{k} (1-\tfrac{1}{m_i}) > 0$, we have:
$$
\deg K(Y,d) \,\geqslant\, 2d(g-1) + (d-1)\sum_{i=1}^{k} \left( 1 - \frac{1}{m_i} \right) \,=\, 2(g-1) -\chi(Y)(d-1) \,>\, 2(g - 1).   
$$ The Riemann-Roch theorem then gives $\dim_\C H^0(Y,K(Y,d)) - 0 = \deg(K(Y,d))+1-g$, hence the result.
\end{proof}

\begin{remark}\label{lemma:useful}
In the last step of the proof of Lemma \ref{prop:dimension diff orient}, we have obtained the following result, to be useful later: if $Y$ is an orientable orbifold of genus $g$, with $\chi(Y) < 0$, then $\dim_\C H^0(Y,K(Y,d)) \geqslant g$. 
\end{remark}

\begin{theorem} \label{prop:dimension diff non-orient} Let $d\geqslant 2$ be an integer, and let $Y$ be a closed connected $2$-orbifold, not necessarily orientable, such that $\chi(Y)<0$. Then
$$\dim_\R H^0(Y, K(Y,d)) = -\chi(|Y|)(2d-1) + 2 \sum_{i=1}^k O(d,m_i) + \sum_{j=1}^\ell O(d,n_j)$$
\end{theorem}
\begin{proof}
If $Y$ is orientable, this is Lemma \ref{prop:dimension diff orient}. Otherwise, it follows from Lemma \ref{prop:dimension diff orient} and the fact that $H^0(Y,K(Y,d))= \Fix_\tau H^0(Y^+,K(Y^+,d))$, with $\tau$ a $\C$-anti-linear involution.
\end{proof}

Let us fix an orbifold dianalytic structure on $Y$ and choose a presentation $Y\simeq [\Sigma \bs X]$ in the sense of Definition \ref{pres_of_Y_def}. We denote the projection by $\pi:X\lra Y$. Since $X$ is orientable, $\pi$ lifts to a map $\pi^+:X\lra Y^+$, where $Y^+$ is the orientation double cover of $Y$. We consider the pullback to $X$ of the orbifold complex dianalytic structure on $Y$: since $X$ has trivial orbifold structure, this is a Riemann surface structure in the usual sense. The map $\pi:X\lra |Y|$ is dianalytic and the map $\pi^+:X\lra |Y^+|$ is holomorphic. We now describe a natural identification between the space $H^0(Y,K(Y,d))$ of regular $d$-differentials on $Y$ and the space of $d$-differentials on $X$ which are invariant by the $\Sigma$-action defined in Section \ref{action_on_hol_diff}.

\begin{theorem} \label{thm:invariant differentials}
Choose a presentation $Y\simeq [\Sigma \bs X]$ and let $\pi:X\lra Y$ be the canonical projection. For every regular $d$-differential $q$ on $Y$, the pullback $\pi^*q$ is a $\Sigma$-invariant holomorphic $d$-differential on $X$, and the map $\pi^*: H^0(Y,K(Y,d)) \lra  \Fix_\Si\,H^0(X,K_X^d)$ thus defined is an isomorphism of real vector spaces.
\end{theorem}

\noindent Therefore, the formula in Theorem \ref{prop:dimension diff non-orient} also computes the dimension of $\Fix_\Sigma(H^0(X,K_X^d))$, independently of the chosen presentation $Y\simeq [\Si \bs X]$. The proof of Theorem \ref{thm:invariant differentials} rests on the following two lemmas, the first of which explains why the numbers $O(d,m)$ appear in the formula for $\dim\Hit(\piY,G)$ (Theorem \ref{theorem:dimension}).

\begin{lemma}  \label{lemma:order of poles}
Let $f: U \lra \mathbb{C}$ be a holomorphic function from an open neighborhood $U$ of $0 \in \mathbb{C}$. Assume that $f$ has a zero at $0$ of order $m$. Let $q$ be a meromorphic differential of degree $d$ on a neighborhood $V$ of $0 \in \mathbb{C}$, with a pole of order $s$ at $0$. Then the pullback $f^*q$ is holomorphic if and only if $s \leqslant O(d,m)$.
\end{lemma}

\begin{proof}
We may assume that $f(z)=z^m$ and $q(z) = \frac{h(z)}{z^s}(dz)^d$, where $h(z)$ is holomorphic and $h(0) \neq 0$. So $(f^*q)(z) = \frac{h(z^m)}{z^{ms}} (d(z^m))^d = m^d h(z^m) z^{d(m-1)-ms} (dz)^d$ and this is holomorphic if and only if $d(m-1)-ms \geqslant 0$, i.e.\ $s\leqslant O(d,m)$.
\end{proof}

\begin{lemma}\label{invariant_differentials_local_picture}
Assume that $f$ is the holomorphic map $z\lmt z^m$, from a small open disk $U$ centered at $0$ to the open disk $V:=f(U)$. The group $\sfrac{\Z}{m\Z}$ acts on $U$ via $z\lmt e^{i\frac{2\pi}{m}} z$, with quotient $(\sfrac{\Z}{m\Z})\backslash U \simeq V$. The map $f^*:H^0(V,K(V,d)) \lra H^0(U,K_U^d)$ induces an isomorphism $$H^0\big(V,K(V,d)\big) \simeq \Fix_{\sfrac{\Z}{m\Z}}\,  H^0\big(U,K_U^d\big).$$
\end{lemma}

\begin{proof}
The map $f:U\lra V$ is a surjective submersion, so $f^*$ defines an injective map $H^0(V,K(V,d)) \hookrightarrow H^0(U,K_U^d)$. Let us now prove that $\mathrm{Im}\,f^*$ is exactly $\Fix_{\sfrac{\Z}{m\Z}}\, H^0(U,K_U^d)$. 

A holomorphic differential  $\phi(z) (dz)^d\in H^0(U,K_U^d)$ is $\sfrac{\Z}{m\Z}$-invariant if and only if $\phi(e^{i\frac{2\pi}{m}} z) = e^{-i\frac{2\pi d}{m}}\phi(z)$ on $U$. We claim that this holds if and only if $\phi(z) = m^d h(z^m) z^{d(m-1)-ms}$ for some holomorphic function $h:V\lra \C$ and an integer $s\leq O(d,m)$. 

That this is indeed sufficient is readily checked. To prove that it is necessary, write $\phi(z) = \sum_{k\geq 0} a_k z^k$ for $|z|$ small enough. Then $e^{i\frac{2\pi d}{m}} \phi(e^{i\frac{2\pi}{m}}z) = \sum_{k\geq 0} a_k e^{i\frac{2\pi (k+d)}{m}} z^k$, which is equal to $\phi(z)$ if and only if $a_k=0$ every time $k+d\neq mq$ for some integer $q$. Thus, $$\phi(z) = \sum_{q\geq \left\lceil{\frac{d}{m}}\right\rceil} a_{mq-d} z^{mq-d} = \frac{1}{z^d} \sum_{q\geq \left\lceil{\frac{d}{m}}\right\rceil} b_{q} z^{mq},$$ with $b_q:=a_{mq-d}$. Let $q_0\geq 0$ be the smallest non-negative integer such that $b_{\left\lceil{\frac{d}{m}}\right\rceil + q_0} \neq 0$ and set $s_0:= O(d,m)-q_0$. Since $O(d,m) := \left\lfloor d -\frac{d}{m}\right\rfloor = d -\lceil\frac{d}{m}\rceil$, one has, for all integer $r\geq q_0$, 
\begin{equation}\label{inequality_to_check_holomorphicity}
\left(\left\lceil{\frac{d}{m}}\right\rceil +r\right) -d + s_0 = (r-q_0)+ \left\lceil{\frac{d}{m}}\right\rceil -d + O(d,m) = r-q_0 \geq 0.
\end{equation} So, to have $z^d\phi(z) = m^d h(z^m) z^{m(d-s_0)}$ with $h$ holomorphic on $V\simeq (\sfrac{\Z}{m\Z})\backslash U$, it suffices to set $$h(z^m) := \frac{z^d\phi(z) }{m^d z^{m(d-s_0)}}  = \frac{1}{m^d} \sum_{q\geq \left\lceil{\frac{d}{m}}\right\rceil +q_0} \frac{b_q z^{mq}}{z^{m(d-s_0)}} = \frac{1}{m^d} \sum_{q\geq \left\lceil{\frac{d}{m}}\right\rceil +q_0} b_q (z^m)^{q-d+s_0},$$ which is indeeed a holomorphic function of $z^m$ in view of \eqref{inequality_to_check_holomorphicity}. One may also consult \cite{Lewittes} for similar computations.

 This proves that if $\phi(z) (dz)^d$ is $(\sfrac{\Z}{m\Z})$-invariant, it is indeed the pullback, under $f:U\lra V$, of a holomorphic section of $K(V,d)$, namely the meromorphic differential $\frac{h(z)}{z^{s_0}} (dz)^d$ on $V$, where the integer $s_0\leq O(d,m)$ is defined as above using $\phi(z)$.
\end{proof}

\begin{proof}[Proof of Theorem \ref{thm:invariant differentials}]
Let $\eta:Y^+\lra Y$ be the orientation double cover of $Y$ and denote by $\Si^+ \lhd \Si$ the group of all holomorphic transformations of $X$ contained in $\Si$. Then $[\Si^+\bs X] \simeq Y^+$ is a presentation of $Y^+$. 

Recall that $\eta$ induces an isomorphism $\eta^*:H^0(Y,K(Y,d)) \simeq \Fix_{\sfrac{\Z}{2\Z}}\, H^0(Y^+,K(Y^+,d))$, where $\sfrac{\Z}{2\Z}$ acts with respect to the canonical real structure of the holomorphic line bundle $K(Y^+,d)$, and let us first prove the result of the theorem for $\pi^+: X\lra  Y^+$ instead of $\pi:X\lra Y$. In this case, sections of $K(Y^+,d) $ are meromorphic sections of $K_{|Y^+|}^{\, d}$ with poles of order at most $O(d,m_i)$ at the cone points.  

We claim that the pullback of regular $d$-differentials under $\pi^+$ induces an isomorphism of complex vector spaces 
$$(\pi^+)^*: H^0(Y^+,K(Y^+,d)) \rightarrow \Fix_{\Si^+}\,H^0(X,K_X^d). $$
First, the map is well-defined: since $|Y^+|\simeq \Si^+\bs X$, the pullback of a differential on $|Y^+|$ is a $\Si^+$-invariant differential on $X$. Moreover, the pullback of a section of $K(Y^+,d)$ is holomorphic by Lemma \ref{lemma:order of poles}. Second, the map $(\pi^+)^*$ is injective, because $\pi^+$ is a surjective submersion. To verify that $(\pi^+)^*$ is surjective, we construct an inverse. Let $\varphi$ be a $\Si^+$-invariant holomorphic differential of degree $d$ on $X$. The map $\pi^+: X \lra |Y^+|$ is ramified exactly over the cone points of $Y^+$. Denote by $X^0$ the Riemann surface $X$ minus the ramification points of $\pi^+$. The restricted map $\pi^+|_{X^0} : X^0 \rightarrow |Y^+|^0$ is a Galois covering. So, since $\varphi$ is $\Si^+$-invariant, there exists a holomorphic differential $\psi$ on $|Y^+|^0$ such that $(\pi^+)^*(\psi)=\varphi$. By Lemma \ref{invariant_differentials_local_picture}, $\psi$ admits a meromorphic extension to $|Y^+|$ and the extension thus defined is a section of $K(Y^+,d)$.

Finally, the action of $\Si$ on $H^0(X,K_X^d)$ induces an action of the group $\Si/\Si^+\, \simeq \sfrac{\Z}{2\Z}$ on $\Fix_{\Si^+}\,H^0(X,K_X^d)$, which coincides with the canonical real structure of the complex vector space $\Fix_{\Si^+}\,H^0(X,K_X^d)$, thus yielding an isomorphism of real vector spaces $\Fix_{\sfrac{\Z}{2\Z}}\, H^0(Y^+,K(Y^+,d)) \simeq  \Fix_{\Sigma}\, H^0(X,K_X^d)$. Since $\pi = \eta \circ \pi^+$, the proof is complete.
\end{proof}

\subsection{The dimension of Hitchin components}

Let $G=\Int(\fgC)^\tau$, where $\fg$ is a split real form of a complex simple Lie algebra of rank $r$. Let $d_1, \dots, d_r$ be the exponents of $\fg$, as in Section \ref{subsec:Hitchin_fibration}. Let $Y$ be a closed orbifold with negative Euler characteristic, and choose an orbifold complex dianalytic structure on $Y$. We define the \emph{Hitchin base} of $Y$ to be the following real vector space (which is a complex vector space if and only if $Y$ is orientable): 
\begin{equation}\label{Hitchin_base_of_an_orbifold}\cB_Y(\fg) := \bigoplus_{\alpha=1}^r H^0(Y, K(Y,d_\alpha+1)),\end{equation} 
where $H^0(Y,K(Y,d))$ is the space of regular $d$-differentials on $Y$, as defined in Section \ref{subsection_on_differentials}. We can now state and prove the main result of this paper.

\begin{theorem}\label{theorem:dimension}
Let $Y$ be a closed connected orbifold of negative Euler characteristic. Denote by $d_1,\,\dotsc\, ,d_r$ the exponents of $\fg$, where $r:=\rk(\fg)$. The Hitchin component $\Hit(\piY,G)$ is homeomorphic to the Hitchin base $\cB_Y(\fg)$, which is a real vector space of dimension 
$$-\chi(|Y|)\dim G + \sum_{\alpha=1}^{r}  \left(2 \sum_{i=1}^k O(d_\alpha+1,m_i) + \sum_{j=1}^\ell O(d_\alpha+1,n_j)\right)$$
\end{theorem}

\begin{proof}
By Corollary \ref{Sigma_fixed_pts_in_Hitchin_base_of_X} and Theorem \ref{thm:invariant differentials}, we have $\Hit(\piY,G) \simeq\Fix_\Si(\cB_X(\fg)) \simeq \cB_Y(\fg)$. Therefore, $\dim\Hit(\piY,G) = \sum_{\alpha=1}^r \dim H^0(Y,K(Y,d_\alpha+1))$. The result is then obtained by summing over $\alpha$, using Theorem \ref{prop:dimension diff non-orient} and the following well-known formula, a consequence of \eqref{equation:decomposition_into_irreducible}:
\begin{equation}   \label{form:exponents dimension}
\sum_{\alpha=1}^{r} (2 d_\alpha + 1) = \dim G.\qedhere
\end{equation}
\end{proof}

\begin{corollary}\label{dimension_bdry_case}
Let $Y$ be a compact connected orbifold with boundary, of negative Euler characteristic, and let $b$ be the number of boundary components of $Y$ that are full $1$-orbifolds. The Hitchin component $\Hit(\piY,G)$ is homeomorphic to a real vector space of dimension 
$$-\chi(|Y|)\dim G + \sum_{\alpha=1}^{r}  \left(2 \sum_{i=1}^k O(d_\alpha+1,m_i) + \sum_{j=1}^\ell O(d_\alpha+1,n_j) \, +\,  2 b\, \left\lfloor \frac{d_\alpha+1}{2}\right\rfloor \right)$$
\end{corollary}

\begin{proof}
By Proposition \ref{Hitchin_comp_for_orbifold_with_boundary} and Theorem \ref{theorem:dimension}, $\Hit(\piY,G)$ is homeomorphic to the Hitchin base $\cB_{mY}(\fg)$, where $mY$ is the closed orbifold with mirror boundary associated to $Y$. Compared to $Y$, the closed orbifold $mY$ has an extra $2b$ corner reflectors, each one of order $2$, so the formula for the dimension follows from Theorem \ref{theorem:dimension}.
\end{proof}

\noindent In Appendix \ref{tables_appendix}, we include a list of the exponents of simple complex Lie algebras which we borrow from \cite{Damianou}. Hereafter, we give a few consequences of Theorem \ref{thm:invariant differentials}, Theorem \ref{theorem:dimension} and Table \ref{table:exponents}.

\begin{corollary} \label{cor:non orient half}
Let $Y^+\lra Y$ be an orientation double cover. Then
$$\dim \Hit(\piY,G) = \frac{1}{2} \dim \Hit(\pi_1 Y^+,G).$$
\end{corollary}

\begin{corollary} \label{cor:Sp(2m)=O(m,m+1)}
For every orbifold $Y$, $\dim \Hit(\piY, \PSp^{\pm}(2m,\R)) = \dim \Hit(\pi_1 Y,\PO(m,m+1))$.
\end{corollary}

Finally, we give an alternate formula for the dimension of $\Hit(\piY,G)$, more similar to the ones given by Thurston \cite{Thurston_notes} and Choi and Goldman \cite{CG} for $G=\PGL(2,\R)$ and $\PGL(3,\R)$.

\begin{theorem}\label{alternate_formula_for_dim}
Under the assumptions of Theorem \ref{theorem:dimension}, set $k_m := \#\{i \mid m_i = m \}$, $\ell_n := \#\{j \mid n_j = n \}$ and $M:=\max_{1\leqslant \alpha\leqslant r}d_\alpha$. Then the dimension of $\Hit(\piY,G)$ can also be written as follows:
$$
- \chi(|Y|) \dim G + \frac{1}{2}(\dim G - r) \left( 2k + \ell \right) - 2 \sum_{m=2}^{M} \left( \sum_{\alpha = 1}^{r} \left\lceil \frac{d_\alpha + 1}{m} - 1 \right\rceil \right) k_m -  \sum_{n=2}^{M} \left( \sum_{\alpha = 1}^{r} \left\lceil \frac{d_\alpha + 1}{n} - 1 \right\rceil \right) \ell_n
$$ where $\lceil x\rceil$ denotes the lowest integer not smaller than $x$.
\end{theorem}

\begin{proof}
Recall that $O(d,m)=\lfloor d- \frac{d}{m} \rfloor$. Using the relation $\left\lfloor d - \frac{d}{m} \right\rfloor = (d-1) - \left\lceil \frac{d}{m} - 1 \right\rceil$ in Theorem \ref{prop:dimension diff non-orient}, we obtain:
$
\dim H^0(Y,K(Y,d)) = - \chi(|Y|) (2d-1) + (d-1)(2k+\ell) - 2 \sum_{m=2}^{d-1} \left\lceil \frac{d}{m} - 1 \right\rceil k_m -  \sum_{n=2}^{d-1} \left\lceil \frac{d}{n} - 1 \right\rceil \ell_n.
$
The rest of the proof is then the same as that of Theorem \ref{theorem:dimension}.
\end{proof}

\begin{corollary}\label{alternate_formula_for_dim_bdry_case}
Under the assumptions of Corollary \ref{dimension_bdry_case}, the dimension of $\Hit(\piY,G)$ can also be written
\begin{center}
$\displaystyle
- \chi(|Y|) \dim G + \frac{1}{2}(\dim G - \rk\,G) \left( 2k + \ell +2b \right) - 2 \sum_{m=2}^{M} \left( \sum_{\alpha = 1}^{r} \left\lceil \frac{d_\alpha + 1}{m} - 1 \right\rceil \right) k_m 
$
\\
$\displaystyle
 -  \sum_{n=2}^{M} \left( \sum_{\alpha = 1}^{r} \left\lceil \frac{d_\alpha + 1}{n} - 1 \right\rceil \right) \ell_n\, - 2b  \sum_{\alpha = 1}^{r} \left\lceil \frac{d_\alpha - 1}{2} \right\rceil.
$
\end{center}
\end{corollary}

\begin{proof}
This again follows from applying the formula of Theorem \ref{alternate_formula_for_dim} to the closed orbifold $mY$, replacing $\ell$ and $\ell_2$ respectively by $\ell+2b$ and $\ell_2+2b$.
\end{proof}

For instance, when $G=\PGL(3,\R)$, we obtain $\dim\Hit(\pi_1 Y,\PGL(3,\R)) = -8\chi(|Y|) + (6k-2k_2) + (3\ell-\ell_2) + 4b$, the same formula as in \cite[p.1069]{CG}.

\begin{corollary}\label{corollary:dimension}
Assume that the orders of all cone points and corner reflectors of $Y$ are greater than the biggest exponent of $G$. Then $\dim \Hit(\piY,G) = - \chi(|Y|) \dim G + \frac{1}{2}(\dim G - \rk\,G) \left( 2k + \ell \right)$. 
\end{corollary}

\subsection{Approximation formula}

The following corollary of Theorem \ref{theorem:dimension} shows that Hitchin's formula remains valid when $Y$ is a non-orientable surface or an orbifold having only mirror points as singularities.

\begin{corollary}\label{cases_of_validity_of_Hitchin_s_formula}
If $Y$ is an orbifold without cone points and corner reflectors (i.e.\ $k=\ell=0$), then
$\dim\Hit(\piY,G) = -\chi(Y) \dim G = -\chi(|Y|)\dim G$.
\end{corollary}

The formula in Corollary \ref{cases_of_validity_of_Hitchin_s_formula} cannot hold in general, because the orbifold Euler characteristic is usually a rational number. However, this formula gives a good approximation for $\dim\Hit(\piY,G)$.

\begin{proposition}\label{proposition:dim_estimate} Let $r(Y,G) := -\chi(Y) \dim G - \dim\Hit(\piY,G)$. Then the following estimate holds:
$$ - \rk(G)\left(k+\frac{\ell}{2}\right)\, \leqslant\, r(Y,G) \,\leqslant\, \frac{3}{2}\rk(G) \left(k+\frac{\ell}{2}\right).$$
More precisely,
$$-\rk(G) \left(\sum_{i=1}^k \left( 1-\frac{1}{m_i} \right) + \frac{1}{2}\sum_{j=1}^\ell \left( 1-\frac{1}{n_j} \right) \right)  \leqslant  r(Y,G) \,\leqslant\,   \rk(G) \left(\sum_{i=1}^k \left( 1+\frac{1}{m_i} \right) + \frac{1}{2}\sum_{j=1}^\ell \left( 1+\frac{1}{n_j} \right) \right).$$
\end{proposition}
\begin{proof}
Write the quantity $r(Y,G)$ using (\ref{form:def euler}), (\ref{form:exponents dimension}) and Theorem \ref{theorem:dimension}. Then use the inequality 
\begin{equation}
 0 \leqslant (d+1) \left( 1-\frac{1}{m}  \right) - \left \lfloor (d+1)\left( 1-\frac{1}{m}  \right) \right \rfloor < 1. \qedhere
\end{equation}
\end{proof}

It is worth noting that, in the families of classical Lie groups, the dimension of the group grows quadratically with the rank, so the estimate is asymptotically good.

\begin{remark}
When $H$ is a split simple real algebraic group and $Y$ is orientable, Larsen and Lubotzky \cite{Larsen_Lubotzky} gave an asymptotic estimate of the dimension of $\mathrm{Hom}^{\mathrm{epi}}(\piY,H)$, which by definition is the Zariski-closure, in $\mathrm{Hom}(\piY,H)$, of the set of Zariski-dense representations $\piY \lra H$. More precisely, one has $ \dim \mathrm{Hom}^{\mathrm{epi}}(\piY,H)/H = -\chi(Y) \dim (H) + O(\rk\, H)$.
But $\mathrm{Hom}^{\mathrm{epi}}(\piY,H)$ is not always comparable with Hitchin components: in Theorem \ref{thm_none_is_zariski_dense}, we classify the Hitchin components (for orbifold groups) that contain no Zariski-dense representations.
\end{remark}

\section{Applications}\label{section:application}
 
In this section, we discuss some new rigidity phenomena that cannot be observed with ordinary surface groups and we classify Hitchin components of dimensions $0$, $1$ and, for orientable orbifolds, $2$. Our results also have applications to the study of Hitchin components of surface groups, the theory of Higgs bundles, and the pressure metric. Finally, we describe certain connected components of deformation spaces of real projective structures on Seifert fibered $3$-manifolds. 

\subsection{Rigidity phenomena}   \label{subsec:rigidity}

An interesting feature of representations of orbifold groups is that they give examples of rigidity phenomena. 
A first type of rigidity is given by $0$-dimensional Hitchin components. To discuss this, we shall assume that the orbifold $Y$ is closed and orientable, as rigidity for non-orientable orbifolds can be deduced immediately from the orientable case by Corollary \ref{cor:non orient half}. For the target group $\PGL(2,\R)$, Thurston showed that $\dim\Hit(\piY,\PGL(2,\R))=0$ if and only if $Y$ has genus $0$ and $3$ cone points. And for $\PGL(3,\R)$, Choi and Goldman showed in \cite{CG} that  $\dim\Hit(\piY,\PGL(3,\R))=0$ if and only if $Y$ has genus $0$ and $3$ cone points, and one of those cone points is of order $2$. We now complete the classification of $0$-dimensional Hitchin components for general $G$, by combining Theorem \ref{theorem:dimension} and Lemma~\ref{lemma:genus zero}, which gives the list of orientable orbifolds with vanishing regular $d$-differentials.

\begin{lemma}\label{lemma:genus zero} 
Let $Y$ be an orientable orbifold of genus $g$ with $k$ cone points, of respective orders $m_1, \,\dotsc\,, m_k$, and assume that $\chi(Y)<0$. Then $\dim_\C H^0(Y,K(Y,d)) = 0$ if and only if $g=0$ and $k$, $d$ and the $m_i$s are as in the list in Table \ref{table:vanishing_diff}.
\end{lemma}
\begin{proof}
Assume that $Y$ is an orientable orbifold of arbitrary genus $g$ such that $\dim_\C H^0(Y,K(Y,d)) = 0$. First, by Remark \ref{lemma:useful}, $g=0$. Hence $k \geqslant 3$, by negativity of the Euler characteristic. Second, using Lemma \ref{prop:dimension diff orient} and the inequality $O(d,m) \geqslant O(d,2)$, we have $0 = \dim_\C H^0(Y,K(Y,d)) \geqslant 1 - 2d + k \left\lfloor \frac{d}{2} \right\rfloor$. 

When $d$ is even, say $d = 2 \delta$, this implies that $0 \geqslant 1 - 4\delta + k \delta$, so $k = 3$. And when $d$ is odd, say $d = 2 \delta + 1$, it implies that $(k-4)\delta \leqslant 1$, therefore $k \leqslant 5$, and (\emph{i}) if $k=5$, then $\delta = 1$, so $d = 3$ and (\emph{ii}) if $k=4$, then by negativity of the Euler characteristic, at least one of the $m_i$ is $\geqslant 3$, so we have
$$  0 =  \dim_\C H^0(Y,K(Y,d)) \geqslant 1 -2 d + 3\,O(d,2) + O(d,3) =  \left\lfloor \frac{\delta-1}{3} \right\rfloor,$$
which implies that $\delta \leqslant 3$, so $d=3, 5$ or $7$. By an easy but long computation, we obtain Table \ref{table:vanishing_diff}.
\end{proof}

\begin{theorem}   \label{thm:dimension 0}
Let $Y$ be an orientable orbifold of negative Euler characteristic. If $\dim\Hit(\piY,G)=0$, then $Y$ is a sphere with $3$ cone points. Conversely, let $Y$ be a sphere with $3$ cone points of respective orders $m_1 \leqslant m_2 \leqslant m_3$
and assume that the tuple $(G,m_1,m_2,m_3)$ satisfies one of the following conditions:
\begin{enumerate}
\item $G=\PGL(2,\R) \simeq \PO(1,2) \simeq \PSp^\pm(2,\R)$ and $\frac{1}{m_1}+\frac{1}{m_2}+\frac{1}{m_3}<1$.
\item $G=\PGL(3,\R)$, $m_1=2$ and $\frac{1}{m_2}+\frac{1}{m_3}<\frac{1}{2}$.
\item $G=\PGL(4,\R)\simeq \PO^\pm(3,3)$, $\PGL(5,\R)$, $\PSp^{\pm}(4,\R)\simeq\PO(2,3)$, $m_1 = 2$, $m_2 = 3$ and $m_3\geqslant7$.
\item $G=\PSp^{\pm}(4,\R)\simeq\PO(2,3)$, $m_1=3$, $m_2 = 3$ and $m_3 \geqslant4$.
\item $G=\GG$, $m_1=2$ and $m_2=4$ or $5$, and $m_3=5$.
\end{enumerate} Then $\dim\Hit(\piY,G) =0$, so any two Hitchin representations of $\piY$ in $G$ are $G$-conjugate in this case, and this happens for infinitely many orbifolds.

Moreover, for all other pairs $(G,Y)$ with $Y$ orientable, Hitchin representations of $\piY$ into $G$ admit non-trivial deformations, i.e.\ $\dim\Hit(\piY,G)>0$.
\end{theorem} 
In particular, if $G$ is one of the following Lie groups, then $\dim\Hit(\piY,G)>0$:
\begin{enumerate}
\item $G=\PGL(n,\R)$ with $n\geqslant 6$,
\item $G=\PSp^\pm(2m,\R)$ or $\PO(m,m+1)$ with $m\geqslant 3$,
\item $G=\PO^\pm(m,m)$ with $m\geqslant 4$,
\item $G$ is an exceptional Lie group and $G\neq \GG$.
\end{enumerate}
\begin{proof}[Proof of Theorem \ref{thm:dimension 0}]
If $\dim\Hit(\piY,G)=0$, then $\dim H^0(Y,K(Y,2)) = 0$, so, by Lemma \ref{lemma:genus zero}, $Y$ has genus $0$ and $3$ cone points. The various statements are consequences of Theorem \ref{theorem:dimension}, Lemma \ref{lemma:genus zero} and Table \ref{table:vanishing_diff}.
\end{proof}

Orbifold groups also give us examples of a second type of rigidity, namely, we find Hitchin components for orbifolds that contain no Zariski-dense representations. This contrasts with what happens for surface groups, for which the subset of Zariski dense representations is always dense in the Hitchin component, see \cite{Guichard_Zariski}. 
But for some orbifolds $Y$ and groups $G$, there exists a proper closed subgroup $H < G$ such that the Zariski closure of \emph{every} Hitchin representation $\rho:\piY\lra G$ lies in a conjugate of $H$. In particular, the image of a Hitchin representation of $\piY$ into $G$ can never be Zariski-dense. We will now classify all triples $(Y,G,H)$ with that property. To do so, we use the following result due to Guichard \cite{Guichard_Zariski}: let $G =\Int(\fg_\C)^\tau$ and denote by $G_0$ the identity component of $G$. If $X$ is a closed orientable surface, $\rho:\piX \lra G_0$ is a Hitchin representation, and $H_\rho$ is the identity component of the Zariski closure of $\rho(\piX)$ in $G$, then the inclusion $H_\rho\hookrightarrow G_0$ is conjugate to one of the following:
\begin{enumerate} 
\item The principal representation $\kappa: \PSL(2,\R) \hookrightarrow G_0$.
\item The standard inclusions $\PSp(2n,\R) \hookrightarrow \PSL(2n,\R)$, $\PSO(n,n+1) \hookrightarrow \PSL(2n+1,\R)$, or $\PSO(n-1,n)\hookrightarrow \PSO(n,n)$.
\item The standard inclusions $(\GG)_0 \hookrightarrow \PSL(7,\R)$ or $(\GG)_0 \hookrightarrow \PSO(3,4)$.
\item The identity $G_0 \longrightarrow G_0$.
\end{enumerate} 
In all those cases, $H_\rho$ is the identity component of a group $H =\Int(\fh_\C)^\tau$, for a simple Lie algebra $\fh_\C \subset \fg_\C$, and the representation $\rho$ is a Hitchin representation in $H$. Moreover, the inclusions $H\hookrightarrow G$ induce injective maps $\Hit(\piX,H)\hookrightarrow\Hit(\piX,G)$. Let now $Y$ be an orbifold (with presentation $[\Si\bs X]$), let $\rho:\piY\lra G$ be a Hitchin representation, and let us denote by $H_\rho$ the neutral component of the Zariski closure of $\rho(\piY)$ in $G$. Since $\rho(\piX)$ is a normal subgroup of finite index of $\rho(\piY)$ and $G$ is centerless, one has $H_\rho=H_{\rho|_{\piX}}$, hence we can apply Guichard's classification to orbifold groups.

\begin{theorem} \label{thm_none_is_zariski_dense}
Let $Y$ be an orientable orbifold of negative Euler characteristic and of genus $g$, with $k$ cone points of respective orders $m_1 \leqslant\, \ldots\, \leqslant m_k$, and let $H$ be a proper subgroup of $G$. If the map $\Hit(\piY,H)\lra \Hit(\piY,G)$ is surjective, then $Y$ has genus $0$. This happens if and only if the triple $(Y,G,H)$ is as follows:
\begin{enumerate}[(a)]
\item The inclusion $\Hit(\piY,\PGL(2,\R)) \hookrightarrow \Hit(\piY,G)$ 
is surjective if and only if one of the following two possibilities hold:
\begin{enumerate}[(1)]
\item Both spaces have dimension $0$. By Theorem \ref{thm:dimension 0}, there are infinitely many orbifolds with this property.
\item $G = \PGL(3,\R)$, $g=0$, $k=4$ and $m_1=m_2=m_3=2$ and $m_4\geqslant 3$, or $k=5$ and all $m_i = 2$.
\end{enumerate}
\item The inclusion $\Hit(\piY,\PSp^{\pm}(2n,\R)) \hookrightarrow \Hit(\piY, \PGL(2n,\R))$ is surjective if and only if $g=0$ and one of the following possibilities hold:
\begin{enumerate}[(1)]
\item $k=3$, $2 \leqslant n \leqslant 10$, $m_1=2$, and:
\begin{enumerate}
\item If $n = 3$, then $m_2 = 3$ and $m_3 \geq 7$, or $m_2=4$ and $m_3\geq 5$.
\item If $n = 4$, then $m_2 = 3$ and $m_3 \geq 7$, or $(m_2,m_3) = (4,5)$, $(4,6)$.
\item If $n = 5$, then $(m_2,m_3) = (3,7)$, $(3,8)$, $(4,5)$ or $(4,6)$.
\item If $n = 6$ or $7$, then $(m_2,m_3) = (3,7)$, $(3,8)$ or $(4,5)$.
\item If $n = 8$, $9$ or $10$, then $(m_2,m_3) = (3,7)$.
\end{enumerate}
\item $k=4$, $2 \leqslant n \leqslant 4$, $m_1=m_2=m_3=2$ and $m_4\geqslant 3$. Moreover: 
\begin{enumerate}
\item If $n = 3$, then $m_4 = 3$ or $4$.
\item If $n = 4$, then $m_4 = 3$.
\end{enumerate}
\item $k=5$, $n=2$ and all $m_i = 2$.
\end{enumerate}
The above also holds for the inclusion $\Hit(\piY,\PO(n-1,n)) \hookrightarrow \Hit(\piY, \PGL(2n-1,\R))$.
\item The inclusion $\Hit(\piY,\PO(n-1,n)) \hookrightarrow \Hit(\piY, \PO^{\pm}(n,n))$ is surjective if and only if $H^0(Y,K(Y,n))$ has dimension $0$. By Table \ref{table:vanishing_diff}, there are infinitely many orbifolds with this property, but this occurs only for $n\leqslant 43$.
\item The inclusion $\Hit(\piY,\GG) \hookrightarrow \Hit(\piY, \PGL(7,\R))$ is surjective if and only if $g=0$, $k=3$, $m_1=2$, $m_2=3$, and $m_3\geqslant 7$.
\item The inclusion $\Hit(\piY,\GG) \hookrightarrow \Hit(\piY, \PO(3,4))$ is surjective if and only if $g=0$, $k=3$, and $m_1=2$, $m_2=3$ and $m_3\geqslant 7$, or $m_1=3$, $m_2=3$ and $m_3\geqslant 4$.
\end{enumerate} In particular, if a representation $\rho\in \Hit(\piX,G)$ extends to $\piY$ for an orbifold $Y\simeq [\Si\bs X]$, then $\rho(\piX)$ is not Zariski-dense in $G$.
\end{theorem}

\begin{proof}
The map $\Hit(\piY,H)\lra \Hit(\piY,G)$ induced by an inclusion $H<G$ can be described explicitly by the differentials appearing in the Hitchin base of the groups $H$ and $G$ and, if an inclusion is surjective, it means that the space of differentials appearing in the Hitchin base of $G$ but not in the one of $H$ must have dimension $0$. From Lemma \ref{lemma:genus zero}, we see that $Y$ has genus $0$. Then, the various statements are immediate consequences of Theorem \ref{theorem:dimension}, Lemma \ref{lemma:genus zero} and Table \ref{table:vanishing_diff}.
\end{proof}

\subsection{Geodesics for the pressure metric}\label{subsection_on_geodesics}

When $X$ is a closed orientable surface, there are two known $\Out(\piX)$-invariant Riemannian metrics on the Hitchin components $\Hit(\piX, \PGL(n,\R))$: the \emph{pressure metric} of \cite{BCLS_pressure} and the \emph{Liouville pressure metric} of \cite{BCLS_liouville}. By restriction, they give $\Out(\piX)$-invariant Riemannian metrics on $\Hit(\piX, \PSp^{\pm}(2n,\R))$, $\Hit(\piX, \PO(n,n+1))$ and $\Hit(\piX, \GG)$. In the special case of $\Hit(\piX, \PGL(3,\R))$, there is also another invariant Riemannian metric, the \emph{Li metric} \cite{Li_metric}. Moreover, for all the groups of rank 2 ($\Hit(\piX, \PGL(3,\R))$, $\Hit(\piX, \PSp^{\pm}(4,\R))$ and $\Hit(\piX, \GG)$) a different invariant Riemannian metric, that is also K\"ahler, was defined by Labourie \cite{Labourie_cyclic_surfaces}. All these Riemannian metrics restrict to the Weil-Petersson metric on Teichm\"uller space, and little is known about their geometric properties, for instance about their geodesics. If $Y$ is a closed orbifold and $Y \simeq [\Si \bs X]$ is a presentation, we obtain the following information directly from Theorem \ref{main_obs_about_Hit_comp_for_orbifolds}.

\begin{proposition}\label{tot_geod_submanifold}
Fix any $\Out(\piX)$-invariant Riemannian metric on $\Hit(\piX,G)$. Then there is a totally geodesic embedding $\Hit(\piY,G) \overset{\simeq}{\lra} \Fix_{\Si}(\Hit(\piX,G))\subset \Hit(\piX,G)$. In particular, if $\dim\Hit(\piY,G)=1$, then $\Hit(\piY,G)$ embeds onto a geodesic of $\Hit(\piX,G)$.
\end{proposition}
\begin{proof}
The group $\Si$ acts on $\Hit(\piX,G)$ as a subgroup of $\Out(\piX)$, which acts on $\Hit(\piX,G)$ by isometries. It is a basic fact that a fixed point set of a subgroup of isometries is totally geodesic.
\end{proof}

In order to give explicit examples of geodesics, we classify orbifolds of negative Euler characteristic with $1$-dimensional Hitchin components. By Theorem \ref{theorem:dimension}, such an orbifold is necessarily non-orientable, so, by Corollary \ref{cor:non orient half}, it suffices to classify orientable orbifolds $Y$ such that $\dim\Hit(\piY,G)=2$. 

\begin{theorem}   \label{thm:dimension 2}
Let $Y$ be an orientable orbifold of negative Euler characteristic. Let $g$ be the genus of $Y$ and let $k$ be the number of cone points, of respective orders $m_1 \leqslant \, \ldots \leqslant m_k$. If $\dim\Hit(\piY,G)=2$, then $Y$ is of one of the following types: a sphere with $3$ cone points, a sphere with $4$ cone points, a torus with $1$ cone point.
Conversely, assume that the pair $(G,Y)$ satisfies one of the following assumptions:
\begin{enumerate}
\item $Y$ is a torus with $1$ cone point or a sphere with $4$ cone points and $G=\PGL(2,\R) \simeq \PO(1,2) \simeq \PSp^\pm(2,\R)$.
\item $Y$ is a sphere with $4$ cone points, $G=\PGL(3,\R)$, $m_1=m_2=m_3=2$ and $m_4\geqslant 4$.
\item $Y$ is a sphere with $3$ cone points and the tuple $(G,m_1,m_2,m_3)$ satisfies one of the following conditions:
\begin{enumerate}
\item $G=\PGL(3,\R)$, $m_1\geqslant 3$, $m_2\geqslant 3$ and $m_3\geqslant 4$. 
\item $G=\PGL(4,\R) \simeq \PO^\pm(3,3)$, $m_1 = 2$, $m_2 \geqslant 4$ and $m_3 \geqslant 5$, or $m_1 = m_2 = 3$ and $m_3\geqslant 4$.
\item $G=\PGL(5,\R)$, $m_1 = m_2 = 3$ and $m_3 = 4$, or $m_1 = 2$, $m_2 = 4$ and $m_3\geqslant 5$.
\item $G=\PGL(6,\R)$ or $\PGL(7,\R)$, $m_1=2$, $m_2=3$ and $m_3\geqslant 7$, or $(m_1,m_2,m_3)=(2,4,5)$.
\item $G=\PGL(n,\R)$ with $n=8,9,10,11$ and $(m_1,m_2,m_3)= (2,3,7)$.
\item $G=\PSp^\pm(4,\R) \simeq \PO(2,3)$, $m_1=2$, $m_2=4$ and $m_3\geqslant 5$, or $m_1=3$ and $m_2,m_3 \geqslant 4$.
\item $G=\PSp^\pm(6,\R)$ or $\PO(3,4)$, $m_1 = 2$, $m_2 = 3$ and $m_3\geqslant 7$, or $(m_1,m_2,m_3)$ is one of the following triples: $(2,4,5)$, $(2,5,5)$, $(3,3,4)$ or $(3,3,5)$.
\item $G=\PSp^\pm(8,\R)$, $\PSp^\pm(10,\R)$, $\PO(4,5)$ or $\PO(5,6)$, and $(m_1,m_2,m_3)=(2,3,7)$.
\item $G=\PO^\pm(4,4)$, $m_1 = 2$, $m_2 = 3$ and $m_3\geqslant 7$, or $(m_1,m_2,m_3)$ is one of the following triples: $(3,3,4)$ or $(3,3,5)$.
\item $G=\PO^\pm(5,5)$ and $(m_1,m_2,m_3)=(2,3,7)$.
\item $G=\GG$, $m_1=2$, $m_2=3$ and $m_3\geqslant 7$, or $m_1=2$, $m_2=4$ or $5$ and $m_3\geqslant 6$, or $(m_1,m_2,m_3)$ is one of the following triples: $(3,3,4)$, $(3,3,5)$, $(3,4,4)$, $(3,4,5)$ or $(3,5,5)$.
\end{enumerate}
\end{enumerate}
Then $\dim\Hit(\piY,G)=2$. For all other pairs $(G,Y)$ with $Y$ orientable, one has $\dim\Hit(\piY,G) \neq 2$.
\end{theorem} 

In particular, if $Y$ is orientable and $G$ is one of the following Lie groups, then $\dim\Hit(\piY,G)>2$:
\begin{enumerate}
\item $G=\PGL(n,\R)$ with $n\geqslant 12$,
\item $G=\PSp^\pm(2m,\R)$ or $\PO(m,m+1)$ with $m\geqslant 6$,
\item $G=\PO^\pm(m,m)$ with $m\geqslant 6$,
\item $G$ is an exceptional Lie group and $G\neq \GG$.
\end{enumerate}

\begin{proof}[Proof of Theorem \ref{thm:dimension 2}]
From the Thurston and Choi-Goldman formulas (see \cite{Thurston_notes,CG} or Table \ref{explicit_formulas_in_low_rank}) for Hitchin components for $\PGL(2,\R)$ and $\PGL(3,\R)$, we can see the following:
\begin{enumerate}
\item $\dim\Hit(\piY,\PGL(2,\R))=2$ if and only if $g=1$ and $k=1$, or $g=0$ and $k=4$. 
\item $\dim\Hit(\piY,\PGL(3,\R))=2$ if and only if $g=0$, $k=4$, $m_1=m_2=m_3=2$ and $m_4 \geqslant 3$, or $g=0$, $k=3$, $m_1\geqslant 3$, $m_2\geqslant 3$ and $m_3\geqslant 4$. 
\end{enumerate} 

So let us assume that $G\neq \PGL(2,\R),\PGL(3,\R)$. In this case, if $g\geqslant 1$, or $g=0$ and $k\geqslant 4$, then there is an even integer $d \geqslant 4$ such that both spaces of differentials of degree $2$ and of degree $d$ have complex dimension at least $1$. So we must have $g=0$ and $k=3$. The remaining statements are consequences of Theorem \ref{theorem:dimension}, Lemma \ref{lemma:genus zero} and Table \ref{table:exponents}.
\end{proof}

See Example \ref{example_of_geodesic} for explicit examples of geodesics in the $\PGL(n,\R)$-Hitchin component of the Klein quartic $\mathcal{K}$ (which, as $\mathcal{K}$ is of genus $3$, is homeomorphic to an open ball of dimension $4(n^2-1)$).

\subsection{Cyclic Higgs bundles}\label{cyclic_Higgs_bundles_section}

Given an orbifold $Y$ with a fixed complex analytic or dianalytic structure, Theorem~\ref{theorem:dimension} gives a parameterization of $\Hit(\piY,G)$ by the Hitchin base $\mathcal{B}_Y(\fg)$. When $G = \PGL(n,\R)$, we will say that a representation in $\Hit(\piY,\PGL(n,\R))$ is \emph{cyclic} if it is parametrized by $(0, \dots, 0, q_n)$, and $(n-1)$-\emph{cyclic} if it is parametrized by $(0, \dots, 0, q_{n-1},0)$. Similarly, if $G$ is one of the groups $\PSp^{\pm}(2m,\R), \PO(m,m+1)$, and $\GG$, a Hitchin representation in $G$ is \emph{cyclic} or $(n-1)$-\emph{cyclic} if the corresponding representation in $\PGL(n,\R)$ is. 

In the case of surface groups, these notions were introduced by Simpson \cite{Simpson_MCA}, Baraglia \cite{Baraglia_thesis} and Collier \cite{Collier_thesis}. Hitchin's equations for such representations take an especially simple form, which makes it possible to understand some of their geometric and asymptotic properties using analytic techniques, see \cite{Baraglia_thesis, Collier_thesis, CollierLi, DaiLi1, DaiLi2}. In case of an orbifold $Y$, given any presentation $Y\simeq [\Sigma\bs X]$, a representation in $\Hit(\piY,G)$ is cyclic or $(n-1)$-cyclic if and only if its image in $\Hit(\piX,G)$ is. In particular, we can transfer to orbifolds all the analytical tools and results that are valid for cyclic and $(n-1)$-cyclic representations for surfaces. 

Using Theorem \ref{main_obs_about_Hit_comp_for_orbifolds}, we can produce examples of orbifolds $Y$ and groups $G$ such that $\Hit(\piY,G)$ contains only cyclic or $(n-1)$-cyclic representations. This phenomenon never happens for surface groups. Thus, the results about cyclic or $(n-1)$-cyclic representations contained in \emph{loc.\ cit.}\ are valid for all points in the Hitchin component of such orbifolds, see Corollary \ref{comp_with_only_cyclic_bdles}. For example, the description of the asymptotic behavior of families of Higgs bundles going at infinity given in \cite{CollierLi} gives a good description of the behavior at  infinity of these special Hitchin components.

As a matter of fact, we can classify all Hitchin components that are parametrized by a single differential:

\begin{theorem}\label{thm:single_differential}
Let $Y$ be an orientable orbifold of negative Euler characteristic and of genus $g$, with $k$ cone points of respective orders $m_1 \leqslant \dotsc \leqslant m_k$, and $G$ a group of rank $\geqslant 2$. If $\Hit(\pi_1 Y,G)$ is parametrized by a single non-vanishing differential, i.e. there exists a unique $N \in \{ 1, \dotsc, r\}$ such that $\dim_\C H^0(Y,K(Y,d_N+1)) \neq 0$ and $\dim_\C H^0(Y,K(Y,d_i+1)) = 0$ for all $i \neq N$, then $Y$ is a sphere with $k$ cone points, with $k = 3, 4$ or $5$. This happens if and only if the pair $(Y,G)$ is one of those listed in Table \ref{table:single_differential}.
\end{theorem}
\begin{corollary}\label{comp_with_only_cyclic_bdles}
Let $Y$ be a sphere with $k$ cone points of respective orders $m_1 \leqslant \dotsc \leqslant m_k$ and let $G$ be one of the groups listed in Table \ref{table:cyclic_Higgs}. Then $\Hit(\pi_1 Y, G)$ consists only of cyclic or $(n-1)$-cyclic representations.
\end{corollary}

Theorem \ref{thm:single_differential} can also be useful to construct geometric examples of representations that lie in some special loci of $\Hit(\piX, G)$ given by the vanishing of some of the differentials. In a few cases, such loci have a clear geometric interpretation. For example, the locus $\{ (q_2, \dots, q_n) \in \Hit(\piX, \PGL(n,\R))  \mid q_{i} = 0\ \textrm{for $i$ odd }\}$ corresponds to those representations which are conjugate to representations in the symplectic group or the split orthogonal group \cite{Hitchin_Teich}, and the locus $\{ (q_2, \dots, q_n) \in \Hit(\piX, \PGL(n,\R))  \mid q_2 = 0\}$ corresponds to those representations in  $\Hit(\piX, \PGL(n,\R))$ admitting an equivariant minimal surface in the symmetric space inducing the conformal structure of $X$ \cite{Labourie_cross_ratios}. Other similarly defined loci are, instead, rather mysterious, for example no known geometric interpretation exists for the following loci: $L_3 := \{ (q_2, q_3, q_4) \in \Hit(\piX, \PGL(4,\R))  \mid q_4 = q_2 = 0\}$, $L_4 := \{ (q_2, q_4, q_6) \in \Hit(\piX, \PSp^{\pm}(6,\R))  \mid q_6 = q_2 = 0\}$, or $L_6 := \{ (q_2, q_4, q_6) \in \Hit(\piX, \PSp^{\pm}(6,\R))  \mid q_4 = q_2 = 0\}$. Orbifold groups allow us to geometrically construct examples of Hitchin representations of surface groups lying in these loci:  let $Y$ be one of the orbifolds from the 3rd line of Table \ref{table:single_differential}, and let $Y \simeq [\Si\bs X]$ be a presentation. The image of the map $\Hit(\piY, \PGL(4,\R)) \hookrightarrow \Hit(\piX, \PGL(4,\R))$ is entirely contained in $L_3$. Similarly, let $Y$ be one of the orbifolds appearing in the 11th or 12th line of Table \ref{table:single_differential}, and let $Y \simeq [\Si\bs X]$ be a presentation. The image of the natural map $\Hit(\pi_1 Y, \PSp^{\pm}(6,\R)) \hookrightarrow \Hit(\pi_1 X, \PSp^{\pm}(6,\R))$ is entirely contained in $L_4$ or $L_6$, respectively.

Even more interestingly, consider the mapping class group equivariant maps given by the following proposition: 

\begin{proposition}\label{prop:mysterious sets}
Let $H$ be of rank $2$ (i.e. one of the groups $\PGL(3,\R), \PSp^{\pm}(4,\R), \GG$), and assume that the exponents of $G$ contain the exponents of $H$. Then there exist mapping class group equivariant embeddings $ \Psi_{H,G} : \Hit(\piX, H) \hookrightarrow \Hit(\piX, G) $.
\end{proposition}

\begin{proof}
This is a consequence of Labourie's conjecture, proved for groups of rank $2$ by Labourie \cite{Labourie_cyclic_surfaces}. Denote the exponents of $H$ by $1, d_2$, and the exponents of $G$ by $d_1', \dots, d_r'$, and assume that $d_2 = d_s'$, for some $s\in \{ 1,\dots,r \}$. The equivariant map is defined in the following way: given a representation $\rho$ in $\Hit(\piX, H)$, by Labourie's conjecture there exists a unique complex structure $X(\rho)$ on the surface such that the equivariant harmonic map is a minimal immersion. The Higgs bundle corresponding  to  $\rho$  for the complex structure $X(\rho)$ is parametrized by differentials $(0,q_{d_2+1})$. With the data of $X(\rho)$ and $q_{d_2+1}$, it is possible to construct a Higgs bundle for $G$ with Riemann surface $X(\rho)$ and differentials all equal to zero, except the $s$-th one which is set to $q_{d_2+1}$. This defines the map to $\Hit(\piX, G)$.
\end{proof}

These maps are defined analytically, but no geometric definition is known, and no geometric characterization of their images. Orbifold groups allow us to geometrically construct examples of Hitchin representations of surface groups lying in the image of the maps $\Psi_{H,G}$: 

\begin{corollary} \label{corol:reps in special loci}
Consider the three cases when $H=\PGL(3,\R)$ and $G = \PGL(4,\R)$, or when $H=\PSp^{\pm}(4,\R)$ and $G = \PSp^{\pm}(6,\R)$, or when $H=\GG$ and $G = \PSp^{\pm}(6,\R)$. Then the map $\Psi_{H,G}$ is unique. Moreover, let $Y$ be one of the orbifolds appearing in the 3rd, 11th, 12th line respectively of Table \ref{table:single_differential}, and let $Y \simeq [\Si\bs X]$ be a presentation. Then the image of the natural map $\Hit(\pi_1 Y, G) \hookrightarrow \Hit(\pi_1 X, G)$ lies inside the image of $\Psi_{H,G} : \Hit(\piX, H) \hookrightarrow \Hit(\piX, G)$.
\end{corollary}

\subsection{Projective structures on Seifert-fibered 3-manifolds}   \label{subsec:projective structures}

In this subsection, we describe an application of our results to the study of the deformation spaces of geometric structures on closed $3$-manifolds. 
Recall that, given a pair $(X,G)$ where $G$ is a Lie group and $X$ is a manifold upon which $G$ acts transitively, an $(X,G)$-\emph{structure} on a manifold $M$ is an atlas of charts taking values in $X$, whose transition functions are restrictions of elements of $G$. An $(X,G)$-\emph{isomorphism} between two $(X,G)$-structures on $M$ is a self-homeomorphism of $M$ which, when expressed in charts for the two structures, is locally restriction of an element of $G$. The \emph{deformation space} of $(X,G)$-structures on $M$, denoted here by $\cD_{X}(M)$, is the set of all $(X,G)$-structures on $M$ up to $(X,G)$-isomorphisms isotopic to the identity. When $G$ is clear from the context, we shall say $X$-structure in place of $(X,G)$-structure.

In the field of geometric topology in dimension $3$, an especially important geometry is the hyperbolic geometry $(\HH^3, \PO(1,3))$, one of Thurston's eight geometries featured in the geometrization theorem. Here we will consider the $3$-dimensional geometry 
$(\PSL(2,\R), \PSL(2,\R))$, where the group acts on itself by left translations. This geometry can be considered as a subgeometry of another one of Thurston's eight geometries, the geometry $(\widetilde{\SL}(2,\R), \mathrm{Isom}(\widetilde{\SL}(2,\R)))$. The latter geometry has a bigger symmetry group, of dimension $4$. We will also consider two other geometries in dimension $3$, the \emph{projective geometry} $(\R\PP^3, \PGL(4,\R))$ and the \emph{contact projective geometry} $\R\PP_\omega^3 :=(\R\PP^3, \PSp^{\pm}(4,\R))$, the latter having this name because the group $\PSp^{\pm}(4,\R)$ acts on $\R\PP^3$ preserving the contact form induced by the standard symplectic form $\omega$ on $\R^4$. 

Hyperbolic geometry has a projective model, the Klein model: the hyperbolic isometries act on the ellipsoid as projective transformations given by the standard embedding $\PO(1,3) < \PGL(4,\R)$. This means that every hyperbolic structure on $M$ induces a projective structure and this gives an embedding of the deformation spaces $\cD_{\HH^3}(M) \subset \cD_{\R\PP^3}(M)$.
When $M$ is a closed manifold admitting a hyperbolic structure, $\cD_{\HH^3}(M)$ has only one point by Mostow rigidity, but the space $\cD_{\R\PP^3}(M) $ might be bigger. The connected component of $\cD_{\R\PP^3}(M)$ containing the hyperbolic structure consists of special projective structures called \emph{convex projective structures}. 
For some $M$, that connected component is just one point (in this case, one says that the hyperbolic structure on $M$ is \emph{projectively rigid}), while for other $M$ it is possible to deform the hyperbolic structure to other convex projective structures (see for example \cite{CLT2,CLT1,Heusener_Porti, Marquis, Choi_Lee}). 

As an application of our results, we can paint a similar picture for manifolds admitting a $\PSL(2,\R)$-structure. This geometry also  has a projective model (obtained by using the principal representation $\kappa: \PGL(2,\R) \lra \PGL(4,\R)$). The group $\kappa(\PSL(2,\R))$  acts on $\R\PP^3$ with two open orbits $\Omega^+, \Omega^- \subset \R\PP^3$, and on one of them, say $\Omega^+$, the action is simply transitive (see \cite{GW} for details). The action of $\kappa(\PSL(2,\R))$ on $\Omega^+$ can be seen as a projective model for the $\PSL(2,\R)$-geometry. Moreover, since the image of $\kappa$ is contained in $\PSp^{\pm}(4,\R)$, this model also has an invariant contact form. This gives maps from the deformation space of $\PSL(2,\R)$-structures to the deformation space of projective structures, but this map is not injective: it is $2$-to-$1$, because of the action of the disconnected group $\kappa(\PGL(2,\R))$, which still preserves $\Omega^+$. We denote the quotient of this action by $\cD_{\PSL(2,\R)}(M)_{/\sim}$ and in this way we get embeddings
$\cD_{\PSL(2,\R)}(M)_{/\sim} \subset \cD_{\R\PP^3_\omega}(M)  \subset \cD_{\R\PP^3}(M)$. We now describe these deformation spaces for all closed $3$-manifolds $M$ admitting a $\PSL(2,\R)$-structure. Recall that $\pi_1 (\PSL(2,\R)) \simeq \Z$ and let us denote by $\PSL^d(2,\R)$ the $d$-fold covering group of $\PSL(2,\R)$.

\begin{proposition}\label{Seifert-fibered_three_mflds}
Assume that $M$ is a closed $3$-manifold admitting at least one $\PSL(2,\R)$-structure. Then there exist a natural number $d$, a closed orientable $2$-orbifold $Y$ (both depending only on $M$), and a representation $\rho \in \Hit(\piY,\PSL(2,\R))$, such that $\rho$ can be lifted to a representation $\rho^d:\piY \lra \PSL^d(2,\R)$ and $M$ is homeomorphic to $\PSL^d(2,\R)/\rho^d(\piY)$.
In particular, $M$ is a Seifert-fibered space with Seifert base equal to the orbifold $Y$. Moreover, for such $Y$ and $d$, every representation $\rho \in \Hit(\piY,\PGL(2,\R))$ can be lifted to a representation $\rho^d:\piY \lra \PSL^d(2,\R)$, and $M$ is homeomorphic to the quotient space $\PSL^d(2,\R)/\rho^d(\piY)$. This quotient carries a natural $\PSL(2,\R)$-structure, which gives a homeomorphism $\Hit(\piY,\PGL(2,\R)) \ni \rho  \lmt \PSL^d(2,\R)/\rho^d(\piY) \in  \cD_{\PSL(2,\R)}(M)_{/\sim}$.
\end{proposition}
\begin{proof}
We fix a left-invariant Riemannian metric on $\PSL(2,\R)$. Then all $\PSL(2,\R)$-structures on closed manifolds are complete (see \cite[Proposition 3.4.10]{Thu97}). Fix a $\PSL(2,\R)$-structure on the closed manifold $M$, and denote by $D:\widetilde{M} \lra \PSL(2,\R)$ the developing map and by $h:\pi_1 M \lra \PSL(2,\R)$ the holonomy representation. Completeness of the structure implies that $D$ is the universal covering of $\PSL(2,\R)$. The action of $\pi_1 M$ on $\widetilde{M}$ by deck transformations is isometric, giving $M$ a $(\widetilde{\SL}(2,\R), \mathrm{Isom}(\widetilde{\SL}(2,\R)))$-structure. Now \cite[Corollary 4.7.3]{Thu97} says that $h$ has discrete image and infinite kernel. Let us denote by $\Gamma$ the quotient group $\pi_1 M/\ker(h)$, and notice that $h$ factors through a discrete and faithful representation $h':\Gamma \lra \PSL(2,\R)$. The developing map $D$ factors through a map 
$D':\widetilde{M}/\ker(h) \lra \PSL(2,\R)$. From the classification of the coverings of the circle, we see that $D'$ is a $d$-sheeted covering, hence $\widetilde{M}/\ker(h) \simeq \PSL^d(2,\R)$ and there is a unique group structure on such a covering that makes $D'$ a group homomorphism. The action of the group $\Gamma = \pi_1 M/\ker(h)$ on $\widetilde{M}/\ker(h)$ is a representation $h^d:\Gamma \lra \PSL^d(2,\R)$ that lifts the representation $h'$. The manifold $M$ is homeomorphic to the quotient $\PSL^d(2,\R)/h^d(\Gamma)$. This implies that $h'(\Gamma)$ is a cocompact discrete subgroup of $\PSL(2,\R)$, in particular, $\Gamma$ is isomorphic to $\piY$, for some orientable orbifold $Y$, and $h'$ is a representation in $\Hit(\piY,\PGL(2,\R))$. This proves the first statement. The fact that $Y$ and $d$ are unique follows from the classification of Seifert fibered $3$-manifolds. The possibility of lifting a representation from $\PSL(2,\R)$ to $\PSL^d(2,\R)$ only depends on the connected component of the representation space where the representation lies. Hence if we can lift one Hitchin representation, we can lift all of them.  
\end{proof}

\noindent To understand better the topology of $M$, one can identify  $\PSL(2,\R)$ with the unit tangent bundle $T^1 \HH^2$ of the hyperbolic plane, and, similarly, $\PSL^d(2,\R)$ can be identified with the $d$-fold covering of $T^1 \HH^2$. All these spaces are circle bundles over $\HH^2$, and the manifold $M$ is an orbifold circle bundle over $Y$. In the case $d=1$, $M$ is called the orbifold unit tangent bundle of $Y$. By Proposition \ref{Seifert-fibered_three_mflds}, the space $\cD_{\PSL(2,\R)}(M)_{/\sim}$ is connected. We denote by $\cD^{\,0}_{\R\PP^3_\omega}(M)$ and $\cD^{\,0}_{\R\PP^3}(M)$ the connected components of $\cD_{\R\PP^3_\omega}(M)$ and $\cD_{\R\PP^3}(M)$ that contain $\cD_{\PSL(2,\R)}(M)_{/\sim}$. We now describe the topology of these spaces.

\begin{lemma}    \label{lemma:component Seifert group}
Let $M$ be a closed Seifert fibered $3$-manifold, whose Seifert base is a closed $2$-orbifold $Y$ with $\chi(Y)<0$. Let $G$ be one of the groups $\PGL(n,\R)$, $\PSp^{\pm}(2m,\R)$, $\PO(m,m+1)$ or $\GG$. Denote by $\phi:\pi_1 M \lra \piY$ the projection to the fundamental group of the Seifert base. Then the map 
$$\phi^*:\Hom(\piY,G)/G \ni [\rho] \lmt [\rho \circ \phi] \in \Hom(\pi_1 M,G)/G$$
restricts to a homeomorphism from $\Hit(\piY,G)$ to a connected component of $\Hom(\pi_1 M,G)/G$.
In particular, $\Hom(\pi_1 M,G)/G$ has a connected component homeomorphic to an open ball. 
\end{lemma}

\begin{proof}
The homomorphism $\phi$ is surjective, hence the map $\phi^*$ is injective and its image is the set of all representations of $\pi_1 M$ having kernel which includes $\ker\phi$. This is a closed condition, hence the map $\phi^*$ is a closed map. Let us now restrict our attention to the case when $G=\PGL(n,\R)$. We claim that the map $\phi^*$, when restricted to the Hitchin component, is an open map. Consider a set of generators $\gamma_1, \dots \gamma_s$ of $\piY$, if $Y$ is orientable, and of $\pi_1 Y^+ < \piY$ if $Y$ is not orientable. We can lift them to elements $\bar{\gamma_1}, \dots, \bar{\gamma_s}$ of $\pi_1 M$. If $\rho \in \Hit(\piY,\PGL(n,\R))$, it is strongly irreducible by Lemma \ref{lemma:centralizer}, hence all representations in a neighborhood $U$ of $\phi^*(\rho)$ send the elements $\bar{\gamma_1}, \dots, \bar{\gamma_s}$ to elements of $\PGL(n,\R)$ generating an irreducible subgroup. The subgroup $\ker\phi$ is central in $\pi_1 M$ if $Y$ is orientable and its centralizer is $\phi^{-1}(\pi_1 Y^+)$ if $Y$ is non-orientable (see \cite[Lemma 2.4.15]{Brin_M}), in either cases it commutes with all the elements $\bar{\gamma_1}, \dots, \bar{\gamma_s}$. This implies that every representation in $U$ sends $\ker\phi$ to the identity element, because in $\PGL(n,\R)$ only the identity commutes with an irreducible subgroup. Hence all the elements of $U$ are in the image of the map $\phi^*$, and this implies our claim. When $G$ is another group in the given list, by Remark \ref{embeddings_of_HC}, all Hitchin representations in $G$ are also Hitchin representations in $\PGL(n,\R)$, hence strongly irreducible as representations in $\PGL(n,\R)$, so we can use the same argument to prove the openness of the map.
\end{proof}

\begin{theorem} \label{theorem:deformation_space_3manifold_Hitchin}
Let $M$ be a closed $3$-manifold admitting a $\PSL(2,\R)$-structure, and take $Y$ as in Proposition \ref{Seifert-fibered_three_mflds}. Then there are homeomorphisms $\Hit(\piY,\PSp^{\pm}(4,\R))  \simeq  \cD^{\,0}_{\R\PP^3_\omega}(M)$ and $\Hit(\piY,\PGL(4,\R))  \simeq   \cD^{\,0}_{\R\PP^3}(M)$. Therefore, the connected components $\cD^{\,0}_{\R\PP^3_\omega}(M)$ and $\cD^{\,0}_{\R\PP^3}(M)$ are homeomorphic to open balls, of respective dimensions $-10 \chi(|Y|) + (8k-2k_2 -2 k_3)$ and $-15 \chi(|Y|) + (12 k - 4 k_2 - 2 k_3)$.
\end{theorem}
\begin{proof}
Let $\rho\in \Hit(\piY,\PGL(4,\R))$. By Proposition \ref{prop:Anosov}, $\rho$ is Anosov, with respect to the standard Borel subgroup $B$, hence we have a $\piY$-equivariant map $\xi$ from the boundary at infinity of $\piY$ to the full flag manifold $\PGL(4,\R)/B$. Fix a presentation $Y \simeq [\Si \bs X]$. Then the restriction $\rho|_{\piX}$ is also Anosov, and shares the same equivariant map $\xi$. Guichard and Wienhard \cite{GW} used the map $\xi$ to construct a domain of discontinuity which has two connected components $\Omega_\rho^+, \Omega_\rho^-$. Consider also a Fuchsian representation $\rho_0 \in \Hit(\piY,\PGL(4,\R))$. In this case, the domain $\Omega_{\rho_0}^+$ coincides with the domain $\Omega^+$ that we used above to describe the projective model of the $\PSL(2,\R)$-geometry. The action of $\rho_0$ on $\Omega_{\rho_0}^+$ is topologically conjugate to the action of $\rho$ on $\Omega_\rho^+$ \cite{GW}. Let $d$ be as in the Proposition \ref{Seifert-fibered_three_mflds}, and consider the $d$-fold covering $\Omega_\rho^{+,d} \lra \Omega_\rho^+$. By Proposition \ref{Seifert-fibered_three_mflds}, the action of $\rho_0$ on $\Omega_{\rho_0}^+$ lifts to an action of $\piY$ on $\Omega_{\rho_0}^{+,d}$. Since the actions are topologically conjugate, the action of $\rho$ on  $\Omega_\rho^+$ also lifts to an action of $\piY$ on $\Omega_\rho^{+,d}$. This defines two continuous maps, $\Psi_\omega: \Hit(\piY,\PSp^{\pm}(4,\R)) \ni \rho  \lra \Omega_\rho^{+,d}/\piY \in  \cD^{\,0}_{\R\PP^3_\omega}(M)$ and $\Psi: \Hit(\piY,\PGL(4,\R))  \ni \rho  \lra \Omega_\rho^{+,d}/\piY \in  \cD^{\,0}_{\R\PP^3}(M)$, 
which land in the connected components $\cD^{\,0}_{\R\PP^3_\omega}(M)$ and $\cD^{\,0}_{\R\PP^3}(M)$, respectively, as the source spaces are connected. We will now prove that they are homeomorphisms. By Lemma \ref{lemma:component Seifert group}, 
$\phi^*(\Hit(\piY,\PGL(4,\R)))$ is a connected component of $\Hom(\pi_1 M, \PGL(4,\R))/\PGL(4,\R)$. Consider then the holonomy map
$\mathrm{Hol}:\cD^{\,0}_{\R\PP^3}(M) \lra \phi^*(\Hit(\piY,\PGL(4,\R)))$, sending a real projective structure to the conjugacy class of its holonomy representation. The map $\mathrm{Hol}$ is a local homeomorphism, and the map $\Psi\circ (\phi^*)^{-1}$ is a section of $\mathrm{Hol}$. Using the fact that $\cD^{\,0}_{\R\PP^3}(M)$ is connected, we see that $\Psi$ is a homeomorphism. For the map $\Psi_\omega$, we use the same arguments with the holonomy map for geometry $\R\PP^3_\omega$.
\end{proof}

\noindent When $M$ is the unit tangent bundle of a closed orientable surface, Lemma \ref{lemma:component Seifert group} and Theorem \ref{theorem:deformation_space_3manifold_Hitchin} were proved in \cite{GW}. Here, we have generalized their result to all closed $3$-manifolds $M$ admitting a $\PSL(2,\R)$-structure and we have set up Diagram \eqref{rigid_proj_structures_diagram}, which illustrates the various statements in Corollary \ref{rigid_proj_structures_intro}.

As a corollary of Theorem \ref{theorem:deformation_space_3manifold_Hitchin}, we can find explicit examples of Seifert fibered 3-manifolds with rigid or deformable projective structures and contact projective structures, see the discussion around Corollary \ref{rigid_proj_structures_intro} for details.

\begin{equation}\label{rigid_proj_structures_diagram}
\begin{tikzcd}        
\Hit(\piY,\PGL(2,\R)) \arrow[hook]{r} \arrow{d}{\simeq}& \Hit(\piY,\PSp^\pm(4,\R)) \arrow[hook]{r} \arrow{d}{\simeq} & \Hit(\piY,\PGL(4,\R)) \arrow{d}{\simeq} \\
\cD_{\PSL(2,\R)}(M)_{/\sim} \arrow[hook]{r} & \cD^{\,0}_{\RP^3_{\omega}}(M) \arrow[hook]{r} \arrow[hook]{d} & \cD^{\,0}_{\RP^3}(M)\arrow[hook]{d} \\
& \cD_{\RP^3_{\omega}}(M) \arrow[hook]{r} & \cD_{\RP^3}(M) 
\end{tikzcd}
\end{equation}


\appendix 

\section{Expected dimensions of Hitchin components}\label{section:appendix}

In this appendix, we compare the dimension of the Hitchin component $\Hit(\piY,\PGL(n,\R))$, which was determined in Theorem \ref{theorem:dimension}, with the dimension that it is possible to guess by examining a presentation of the group $\piY$. We will call the latter dimension the \emph{expected} dimension of the Hitchin component, and we will show here that the two dimensions agree. For triangle groups, J.P.~Burelle has studied the expected dimension of non-Hitchin components in \cite{Burelle}.

For simplicity, we will restrict our attention to orientable orbifolds $Y$ and to the target group $\PGL(n,\R)$. So let $Y$ be a closed orientable $2$-orbifold with $k$ cone points of orders $m_1, \dotsc, m_k$, with underlying space $|Y|$ a surface of genus $g$. Then $\pi_1 Y$ has a presentation of the standard form
\begin{equation}\label{standard_presentation}
\quad \langle a_1, b_1, \dotsc, a_g, b_g, x_1, \dotsc, x_k \mid
 [a_1,b_1] \dotsm [a_g,b_g] \, x_1 \dotsm x_k = 1 = x_1^{m_1} = \dotsm = x_k^{m_k} \rangle.
\end{equation}

We can define the \emph{expected dimension} $\dim_{e} \Hit(\piY,\PGL(n,\R))$ of the Hitchin component considering that for each $i = 1, \dotsc, g$, the generators $a_i$ and $b_i$ can be mapped to elements of $\PGL(n,\R)$ that form an open subset (see Proposition \ref{prop:Hitchin rep loxodromic}), so we count a dimension $\dim\PGL(n,\R)$ for each one of them. For each $j = 1, \dotsc, k$, instead, we have seen in the proof of Proposition \ref{Hitchin_rep_faithful} that the generator $x_j$ can be mapped to an element of $\PGL(n,\R)$ which is conjugate to $\kappa(\tau_m)$, where $\kappa$ is the principal representation, as in (\ref{re_ppal_sl2}), and $\tau_{m}$ is the matrix
$$ 
\begin{bmatrix}
 \cos \tfrac{\pi}{m} & -\sin \tfrac{\pi}{m} \\
 \sin \tfrac{\pi}{m} & \phantom{-}\cos \tfrac{\pi}{m} 
 \end{bmatrix}
 \in \PGL(2,\R).
$$
We denote by $\mathcal{D}_n\left(\sfrac{\Z}{m\Z}\right)$ the component of the representation variety $\Hom\left(\sfrac{\Z}{m\Z},\PGL(n,\R)\right)$ containing the representation $\rho$ given by $\rho(1) = \kappa(\tau_m)$. For every $x_j$, we count a dimension $\dim\mathcal{D}_n\left(\sfrac{\Z}{m_j\Z}\right)$. We also need to consider the relation $[a_1,b_1] \dotsm [a_g,b_g] \, x_1 \dotsm x_k = 1$, so we subtract a term equal to $\dim\PGL(n,\R)$. Finally, since the Hitchin component is the space of conjugacy classes of Hitchin representations, we subtract again $\dim\PGL(n,\R)$. We obtain in this way the \emph{expected dimension} of the Hitchin components: 
$$\dim_{e} \Hit(\piY,\PGL(n,\R)) :=  (2 g - 2) \dim \PGL(n,\R) +  \sum_{i=1}^{k} \dim \mathcal{D}_n (\sfrac{\Z}{m_i\Z}).$$
The arguments used to define the expected dimension are just a way to guess the dimension, but they do not constitute a proof that the actual dimension is the same. We will now prove that this expected dimension agrees with the actual dimension. We note that Theorem \ref{theorem:dimension} determines not only the dimension, but also the topology of the Hitchin component.

\begin{remark}
We can define similarly the expected dimension of other components of the representation space $\Hom(\piY,\PGL(n,\R))/\PGL(n,\R)$. However, it may not be true that the expected dimension is in fact equal to the actual dimension in those cases. In fact, the expected dimension can sometimes be negative (\cite{Burelle}). That is why, in this paper, we do not mention the expected dimensions of components of $\Rep(\piY,\PGL(n,\R))$ other than the Hitchin component.
\end{remark}

\begin{proposition}\label{Expected_dim} One has $\dim \Hit(\piY,\PGL(n,\R)) = \dim_{e} \Hit(\piY,\PGL(n,\R))$.
\end{proposition}
\begin{proof} Since we know from Theorem \ref{theorem:dimension} that the Hitchin component is of dimension $(2 g - 2) \dim \PGL(n,\R) + 2 \sum_{d=2}^{n}\sum_{i=1}^k O(d,m_i)$, the only thing we need to show is that $\dim \mathcal{D}_n \left(\sfrac{\Z}{m\Z}\right) = 2 \sum_{d=2}^{n} O(d,m)$. In \cite{Long_Thistlethwaite}, Long and Thistlethwaite introduced an arithmetic function of two variables $\sigma(n,m)$ for $n, m \geqslant 2$, and showed that $\mathcal{D}_n \left(\sfrac{\Z}{m\Z}\right)$ is of dimension $n^2 - \sigma(n,m)$. Here, if $q$ and $r$ are the quotient and remainder on dividing $n$ by $m$, respectively, i.e. $n = m q + r$ with $0 \leqslant r < m$, then $\sigma(n,m) = (n+r)q + r$. Then the lemma is a consequence of the following simple computation:
\begin{align*}
2 \sum_{d=2}^{n} O(d,m) & \,=\,  2 \sum_{d=1}^{n}   \left( d + \left\lfloor  - \frac{d}{m} \right\rfloor  \right) \,=\,  2 \sum_{d=1}^{n} d  - 2 \sum_{d=1}^{mq+r} \left\lceil \frac{d}{m} \right\rceil \\ 
& \,=\,  2 \cdot \frac{n(n+1)}{2} - 2 \left( m \cdot \frac{q(q+1)}{2} + r(q+1) \right) \,=\,  n^2 - \sigma(n,m). \qedhere
\end{align*}
\end{proof}

\newpage

\section{Tables}\label{tables_appendix}

\begin{table}[h]
\small
\begin{center}
{\renewcommand{\arraystretch}{1.2}
\renewcommand{\tabcolsep}{0.5cm}
\begin{tabular}{llll}
\toprule
$\Int(\fg_\C)^\tau$  & Dimension  & Rank &  Exponents            \\
\midrule
$\PGL(n,\R)\qquad \ \,(n\geqslant 2)$       & $n^2-1$  & $n-1$            & $1,2, \dots, n-1$       \\ 
$\PSp^{\pm}(2m,\R)\quad\ (m\geqslant 1)$   & $m(2m+1)$ & $m$            & $1,3, \dots, 2m-1$      \\
$\PO(m,m+1)\quad (m\geqslant 1)$        & $m(2m+1)$      & $m$       & $1,3, \dots, 2m-1$      \\
$\PO^{\pm}(m,m)\qquad (m\geqslant 3)$    & $m(2m-1)$ &   $m$         & $1,3, \dots, 2m-3; m-1$ \\
$\GG$                 & $14$          & $2$         & $1,5$                   \\
$\mathbf{F}_4$        & $52$           & $4$        & $1,5,7,11$              \\
$\mathbf{E}_6$        & $78$           & $6$        & $1,4,5,7,8,11$          \\
$\mathbf{E}_7$        & $133$         & $7$         & $1,5,7,9,11,13,17$      \\
$\mathbf{E}_8$        & $248$         & $8$         & $1,7,11,13,17,19,23,29$   \\
\bottomrule
\end{tabular}}
\end{center}
\caption{The dimension and exponents of simple complex Lie algebras.}\label{table:exponents}
\end{table}

\begin{table}[h]
\small
\begin{center}
{\renewcommand{\arraystretch}{1.2}
\renewcommand{\tabcolsep}{0.5cm}
\begin{tabular}{ll}
\toprule
Group  & Dimension of $\Hit(\piY,G)$ for $Y$ closed           \\
\midrule
$\PGL(2,\R) 
$ &  $-3 \chi(|Y|) + 2 k + \ell$ \quad \cite{Thurston_notes}\\
$\PGL(3,\R)$ & $-8\chi(|Y|) +(6 k - 2 k_2) + (3 \ell - \ell_2)$ \quad \cite{CG} \\
$\PGL(4,\R) 
$ & $ -15 \chi(|Y|) + (12 k - 4 k_2 - 2 k_3) + (6 \ell - 2 \ell_2- \ell_3)$ \\
$\PGL(5,\R)$ & $-24 \chi(|Y|) + (20 k - 8 k_2 - 4 k_3 - 2 k_4) + (10\ell - 4 \ell_2- 2 \ell_3 - \ell_4)$ \\
$\PGL(6,\R)$ & $-35 \chi(|Y|) + (30 k - 12 k_2 - 6 k_3 - 4 k_4 - 2 k_5) + (15 \ell - 6 \ell_2- 3 \ell_3 - 2 \ell_4 - \ell_5)$ \\
$\PGL(7,\R)$ & $-48 \chi(|Y|) + (42 k - 18 k_2 - 10 k_3 - 6 k_4 - 4 k_5 - 2 k_6) + (21 \ell - 9 \ell_2- 5 \ell_3 - 3 \ell_4 - 2 \ell_5 - \ell_6)$ \\
$\PSp^\pm(4,\R)
$ & $-10 \chi(|Y|) + (8k-2k_2 -2 k_3) + (4\ell - \ell_2 -\ell_3)$\\
$\PSp^\pm(6,\R)$ & $-21 \chi(|Y|) + (18k-6k_2 -4 k_3 - 2 k_4 - 2 k_5) + (9\ell - 3\ell_2 - 2\ell_3 - \ell_4 - \ell_5)$ \\
$\PSp^\pm(8,\R)$ & $-36 \chi(|Y|) + (32 k - 12 k_2 - 8 k_3 - 4 k_4 - 4 k_5 - 2 k_6 - 2 k_7)$ \\
             & $\phantom{-36 \chi(|Y|)} + (16 \ell - 6 \ell_2- 4 \ell_3 - 2 \ell_4 - 2 \ell_5 - \ell_6 - \ell_7)$ \\
$\PSp^\pm(10,\R)$ & $-55 \chi(|Y|) + (50 k - 20 k_2 - 14 k_3 - 8 k_4 - 6 k_5 - 4 k_6 - 4 k_7 - 2 k_8 - 2 k_9)$ \\
             & $\phantom{-55 \chi(|Y|)} + (25 \ell - 10 \ell_2- 7 \ell_3 - 4 \ell_4 - 3 \ell_5 - 2 \ell_6 - 2 \ell_7 - \ell_8 - \ell_9)$ \\
$\PSp^\pm(12,\R)$ & $-78 \chi(|Y|) + (72 k - 30 k_2 - 20 k_3 - 12 k_4 - 10 k_5 - 6 k_6 - 6 k_7 - 4 k_8 - 4 k_9 - 2 k_{10} - 2 k_{11})$ \\
             & $\phantom{-78 \chi(|Y|)} + (36 \ell - 15 \ell_2- 10 \ell_3 - 6 \ell_4 - 5 \ell_5 - 3 \ell_6 - 3 \ell_7 - 2 \ell_8 - 2 \ell_9 - \ell_{10} - \ell_{11})$ \\
$\PO^\pm(4,4)$ & $-28 \chi(|Y|) + (24 k - 8 k_2 - 6 k_3 - 2 k_4 - 2 k_5) + (12 \ell - 4 \ell_2- 3 \ell_3 - \ell_4 - \ell_5)$ \\
$\PO^\pm(5,5)$ & $-45 \chi(|Y|) + (40 k - 16 k_2 - 10 k_3 - 6 k_4 - 4 k_5 - 2 k_6 - 2 k_7)$ \\
             & $\phantom{-45 \chi(|Y|)} + (20 \ell - 8 \ell_2- 5 \ell_3 - 3 \ell_4 - 2 \ell_5 -  \ell_6 - \ell_7)$ \\
$\PO^\pm(6,6)$ & $-66 \chi(|Y|) + (60 k - 24 k_2 - 16 k_3 - 10 k_4 - 8 k_5 - 4 k_6 - 4 k_7 - 2 k_8 - 2 k_9)$ \\
             & $\phantom{-66 \chi(|Y|)} + (30 \ell - 12 \ell_2- 8 \ell_3 - 5 \ell_4 - 4 \ell_5 - 2 \ell_6 - 2 \ell_7 - \ell_8 - \ell_9)$ \\
$\GG$ & $-14 \chi(|Y|) + (12 k - 4 k_2 - 2 k_3 - 2 k_4 - 2 k_5) + (6 \ell - 2 \ell_2- \ell_3 - \ell_4 - \ell_5)$ \\
$\mathbf{F}_4$  &  $-52 \chi(|Y|) + (48 k - 20 k_2 - 12 k_3 - 8 k_4 - 8 k_5 - 4 k_6 - 4 k_7 - 2 k_8 - 2 k_9 - 2 k_{10} - 2 k_{11})$ \\
             & $\phantom{-52 \chi(|Y|)} + (24 \ell - 10 \ell_2- 6 \ell_3 - 4 \ell_4 - 4 \ell_5 - 2 \ell_6 - 2 \ell_7 - \ell_8 - \ell_9 - \ell_{10} - \ell_{11})$ \\
$\mathbf{E}_6$  & $-78 \chi(|Y|) + (72 k - 32 k_2 - 18 k_3 - 14 k_4 - 10 k_5 - 6 k_6 - 6 k_7 - 4 k_8 - 2 k_9 - 2 k_{10} - 2 k_{11})$ \\
             & $\phantom{-52 \chi(|Y|)} + (36 \ell - 16 \ell_2- 9 \ell_3 - 7 \ell_4 - 5 \ell_5 - 3 \ell_6 - 3 \ell_7 - 2 \ell_8 - \ell_9 - \ell_{10} - \ell_{11})$ \\
\bottomrule
\end{tabular}}
\end{center}
\caption{Dimension of Hitchin components for groups of rank $\leqslant 6$.}\label{explicit_formulas_in_low_rank}
\end{table}

\newpage

\begin{table}[h]
\small
\begin{center}
{\renewcommand{\arraystretch}{1.2}
\renewcommand{\tabcolsep}{0.5cm}
\begin{tabular}{lll}
\toprule
Group  & Degree & $(m_1, \ldots, m_k)$ with $m_i \leq m_{i+1}$, $\sum \tfrac{1}{m_i} < k - 2$ \\
\midrule
$\PGL(3,\R)$   & $2$ & $(2,2,2,2,2)$ or $(2,2,2,m_4)$ \\
               & $3$ & $(m_1,m_2,m_3)$ with $m_1 \geqslant 3$ \\
$\PGL(4,\R)$   & $3$ & $(3,3,m_3)$ \\
               & $4$ & $(2,m_2,m_3)$ with $m_2 \geqslant 4$ \\
$\PGL(5,\R)$   & $4$ & $(2,4,m_3)$ \\
$\PGL(n,\R)$ with $n = 6, 7$  & $6$ & $(2,3,m_3)$ \\
$\PSp^\pm(4,\R)$ or $\PO(2,3)$  & $4$ & $(m_1,m_2,m_3)$ with $m_2 \neq 3$ \\
$\PSp^\pm(6,\R)$ or $\PO(3,4)$  & $6$ & $(2,3,m_3)$ or $(3,3,m_3)$ \\
$\GG$  & $6$ & $(m_1,m_2,m_3)$ with $ (m_1, m_3) \neq (2, 5)$ \\
\bottomrule
\end{tabular}}
\end{center}
\caption{Hitchin components containing only cyclic or $(n-1)$-cyclic Higgs bundles.}\label{table:cyclic_Higgs}
\end{table}


\begin{table}[h]
\small
\begin{center}
{\renewcommand{\arraystretch}{1.2}
\renewcommand{\tabcolsep}{0.5cm}
\begin{tabular}{lll}
\toprule
$k$ & $d$ & $(m_1, \,\dotsc\,, m_k)$ with $m_i \leq m_{i+1}$, $\sum \tfrac{1}{m_i} < k - 2$ \\ 
\midrule
\multirow{18}{*}{$3$} & $2$ & $(m_1,m_2,m_3)$\\
& $3$ & $(2,m_2,m_3)$  \\
& $4$ & $(2,3,m_3)$, $(3,3,m_3)$ \\
& $5$ & $(2,3,m_3)$, $(2,4,m_3)$, $(3,3,4)$, $(3,4,4)$, $(4,4,4)$ \\
& $6$ & $(2,4,5)$, $(2,5,5)$ \\
& $7$ & $(2,3,m_3)$, $(2,m_2,m_3)$ with $4 \leqslant m_2 \leqslant m_3 \leqslant 6$, $(3,3,m_3)$ with $4 \leqslant m_3 \leqslant 6$ \\
& $8$ & $(2,3,7)$ \\
& $9$ & $(2,3,7)$, $(2,3,8)$, $(2,4,m_3)$ with $5 \leqslant m_3 \leqslant 8$ \\
& $10$ & $(2,3,m_3)$ with $7 \leqslant m_3 \leqslant 9$, $(3,3,4)$ \\
& $11$ & $(2,3,m_3)$ with $7 \leqslant m_3 \leqslant 10$, $(2,4,5)$, $(2,5,5)$ \\
& $13$ & $(2,3,m_3)$ with $7 \leqslant m_3 \leqslant 12$, $(2,4,5)$, $(2,4,6)$, $(3,3,4)$ \\
& $15, 16$ & $(2,3,7)$ \\
& $17$ & $(2,3,7)$, $(2,3,8)$, $(2,4,5)$ \\
& $19$ & $(2,3,m_3)$ with $7 \leqslant m_3 \leqslant 9$ \\
& $21$ & $(2,4,5)$ \\
& $22$, $23$ & $(2,3,7)$ \\
& $25$ &  $(2,3,7)$, $(2,3,8)$\\
& $29, 31, 37, 43$ & $(2,3,7)$ \\
\midrule
\multirow{3}{*}{$4$} & $3$ & $(2,2,2,m_4)$ \\
& $5$ & $(2,2,2,3)$, $(2,2,2,4)$ \\
& $7$ & $(2,2,2,3)$ \\
\midrule
$5$ & $3$ & $(2,2,2,2,2)$ \\
\bottomrule
\end{tabular}}
\end{center}
\caption{Spheres with $k$ cone points satisfying $H^0(Y,K(Y,d))=0$.}\label{table:vanishing_diff}
\end{table}


\begin{table}[h]
\small
\begin{center}
{\renewcommand{\arraystretch}{1.2}
\renewcommand{\tabcolsep}{0.5cm}
\begin{tabular}{lll}
\toprule
Group & $d_N + 1$\!\!\!\!\!\!\!\!\! & $(m_1, \ldots, m_k)$ with $m_i \leq m_{i+1}$, $\sum \tfrac{1}{m_i} < k - 2$            \\
\midrule
$\PGL(3,\R)$  & $2$ & $(2,2,2,2,2)$ or $(2,2,2,m_4)$ \\
              & $3$ & $(m_1,m_2,m_3)$ with $m_1 \geqslant 3$ \\
$\PGL(4,\R)$  & $3$ & $(3,3,m_3)$ \\
              & $4$ & $(2,m_2,m_3)$ with $m_2 \geqslant 4$ \\
$\PGL(5,\R)$  & $3$ & $(3,3,4)$ \\
              & $4$ & $(2,4,m_3)$ \\
$\PGL(n,\R)$ with $n = 6, 7$ & $4$ & $(2,4,5)$ \\
              & $6$ & $(2,3,m_3)$ \\
$\PGL(n,\R)$ with $n = 8, 9, 10, 11$ & $6$ & $(2,3,7)$ \\
$\PSp^\pm(4,\R)$ or $\PO(2,3)$  & $4$ & $(m_1,m_2,m_3)$ with $m_2 \neq 3$ \\
$\PSp^\pm(6,\R)$ or $\PO(3,4)$  & $4$ & $(2,4,5)$ or $(2,5,5)$ \\
                                & $6$ & $(2,3,m_3)$ or $(3,3,m_3)$ \\
$\PSp^\pm(2m,\R)$ or $\PO(m,m+1)$  with $m = 4, 5$ & $6$ & $(2,3,7)$\\
$\PO^\pm(4,4)$  & $6$ & $(2,3,m_3)$ or $(3,3,m_3)$ \\
$\PO^\pm(5,5)$  & $6$ & $(2,3,7)$ \\
$\GG$  & $6$ & $(m_1,m_2,m_3)$ with $ (m_1, m_3) \neq (2, 5)$ \\
\bottomrule
\end{tabular}}
\end{center}
\caption{Hitchin components parametrized by a single non-vanishing differential.}\label{table:single_differential}
\end{table}

\end{document}